\documentclass[10pt]{article}

\usepackage[utf8]{inputenc}
\usepackage[T1]{fontenc}
\usepackage{lmodern}

\usepackage{amsfonts}
\usepackage{amsmath}
\usepackage{amsthm}
\usepackage{amssymb}
\usepackage{stmaryrd}
\usepackage{enumerate}
\usepackage{enumitem}
\usepackage{adjustbox}
\usepackage{caption}

\usepackage{tikz}
\usetikzlibrary{cd}
\usetikzlibrary{positioning}
\usetikzlibrary{arrows.meta}
\usepackage[hidelinks]{hyperref}

\theoremstyle{plain}
\newtheorem{theorem}{Theorem}[section]
\newtheorem{proposition}[theorem]{Proposition}
\newtheorem{lemma}[theorem]{Lemma}
\newtheorem{fact}[theorem]{Fact}
\newtheorem{corollary}[theorem]{Corollary}
\newtheorem{example}[theorem]{Example}
\newtheorem*{claim*}{Claim}

\theoremstyle{definition}
\newtheorem{definition}[theorem]{Definition}
\newtheorem*{notation*}{Notation}
\newtheorem{question}[theorem]{Question}

\theoremstyle{remark}
\newtheorem{remark}[theorem]{Remark}

\setlist[enumerate]{itemsep=0mm}

\newcommand{\bA}{\mathbb{A}}
\newcommand{\bB}{\mathbb{B}}
\newcommand{\NN}{\mathbb{N}}
\newcommand{\ZZ}{\mathbb{Z}}
\newcommand{\QQ}{\mathbb{Q}}
\newcommand{\RR}{\mathbb{R}}
\newcommand{\CC}{\mathbb{C}}

\newcommand{\cA}{\mathcal{A}}
\newcommand{\cB}{\mathcal{B}}
\newcommand{\cC}{\mathcal{C}}
\newcommand{\cE}{\mathcal{E}}
\newcommand{\cG}{\mathcal{G}}
\newcommand{\cCO}{\mathcal{CO}}
\newcommand{\bRO}{\mathbb{RO}}
\newcommand{\cNWD}{\mathcal{NWD}}
\newcommand{\cSBP}{\mathcal{SBP}}
\newcommand{\cM}{\mathcal{M}}
\newcommand{\cN}{\mathcal{N}}
\newcommand{\cP}{\mathcal{P}}
\newcommand{\fC}{\mathfrak{C}}
\newcommand{\fU}{\mathfrak{U}}
\newcommand{\id}{\mathrm{id}}
\newcommand{\ext}{\mathrm{ext}}
\newcommand{\bd}{\operatorname{bd}}
\newcommand{\cl}{\operatorname{cl}}
\newcommand{\st}{\operatorname{st}}
\newcommand{\dom}{\operatorname{dom}}
\newcommand{\Def}{\operatorname{Def}}
\newcommand{\End}{\operatorname{End}}
\newcommand{\Ker}{\operatorname{Ker}}
\newcommand{\Per}{\operatorname{Per}}
\newcommand{\Sym}{\operatorname{Sym}}
\newcommand{\lef}{\le_{\text{fin}}}
\newcommand{\gne}{\ge \adjustbox{lap=-6mm, raise=-0.3mm, scale=0.5}{\textbf{\textbackslash}}}
\newcommand{\symdif}{\mathbin{\Delta}}
\newcommand{\sqsubsetneq}{\mathbin{\sqsubseteq \adjustbox{lap=-5.5mm, raise=-1mm, scale=0.5}{\textbf{/}} \hspace{-1.5mm}}}
\newcommand{\sqsupsetneq}{\mathbin{\sqsupseteq \adjustbox{lap=-6mm, raise=-1mm, scale=0.5}{\textbf{/}} \hspace{-1.5mm}}}

\newcommand{\zero}{0}
\newcommand{\one}{1}

\newcommand{\<}{\langle}
\renewcommand{\>}{\rangle}

\renewcommand{\Im}{\operatorname{Im}}
\renewcommand{\le}{\leqslant}
\renewcommand{\ge}{\geqslant}
\renewcommand{\bar}{\overline}
\renewcommand{\int}{\operatorname{int}}

\newcommand{\noproof}{\pushQED{\qed} \qedhere \popQED}

\begin{document}

\begin{titlepage}
\begin{center}
\vspace{10mm}
{\huge \textsc{University of Wrocław}}

\vspace{25mm}

{\LARGE \textsc{Doctoral Thesis}}

\vspace{25mm} 

\hspace*{-3mm} \rule{12.57cm}{2pt} \\[3ex]
{\Huge \textbf{Ellis groups in model theory and strongly generic sets}} \\[1ex]
\rule{11cm}{2pt}

\vspace{25mm}

{\huge Adam Malinowski}

\vspace{25mm}

{\Large Advisor: prof. Ludomir Newelski}

\vspace{40mm}

{\Large Wrocław 2023}
\end{center}
\end{titlepage}

\begin{titlepage}

\section*{\centering Abstract}
Assume $G$ is a group and $\cA$ is an algebra of subsets of $G$ closed under left translation. We study various ways to understand the Ellis group of the $G$-flow $S(\cA)$ (the Stone space of $\cA$), with particular interest in the model-theoretic setting where $G$ is definable in a first order structure $M$ and $\cA$ consists of externally definable subsets of $G$.

In one part of the thesis we explore strongly generic sets. Maximal algebras of such sets are shown to carry enough information to retrieve the Ellis group. A subset of $G$ is strongly generic if each non-empty Boolean combination of its translates is generic. Trivial examples include what we call \emph{periodic} sets, which are unions of cosets of finite index subgroups of $G$. We give several characterizations of strongly generic sets, in particular, we relate them to almost periodic points of the flow $2^G$. For groups without a smallest finite index subgroup we show how to construct non-periodic strongly generic subsets in a systematic way. When $G$ is definable in a model $M$, a definable, strongly generic subset of $G$ will remain as such in any elementary extension of $M$ only if it is strongly generic in $G$ in an adequately uniform way. Sets satisfying this condition are called \emph{uniformly strongly generic}. We analyse a few examples of these sets in different groups.

In the second part we assume that $G$ is a topological group and consider a particular algebra of its subsets denoted $\cSBP$. It consists of subsets of $G$ that have the \emph{strong Baire property}, meaning nowhere dense boundary. We explicitly describe the Ellis group of $S(\cA)$ for an arbitrary subalgebra $\cA$ of $\cSBP$ under varying assumptions on the group $G$, including the case when $G$ is a compact topological group. We use this description to relate the Ellis groups computed for a model and its elementary extension in particular scenarios. Under some of those assumptions we also decide whether the obvious inclusions between the families of strongly generic, uniformly strongly generic and periodic sets can be reversed. These results can be applied in o-minimal structures, in which externally definable subsets are proved to have the strong Baire property. Finally, we propose a procedure of finding a maximal generic algebra in a given subalgebra of $\cSBP$ given that we succeeded in doing so while neglecting nowhere dense sets.

\vspace{6mm}

\subsection*{\centering Acknowledgements}
I would like to express deep gratitude to my advisor, Ludomir Newelski. His expert guidance, limitless engagement and incredibly dry wit have earned my truest appreciation.
\end{titlepage}

\thispagestyle{empty}
\tableofcontents
\newpage

\pagenumbering{arabic}

\section{Introduction}

The idea of applying topological dynamics in model theory is due to Newelski \cite{New09, New12Aug}, who suggested it could serve to extend the results from stable group theory to unstable context. It was later explored and broadened by several authors, see e.g. \cite{Pil13, GPP14, Jag15, CS18, KNS19}. One area of study is to investigate dynamical objects, such as the Ellis group,\footnote{The phrase \emph{Ellis group} is now a common name in model theory for an object defined in Section 2.} associated with any definable group. From a model-theoretic perspective an important question arises: to what extent, if any, are these objects preserved when computed for different models of the same theory? This dissertation aims to examine some constructions and techniques related to the issue, among which are the strongly generic sets.

Assume $G$ is a group definable in a first-order structure $M$. The space $S_G(M)$, the Stone space of the algebra of $M$-definable subsets of $G$, is naturally a $G$-flow. In case $M$ is stable, $S_G(M)$ is canonically isomorphic to the enveloping (Ellis) semigroup of itself and there is a unique minimal ideal $\cM$ of $S_G(M)$, consisting of the generic types of $G$. $\cM$ is also a group isomorphic to $G/G^0$, where $G^0$ is the connected component of $G$. However, when $M$ is not stable, $S_G(M)$ is typically not isomorphic to its Ellis semigroup and we need to consider $S_{\ext, G}(M)$ instead, which is the Stone space of the algebra of \emph{externally} definable subsets of $G$. As there is usually little relation between subsets externally definable in a model $M$ and its elementary extension $M^*$, it remains unclear whether any properties of the Ellis groups defined with respect to the corresponding $G$-flows, $S_{\ext, G}(M)$ and $S_{\ext, G}(M^*)$, are preserved.

An idea to overcome this is to express the Ellis group in terms of objects that are closer to definability, so that their properties will carry over between models. An algebra of subsets of $G$ closed under left translation is called a \emph{$G$-algebra}. It was observed in \cite{New12Feb} that to any Ellis group we can assign an \emph{image algebra}, which is a particular $G$-algebra of externally definable subsets of $G$. Its unique property is that it consists of generic sets, i.e. subsets of $G$ for which we can find finitely many translates covering $G$. It is easy to find examples of generic sets, but it is quite unusual for them to form non-trivial $G$-algebras. A set generating a $G$-algebra consisting only of generic sets (except the empty set) is thereby called \emph{strongly generic}. 

Newelski proved that any maximal $G$-algebra consisting of externally definable generic sets is an image algebra and any minimal ideal is determined by image algebras up to homeomorphism. With a little bit more work, also the algebraic structure of the Ellis group can be retrieved from strongly generic sets. We present the details in Section 2. On the other hand, strongly generic sets feature a certain kind of regularity which ideally could reveal itself uniformly across many models of the same theory. For instance, Newelski proved in \cite{New12Feb} that a strongly generic definable subset of a stable group must be \emph{periodic}, i.e. a union of cosets of a finite index subgroup. These properties lead us to believe that strongly generic sets are worth a study on their own. Ultimately this may result in a discovery of some interesting connections between Ellis groups of a model and its elementary extension.

We also consider a strenghtening of the notion of strongly generic sets, which are the uniformly strongly generic sets. These are precisely the sets that remain strongly generic in every elementary extension, taking it one step closer to the realization of the idea from the previous paragraph. However, not every strongly generic set is uniformly strongly generic. In fact, we prove that certain kinds of groups, including the compact groups, are guaranteed to have a strongly generic subset that is not uniformly strongly generic. 

Our second idea is localization. Given a large Boolean algebra $\cA$, any ultrafilter $p \in S(\cA)$ is determined by the family of its restrictions $p \cap \cA_i$ to subalgebras $\cA_i$ of $\cA$ that jointly generate it. For example, it is common in stability theory to view a type $p \in S(M)$ as a consistent collection of $\Delta$-types, where $\Delta$ is a finite set of formulas. This representation is sufficient to capture vital properties of types, such as forking or genericity, by means of their local ranks. It is easy to show, as we do in Section 4, that the Ellis group of $S_{\ext, G}(M)$ can be similarly decomposed into simpler, local parts. Thereby a potentially fruitful approach to understanding the Ellis group is to identify and study its tractable fragments with respect to this decomposition. Assuming that $G$ is a topological group, we explore one such fragment, corresponding to the algebra $\cSBP$ of all sets having what we name the \emph{strong Baire property} (abbreviated SBP), in Section 4. 

The main contribution of the thesis consists in:
\begin{itemize}[label=--]
\item Establishing a correspondence between strongly generic sets in an arbitrary group $G$ and almost periodic points in the flow $2^G$; see Theorem~\ref{thm:sgen-aper}.
\item Providing a wide range of examples of sets that are strongly generic but not periodic; see Proposition~\ref{prop:tf-sg} together with Proposition~\ref{prop:tree-nper}.
\item Providing  and analysing a few examples of uniformly strongly generic sets that are not periodic; see Examples~\ref{ex:usg-nper1}, \ref{ex:usg-nper2}, \ref{ex:usg-nper3}.
\item Discovering the algebra $\cSBP$, which allows of a full description of the Ellis group in multiple scenarios; see Theorems~\ref{thm:pf-ell}, \ref{thm:comp-ell}, \ref{thm:pre-ell}. Furthermore, identifying the boundary of the techniques used to give that description; see Subsection~\ref{ss:infty}.
\item Proving that every infinite compact Hausdorff group has a strongly generic subset with SBP that is not periodic; see Theorem~\ref{thm:comp-nusg}.
\item Showing that externally definable subsets of densely ordered o-minimal structures have SBP, making the previous results applicable to groups definable in such structures; see Proposition~\ref{prop:ext-sbp} and Corollary~\ref{cor:ext-sbp}.
\item Revealing a correspondence between maximal generic subalgebras of $\cSBP$ and maximal generic subalgebras of the algebra of regular open subsets of $G$; see Theorem~\ref{thm:max-reg}.
\end{itemize}

The dissertation is organized as follows. Section 2 serves as a reminder of known theory that is the starting point of our study. We review classical topics such as topological dynamics and Stone duality, then proceed to recount the work from \cite{New12Feb, New12Aug} that motivates our interest in strongly generic sets. The next two sections contain our original results. Section 3 explores the basic properties and characterizations of strongly generic sets and describes an organized way to construct them. It also introduces uniformly strongly generic sets and gives a few interesting examples. In Section 4 we compute the Ellis group of an arbitrary subalgebra of $\cSBP$ under varying assumptions on the group $G$. Next we explain why our approach fails for other groups. Finally, we propose a technique of obtaining maximal generic algebras in $\cSBP$ from like algebras consisting of regular open subsets of $G$.

\section{Preliminaries}

\subsection{Topological dynamics}

In this subsection we recall the relevant notions and results from topological dynamics. For a comprehensive study on the subject the reader is directed to \cite{Aus88}. Throughout the subsection $G$ is a fixed (discrete) group.

A $G$\emph{-flow} is a compact (Hausdorff) topological space $X$ on which $G$ acts by homeomorphisms.\footnote{The classical definition of a $G$-flow is more general, assuming that $G$ is a topological group and the action need be jointly continuous. However, in our usage of the notion $G$ will always be treated as discrete.} A \emph{subflow} is a subset $Y \subseteq X$ closed both topologically and under the action of $G$. For any point $p \in X$ there is a smallest subflow of $X$ containing $p$, namely $\cl(G \cdot p)$. A minimal subflow is a non-empty subflow without a non-empty proper subflow. By compactness and the Zorn's lemma, any non-empty subflow of $X$ contains a minimal subflow. A point $p \in X$ is called \emph{almost periodic} if it belongs to any minmal subflow, or equivalently, if $\cl(G \cdot p)$ is a minimal subflow. A $G$\emph{-flow morphism} is a continuous function between $G$-flows preserving the action of $G$. It is a $G$\emph{-flow isomorphism} if it has an inverse which is a morphism.

Below we state some basic facts about $G$-flows to be used later.

\begin{remark} \label{rem:gen} Let $f_1, f_2 : X \to Y$ be $G$-flow morphisms. Then the set
\[
X_0 = \{ x \in X : f_1(x) = f_2(x) \}
\]
is a subflow of $X$.
\end{remark}

\begin{fact} \label{fact:per} Assume $f : X \to Y$ is a $G$-flow morphism.
\begin{enumerate}[label=(\roman{*})]
\item If $X_0 \subseteq X$ is a subflow, then $f[X_0] \subseteq Y$ is a subflow.
\item If $Y_0 \subseteq Y$ is a subflow, then $f^{-1}[Y_0] \subseteq X$ is a subflow.
\item If $X_0 \subseteq X$ is a minimal subflow, then $f[X_0] \subseteq Y$ is a minimal subflow.
\item If $p \in X$ is almost periodic, then $f(p) \in Y$ is almost periodic.
\item If $q \in Y$ is almost periodic and $f$ is onto (or just $q \in f[X]$), then $q = f(p)$ for some almost periodic $p \in X$.
\end{enumerate}
\end{fact}

\begin{proof} (i) -- (iv) are easy. To prove (v), let $Y_0 = \cl(G \cdot q)$ and $X_0 = f^{-1}[Y_0]$. Then $X_0 \subseteq X$ is a non-empty subflow, so it contains a minimal subflow $X_1$. We have that $f[X_1] \subseteq Y_0$ is a subflow, so in fact $f[X_1] = Y_0$ by minimality of $Y_0$. In particular, there is $p \in X_1$ such that $f(p) = q$.
\end{proof}

\begin{remark} \label{rem:min-gen} A non-empty subflow $X_0 \subseteq X$ is minimal if and only if for each open $U \subseteq X$ satisfying $U \cap X_0 \neq \varnothing$ there exist $g_1, \ldots, g_n \in G$ such that $X_0 \subseteq g_1U \cup \ldots \cup g_n U$.
\end{remark}

\begin{proof} $({\implies})$ The set $X_0 \setminus G \cdot U$ is a proper subflow of $X_0$, so it is empty, i.e. $X_0 \subseteq G \cdot U$. The conclusion follows from compactness.

\noindent
$({\impliedby})$ Assume for contradiction that there is a non-empty proper subflow $X_1 \subseteq X_0$ and take an open set $U \subseteq X$ such that $U \cap X_0 \neq \varnothing$ and $U \cap X_1 = \varnothing$. Then $gU \cap X_1 = \varnothing$ for each $g \in G$, which contradicts the assumption.
\end{proof}

\begin{definition} \label{def:gen} Assume $X$ is a $G$-flow.
\begin{enumerate}[label=(\roman{*})]
\item A subset $U \subseteq X$ is \emph{generic} if $X = \bigcup_{i=1}^n g_i U$ for some $g_1, \ldots, g_n \in G$.
\item A point $p \in X$ is \emph{generic} if every open neighbourhood of $p$ is generic.
\end{enumerate}
\end{definition}

\begin{fact}[{\cite[Lemma 1.7, Corollary 1.9]{New09}}] \label{fact:gen-min} A generic point $p \in X$ exists if and only if there is a unique minimal subflow $X_0 \subseteq X$. In this case $X_0$ consists precisely of the generic points in $X$.
\end{fact}

\begin{proof} Since distinct minimal subflows are clearly disjoint, it suffices to show that a point $p \in X$ is generic if and only if $p$ belongs to every minimal subflow of $X$. For one direction, assume that $p \in X$ is generic and $p \notin X_0$ for some minimal subflow $X_0 \subseteq X$. The set $U = X \setminus X_0$ is an open neighbourhood of $p$ closed under the action of $G$, so $p$ is not generic, which is a contradiction. 

For the other direction, assume a non-generic point $p$ belongs to every minimal subflow and take an open neighbourhood $U$ of $p$ that is not generic. By compactness, $X_1 := X \setminus G \cdot U$ is a non-empty subflow of $X$, so it contains a minimal subflow $X_0 \subseteq X_1$. Clearly $p \notin X_0$, a contradiction.
\end{proof}

Assume $X$ is a $G$-flow. For any $g \in G$, we denote by $\pi_g : X \to X$ the homemorphism given by the action of $G$ on $X$. Regarding $\{ \pi_g : g \in G \}$ as a subset of $X^X$ endowed with the product topology, the \emph{enveloping semigroup} or \emph{Ellis semigroup} $E(X)$ is the closure of this subset. $E(X)$ equipped with function composition is a semigroup with identity $\pi_e$. It is moreover a \emph{left topological semigroup}, meaning that the semigroup operation is continuous in the first coordinate. $E(X)$ is also a $G$-flow itself with the action $g \cdot f = \pi_g \circ f$, where $g \in G, f \in E(X)$. 

More generally, assume $S$ is a left topological semigroup with identity. A subset $I \subseteq S$ is a (left) ideal, written $I \trianglelefteqslant S$, provided that $S \cdot I \subseteq I$. It is a minimal ideal, denoted $I \trianglelefteqslant_m S$, if it is minimal among the non-empty ideals. For every element $a \in S$ there is a smallest ideal containing $a$, namely $S \cdot a$. Moreover if $I \trianglelefteqslant_m S$, then $I \cdot a = I$ for each $a \in I$. When $S$ is compact, it is routine to show that any ideal contains a minimal ideal. An element $u \in S$ is called an idempotent if it satisfies $u^2 = u$.

\begin{remark} \label{rem:sub} \leavevmode
\begin{enumerate}[label=(\roman{*})]
\item For $f \in E(X)$, the smallest subflow of $E(X)$ containing $f$ equals the smallest ideal containing $f$.
\item In $E(X)$, the minimal subflows and minimal ideals coincide.
\end{enumerate}
\end{remark}
\begin{proof} (i) The function $r_f : E(X) \to E(X)$ defined by $r_f(h) = h \circ f$ is continuous, so by compactness of $E(X)$,
\[
\cl( G \cdot f ) = \cl r_f[\{ \pi_g : g \in G \}] = r_f[\cl \{ \pi_g : g \in G \}] = r_f[E(X)] = E(X) \circ f.
\]

\vspace{2mm}
\noindent
(ii) For $Y \subseteq E(X)$, we have that
\begin{align*}
Y \text{ is a minimal subflow} & \iff (\forall y \in Y) \, Y = \cl( G \cdot y ), \\
Y \text{ is a minimal ideal}& \iff (\forall y \in Y) \, Y = E(X) \circ y.
\end{align*}
The conclusion follows from (i).
\end{proof}

Below we formulate a fundamental theorem of Ellis:

\begin{theorem}[\cite{Ell69}] \label{thm:ell} Assume $S$ is a compact left topological semigroup.
\begin{enumerate}[label=(\roman{*})]
\item Given $I \trianglelefteqslant_m S$, the set $J(I)$ of idempotents in $I$ is non-empty. 
\item For every $I \trianglelefteqslant_m S$ and $u \in J(I)$, the set $uI$ is a group and $I$ is a disjoint union of the groups $uI$, where $u \in J(I)$.
\item The groups $uI$, where $I \trianglelefteqslant_m S$ and $u \in J(I)$, are all isomorphic.
\item For each $I, J \trianglelefteqslant_m S$ and $u \in J(I)$ there is $v \in J(J)$ such that $uv = v$ and $vu = u$.
\end{enumerate}
\end{theorem}
In particular, the theorem applies to $E(X)$. In the context of model theory the groups $uI$, into which any minimal ideal splits, are called \emph{ideal subgroups}, and their common isomorphism type is the \emph{Ellis group} of the $G$-flow $X$.

The following facts will be used for explicit Ellis group computation:

\begin{lemma} \label{lem:ell-epi} Assume $\cG$ is a group, $I \trianglelefteqslant_m E(X)$ and $\varphi : I \to \cG$ is a semigroup epimorphism such that $\varphi^{-1}[ \{ e \} ] = J(I)$. Then the Ellis group of $X$ is isomorphic to $\cG$.
\end{lemma}

\begin{proof} Take any $u \in J(I)$. We will check that $\varphi \restriction uI$ is an isomorphism of groups $uI \to \cG$. It is clearly injective because $\ker( \varphi \restriction uI ) = uI \cap J(I) = \{ u \}$. For the proof of surjectivity, fix $g \in \cG$ and pick $q \in I$ such that $\varphi(q) = g$. Then $uq \in uI$ and $\varphi(uq) = \varphi(u) \varphi(q) = g$.
\end{proof}

\begin{corollary} \label{cor:ell-epi} Assume $\cG$ is a group, $I \trianglelefteqslant_m E(X)$ and $\varphi : E(X) \to \cG$ is a semigroup epimorphism such that $I \cap \varphi^{-1}[ \{ e \} ] = J(I)$. Then the Ellis group of $X$ is isomorphic to $\cG$.
\end{corollary}

\begin{proof} By Lemma~\ref{lem:ell-epi} it suffices to show that $\varphi[I] = \cG$. Take any $u \in J(I)$, fix $g \in G$ and pick $f \in E(X)$ satisfying $\varphi(f) = g$. Then $fu \in I$ and 
\[
\varphi(fu) = \varphi(f) \varphi(u) = \varphi(f) = g. \qedhere
\]
\end{proof}

In the last lemma and corollary the respective assumptions $\varphi^{-1}[ \{ e \} ] = J(I)$ and $I \cap \varphi^{-1}[ \{ e \} ] = J(I)$ each can be replaced with a left-to-right inclusion ($\subseteq$), since the reverse inclusion holds automatically.

\subsection{Boolean algebras and Stone spaces}

Here we recall the correspondence (formally: dual equivalence of categories) between Boolean algebras and Stone spaces, which we will generally refer to as \emph{Stone duality}. A \emph{Stone space} is a topological space which is compact Hausdorff and totally disconnected, meaning it has a basis of sets that are both closed and open.\footnote{We call such sets \emph{clopen}.} To any Boolean algebra $\cA$ we assign the space of ultrafilters on $\cA$, called the \emph{Stone space of} $\cA$ and denoted as $S(\cA)$. For $A \in \cA$, define 
\[
[A] = \{ p \in S(\cA) : A \in p \}.
\] 
The set $S(\cA)$ equipped with the topology generated by $\{ [A] : A \in \cA \}$ is a Stone space. On the other hand, when $X$ is a topological space, we assign to it the family $\cCO(X)$ of clopen subsets of $X$, which is a Boolean algebra.

When $S$ is a set and $\cA$ is an algebra of subsets of $S$, for $s \in S$, let
\[
\hat{s} = \{ A \in \cA : s \in A \}.
\]
It is an ultrafilter on $\cA$ which we call the \emph{principal ultrafilter corresponding to} $s$. If additionally $S = X$ is a Stone space and $\cA = \cCO(X)$, the map $x \mapsto \hat{x}$ is a homemorphism between $X$ and $S(\cCO(X))$. On the other hand, if $\cA$ is a Boolean algebra, the map $A \mapsto [A]$ is an isomorphism between $\cA$ and $\cCO(S(\cA))$.

Let $\cA, \cB$ be Boolean algebras. To a homomorphism $\varphi : \cA \to \cB$ we assign the continuous function $\varphi^* : S(\cB) \to S(\cA)$, called the \emph{dual map of} $\varphi$, given by the condition $A \in \varphi^*(q) \iff \varphi(A) \in q$. Furthermore, to a continuous function $f : X \to Y$ between topological spaces $X, Y$ we assign the homomorphism $f^* : \cCO(Y) \to \cCO(X)$, also called the dual map of $f$, defined as $f^*(B) = f^{-1}[B]$. These assignments make a pair of contravariant functors between the category of Boolean algebras and the category of Stone spaces. So we have $(\varphi \circ \psi)^* = \psi^* \circ \varphi^*$ etc. 

When $\cCO(S(\cA))$ is identified with $\cA$ via the isomorphism defined above, the map $(\varphi^*)^* : \cCO(S(\cA)) \to \cCO(S(\cB))$ is identical to $\varphi : \cA \to \cB$. If $X$ and $Y$ are Stone spaces, the same holds the other way: for $f : X \to Y$, $(f^*)^*$ is identical to $f$ after the identification of $S(\cCO(X))$ with $X$. Hence the pair of functors establishes dual equivalence of the categories of Boolean algebras and Stone spaces.

\begin{fact} \label{fact:dual-func} Assume $\varphi : \cA \to \cB$ is a homomorphism of Boolean algebras and consider the dual map $\varphi^* : S(\cB) \to S(\cA)$.
\begin{enumerate}[label=(\roman{*})]
\item If $\varphi$ is injective, then $\varphi^*$ is surjective.
\item If $\varphi$ is surjective, then $\varphi^*$ is injective. 
\end{enumerate}
\end{fact}
The same fact holds for any continuous function $f : X \to Y$ between Stone spaces and its dual. It follows that both implications are actually equivalences. 

\begin{fact} \label{fact:im} Assume $\varphi : \cA \to \cB$ is a homomorphism of Boolean algebras. Then the image of $\varphi^* : S(\cB) \to S(\cA)$ consists precisely of all $p \in S(\cA)$ such that $p \cap \Ker \varphi = \varnothing$. \noproof
\end{fact}

Assume $G$ is a group. A Boolean algebra on which $G$ acts by automorphisms will be called a $G$\emph{-algebra}. A $G$\emph{-algebra homomorphism} is a Boolean algebra homomorphism preserving the group action. When $\cA$ is a $G$-algebra, there is a natural structure of a $G$-flow on $S(\cA)$, namely $g \cdot q = \{ gA : A \in q \}$ for $g \in G, q \in S(\cA)$. Conversely, when $X$ is a topological space on which $G$ acts by homeomorphisms (such as a $G$-flow), the action $g \cdot A = \{ g \cdot p : p \in A \}$, where $g \in G, A \in \cCO(X)$, defines the structure of a $G$-algebra on $\cCO(X)$. All facts related to Stone duality naturally extend to $G$-flows and $G$-algebras.

Assume $\cA$ is a $G$-algebra. When a subset $\cB \subseteq \cA$ is closed under all $G$-algebra operations, we call it a \emph{$G$-subalgebra}, denoted as $\cB \le \cA$, and regard it as a $G$-algebra with the induced structure. In particular, given $\cB \subseteq \cP(G)$, we write $\cB \le \cP(G)$ when $\cB$ is closed under union, complement and left translation, and regard it as a $G$-algebra equipped with these operations. On the other hand, sometimes $\cB$ already has a structure of a $G$-algebra and $\cB \subseteq \cA$, in which case we write $\cB \le \cA$ to indicate that the structures coincide, or $\cB \not \le \cA$ when they are different.\footnote{This happens in Subsection~\ref{sub:reg}: $\bRO \subseteq \cSBP$ are $G$-algebras, but $\bRO \not \le \cSBP$.} 

Sometimes we simply write \emph{algebra} instead of $G$-algebra. Whenever we speak of an algebra without the action of $G$, we use the full term \emph{Boolean algebra}.

The $G$-algebras have a natural variant of the first isomorphism theorem:
\begin{fact} Assume $\varphi : \cA \to \cB$ is a $G$-algebra homomorphism.
\begin{enumerate}[label=(\roman{*})]
\item $\Ker \varphi = \{ A \in \cA : \varphi(A) = \zero \}$ is a $G$-ideal in $\cA$, meaning an ideal closed under the action of $G$.
\item $\Im \varphi = \{ \varphi(A) : A \in \cA \}$ is a $G$-subalgebra of $\cB$.
\item If $I \trianglelefteqslant \cA$ is any $G$-ideal, there is a natural $G$-algebra structure on the quotient $\cA/I$ and a natural quotient epimorphism $\pi : \cA \to \cA / I$.
\item There is a unique isomorphism $\overline{\varphi} : \cA / \Ker \varphi \to \Im \varphi$ such that $\varphi$ decomposes as $\cA \xrightarrow{\ \pi \ } \cA / \Ker \varphi \xrightarrow{\ \overline{\varphi} \ } \Im \varphi$.
\end{enumerate}
\end{fact}

In the following remark and corollary, $\cA \le \cP(G)$ is arbitrary.

\begin{remark} \label{rem:den} The orbit of $\hat{e}$ is dense in $S(\cA)$, i.e. $S(\cA) = \cl( G \cdot \hat{e} )$.
\end{remark}
\begin{proof} We have that $G \cdot \hat{e} = \{ \hat{g} : g \in G \}$, which is clearly dense in $S(\cA)$.
\end{proof}

\begin{corollary} \label{cor:mor} $G$-flow morphisms $f_1, f_2 : S(\cA) \to Y$ are equal if and only if $f_1(\hat{e}) = f_2(\hat{e})$.
\end{corollary} 

\begin{proof} Follows directly from Remarks~\ref{rem:gen} and \ref{rem:den}.
\end{proof}

\subsection{Alternative representation of the Ellis semigroup}

In the next two subsections we present the relevant portions of \cite{New12Feb, New12Aug} together with some thus far unpublished folklore. Several proofs have been reworked with the use of Stone duality, which we think offers a valuable perspective on the subject.

Assume that $G$ is a group and $\cA \le \cP(G)$. Thus $X = S(\cA)$ is a Stone space and a $G$-flow.

\begin{definition}[\cite{New12Aug}, before Lemma 1.2]
For $A \in \cA$ and $q \in S(\cA)$, let
\begin{equation} \label{eq:dop}
d_q(A) := \{ h \in G : A \in hq \} = \{ h \in G : h^{-1} A \in q \}.
\end{equation}
\end{definition}
It is easy to check that for a fixed $q \in S(\cA)$, the function $d_q : \cA \to \cP(G)$ is a $G$-algebra homomorphism. It need not always be the case that $d_q(A) \in \cA$. If, however, $d_q(A) \in \cA$ for all $q \in S(\cA)$ and $A \in \cA$, we say that $\cA$ is $d$\emph{-closed}.

\begin{remark} \label{rem:dop} The definition $(\ref{eq:dop})$ also makes sense (and defines a $G$-algebra homomorphism) whenever $q \in S(\cB)$, where $\cB$ is a $G$-algebra containing $\cA$, particularly for $q \in \beta G = S(\cP(G))$. Fix such $\cB$ and let $r = r_{\cB, \cA} : S(\cB) \to S(\cA)$ be the dual map of the inclusion $\cA \to \cB$, which we call the \emph{restriction}. Then for any $A \in \cA, q \in S(\cB)$ we have that
\[
d_q(A) = d_{r(q)}(A),
\]
which implies that the function $d_q : \cA \to \cA$ depends only on $r(q) = q \cap \cA$.
\end{remark}

More generally:
\begin{fact} \label{fact:dual} Assume $\varphi : \cA \to \cB$ is a homomorphism of $G$-algebras, $A \in \cA$ and $p \in S(\cB)$. Then $d_p( \varphi(A) ) = d_{\varphi^*(p)}(A)$. \noproof
\end{fact}

\begin{definition} Assume $\cA$ is $d$-closed and $p, q \in S(\cA)$. The multiplication on $S(\cA)$ is defined by
\begin{equation} \label{eq:ast}
p \ast q := \{ A \in \cA : d_q(A) \in p \}.
\end{equation}
\end{definition}
Again, it can be checked directly (see \cite{New12Aug}) that $p \ast q \in S(\cA)$ and $*$ is associative and continuous in the first coordinate, so $(S(\cA), \ast)$ is a left topological semigroup. Below we prove it using the Stone duality.

The operations $d_q$ and $\ast$ can be understood as follows. For $q \in S(\cA)$, denote as $r_q : G \to S(\cA)$ the function $r_q(g) = g \cdot q$. When $G$ is given the discrete topology and the action of itself by left translation, $r_q$ becomes a morphism between topological spaces on which $G$ acts by homeomorphisms. Thus from Stone duality we get the $G$-algebra homomorphism $r_q^* : \cCO(S(\cA)) \to \cP(G)$, defined by $r_q^*(C) = r_q^{-1}[C]$. After the standard identification $r_q^*$ can be thought of as a map $\cA \to \cP(G)$, whereafter it is identical to $d_q$.  It follows that $d_q$ is a $G$-algebra homomorphism.

In case $\cA$ is $d$-closed, we have that $d_q : \cA \to \cA$ and we denote the dual $G$-flow morphism $d_q^* : S(\cA) \to S(\cA)$ as $\hat{r}_q$. The maps $r_q$ and $\hat{r}_q$ agree on $G$ in the sense that $\hat{r}_q(\hat{g}) = r_q(g)$ for $g \in G$, thus from now we write $r_q$ for both.\footnote{We may think of $G$ as a subset of $S(\cA)$ via the map $G \to S(\cA), g \mapsto \hat{g}$, even though it need not be injective.} We also have that $p \ast q = r_q(p)$, hence $\ast$ is continuous in the first coordinate. Moreover, $r_q(r_p( \hat{e} )) = r_q(p) = p \ast q = r_{p \ast q}(\hat{e})$, where both $r_q \circ r_p$ and $r_{p \ast q}$ are $G$-flow morphisms. From Corollary~\ref{cor:mor} we get $r_{p \ast q} = r_q \circ r_p$, which means precisely that $\ast$ is associative.

\begin{lemma} \label{lem:ast-prod}
Assume that $\cA$ is d-closed, $p, q \in S(\cA)$ and $A \in p, B \in q$. Moreover, assume that the set
\[
AB := \{ ab : a \in A, b \in B \}
\]
belongs to $\cA$. Then $AB \in p \ast q$.
\end{lemma}

\begin{proof}
We have that $A \subseteq d_q( AB )$ because for any $a \in A$ we have $B \subseteq a^{-1} AB$, so $a^{-1} AB \in q$. It follows that $d_q(AB) \in p$, i.e. $AB \in p \ast q$.
\end{proof}

There is another description of the $d_q$ operation:
\begin{remark} \label{rem:dlim} Consider $Y = \cP(G)$ as a topological space with topology induced from the product topology on $2^G$ via the natural bijection. It is a Stone space and a $G$-flow with the action $\varphi_g(A) := Ag^{-1} = \{ a g^{-1} : a \in A \}$, which we denote here as $\varphi_g$ to avoid confusion with ordinary left translation.
\begin{enumerate}[label=(\roman{*})]
\item For $g \in G, A \in \cA$ we have that $d_{\hat{g}}(A) = Ag^{-1}$.
\item For a fixed $A \in \cA$, the map $f : S(\cA) \to \cP(G)$ defined by $f(q) = d_q(A)$ is a $G$-flow morphism.
\item For $A \in \cA$ and $q \in S(\cA)$,
\[
d_q(A) = \lim_{\hat{g} \to q} Ag^{-1}.
\]
\item For $A \in \cA$, we have that $\{ d_q A : q \in S(\cA) \} = \cl_{\cP(G)} \{ Ag^{-1} : g \in G \}$.
\end{enumerate}
\end{remark}

\begin{proof} \mbox{} \\[1ex]
(i) By definition
\[
h \in d_{\hat{g}}(A) \iff h^{-1} A \in \hat{g} \iff g \in h^{-1} A \iff h \in Ag^{-1}.
\]

\vspace{2mm} \noindent
(ii) The family $\{ V_b : b \in G \} \cup \{ \cP(G) \setminus V_b : b \in G \}$ is a subbasis of $\cP(G)$, where $V_b = \{ B \in \cP(G) : b \in B \}$. The map $f$ is continuous as for $b \in G$:
\[
f(q) \in V_b \iff b \in f(q) \iff b \in d_q(A) \iff b^{-1} A \in q \iff q \in [b^{-1} A],
\]
so the preimage $f^{-1}[V_b] = [b^{-1} A]$ is a clopen set. Moreover, for $g \in G$,
\begin{align*}
h \in f(gq) & \iff h \in d_{gq}(A) \iff A \in hgq \\
& \iff hg \in d_q(A) \iff h \in f(q) g^{-1},
\end{align*}
so $f(gq) = \varphi_g \big( f(q) \big)$.

\vspace{2mm} \noindent
(iii) Follows from (i) and (ii).

\vspace{1mm}
\noindent
(iv) By Remark~\ref{rem:den} and the compactness of $S(\cA)$:
\begin{align*}
\{ d_q A : q \in S(\cA) \} & = f[S(\cA)] = f[ \cl \{ \hat{g} : g \in G \} ] \\
& = \cl \{ f(\hat{g}) : g \in G \} = \cl \{ Ag^{-1} : g \in G \}. \qedhere
\end{align*}
\end{proof}

\begin{corollary} \label{cor:dcl-rt} If $\cA$ is d-closed, then it is closed under right translation.
\end{corollary}

Since by Remark~\ref{rem:dop} the $G$-algebra $\cA$ is $d$-closed iff $d_q(A) \in \cA$ for each $A \in \cA, q \in \beta G$, i.e. it is closed under operations $\{ d_q : q \in \beta G \}$ which do not depend on $\cA$, any intersection of $d$-closed $G$-subalgebras of $\cP(G)$ is again $d$-closed. Thus the $G$-algebra $\cA$ has a $d$\emph{-closure} $\cA^d$, the smallest $d$-closed $G$-algebra containing it, which is the intersection of all such $G$-algebras. 

Using this notion we define a useful generalization of (\ref{eq:ast}): the same formula defines a function $\ast : S(\cB) \times S(\cA) \to S(\cA)$ whenever $\cA$ is a $G$-algebra and $\cB$ is a $G$-algebra containing $\cA^d$. Again, for $p \in S(\cB), q \in S(\cA)$, we have that $p \ast q = r_q(p)$, where $r_q : S(\cB) \to S(\cA)$ is the dual map of $d_q : \cA \to \cB$, hence $\ast$ is continuous in the first coordinate. By the same argument as before it is also associative in the sense that if $\cC$ is a $G$-algebra containing $\cB^d$, then $(p \ast q) \ast s = p \ast (q \ast s)$ for $p \in S(\cC), q \in S(\cB), s \in S(\cA)$.

The following fact is an explicit description of $d$-closure.

\begin{fact} \label{fact:dcl-gen} $\cA^d$ is the Boolean algebra generated by the set
\[
\cC = \bigcup_{q \in S(\cA)} d_q[ \cA ] = \bigcup_{q \in \beta G} d_q[ \cA ].
\]
\end{fact}

\begin{proof} Let $\< \cC \>$ be the Boolean algebra generated by $\cC$. Clearly $\< \cC \> \subseteq \cA^d$ and $\cA \subseteq \< \cC \>$ because $d_{\hat{e}} = \id$ by Remark~\ref{rem:dlim}, so it suffices to show that $\< \cC \> \le \cP(G)$ and it is d-closed. The set $\cC$ is closed under left translation since $d_q[\cA] \le \cP(G)$ for any $q \in \beta G$. Thus $\< \cC \> \le \cP(G)$. To show that it is $d$-closed, it suffices to check that $\cC$ is closed under $d_p$ for $p \in \beta G$, as then $d_p[ \< \cC \> ] = \< d_p[ \cC ] \> \subseteq \< \cC \>$. Fix $p \in \beta G$ and $C \in \cC$, so that $C = d_q(A)$ for some $q \in S(\cA), A \in \cA$. By Fact~\ref{fact:dual}, $d_p(C) = d_p( d_q(A) ) = d_{d_q^*(p)}(A) = d_{p \ast q}(A) \in \cC$, where $d_q : \cA \to \cP(G)$.
\end{proof}

\begin{remark} \label{rem:ast-hom} Assume $\cA \le \cB \le \cP(G)$ are d-closed and consider $S(\cA)$, $S(\cB)$ as semigroups with $\ast$. Then the restriction $\pi : S(\cB) \to S(\cA)$ is an epimorphism.
\end{remark}
\begin{proof} 
Take $q \in S(\cB)$. The diagram on the left commutes, hence by Stone duality, the diagram on the right also commutes:
\begin{displaymath}
\begin{minipage}{0.3 \textwidth}
\begin{center}
\begin{tikzcd}
\cB \arrow{r}{d_q} & \cB \\
\cA \arrow{u}{i} \arrow{r}{d_{\pi(q)}} & \cA \arrow{u}{i}
\end{tikzcd}
\end{center}
\end{minipage}
\begin{minipage}{0.17 \textwidth}
\begin{center}
\begin{tikzpicture} 
\clip (-1, -1) rectangle (1, 1);
\draw (-0.4, -0.09) -- (0.1, -0.09) -- (0.1, -0.2) -- (0.4, 0) -- (0.1, 0.2) -- (0.1, 0.09) -- (-0.4, 0.09) -- cycle;
\end{tikzpicture}
\end{center}
\end{minipage}
\begin{minipage}{0.3 \textwidth}
\begin{center}
\begin{tikzcd}
S(\cB) \arrow{d}{\pi} & S(\cB)  \arrow{l}{r_q} \arrow{d}{\pi} \\
S(\cA) & S(\cA) \arrow{l}{r_{\pi(q)}}
\end{tikzcd}
\end{center}
\end{minipage}
\end{displaymath}
It follows that $\pi(r_q(p)) = r_{\pi(q)}(\pi(p))$ for $p \in S(\cB)$, so $\pi(p \ast q) = \pi(p) \ast \pi(q)$. 
\end{proof}

The next result, due to Newelski, has not been published before.

\begin{theorem} \label{thm:iso} The following are isomorphic as semigroups:
\begin{enumerate}[label=(\roman{*})]
\item $E(S(\cA))$ with the structure described before;
\item $S(\cA^d)$ with $\ast$;
\item $\End_G(\cA^d)$, the set of all $G$-algebra endomorphisms of $\cA^d$ with function composition;
\item $\End_G(S(\cA^d))^{\mathrm{op}}$, the opposite semigroup to the set of all $G$-flow endomorphisms of $S(\cA^d)$ with function composition.
\end{enumerate}
Moreover, the first two are isomorphic as $G$-flows via the same isomorphism.
\end{theorem}

\begin{proof}
(i) $\cong$ (ii) We will check that the function $\ell : S(\cA^d) \to E(S(\cA))$ defined by $\ell(p) = \ell_p$, where $\ell_p(q) = p \ast q$ for $q \in S(\cA)$, is an isomorphism of semigroups and $G$-flows. For $g \in G$, we have that $\ell_{\hat{g}}(q) = \hat{g} \ast q = g \cdot q$, so $\ell(\hat{g}) = \pi_g$. Since $\{ \hat{g} : g \in G \}$ is dense in $S(\cA^d)$, it follows that 
\[
\ell[S(\cA^d)] = \ell \big[ \cl \{ \hat{g} : g \in G \} \big] = \cl \{ \ell( \hat{g} ) : g \in G \} = \cl \{ \pi_g : g \in G \} = E(S(\cA)),
\]
so $\ell$ is both well defined and surjective. It is also injective since 
\[
\ell(p)(\hat{e}) = p \ast \hat{e} = p.
\]
Since $\ast$ is associative, $\ell$ is a semigroup homomorphism. Furthermore, 
\[
\ell(g \cdot q) = \ell( \hat{g} \ast q ) = \ell( \hat{g} ) \circ \ell( q ) = \pi_g \circ \ell(q),
\]
so $\ell(g \cdot q) = g \cdot \ell(q)$. It remains to prove that $\ell$ is a homemorphism.

A subbasis of $E(S(\cA))$ consists of sets of the form 
\[
V_{q, A} = \{ f \in E(S(\cA)) : f(q) \in [A] \},
\]
where $q \in S(\cA)$ and $A \in \cA$. We have that 
\begin{align*}
\ell(p) \in V_{q, A} \iff \ell_p(q) \in [A] \iff A \in p \ast q \iff d_q(A) \in p.
\end{align*}
Hence $\ell^{-1}[V_{q, A}] = [d_q(A)]$ is clopen in $S(\cA^d)$, which implies the continuity of $\ell$. By Fact~\ref{fact:dcl-gen}, the sets $\{ d_q(A) : q \in S(\cA), A \in \cA \}$ generate $\cA^d$ as a Boolean algebra. It follows that the family $\{ [d_q(A)] : q \in S(\cA), A \in \cA \}$ is a subbasis of $S(\cA^d)$, meaning that $\ell^{-1}$ is also continuous.

\vspace{2mm} \noindent
(ii) $\cong$ (iv) The function $r : S(\cA^d) \to \End_G(S(\cA^d))$, $r(q) = r_q = (d_q)^*$, is clearly a well defined semigroup antihomomorphism. It is injective because if $r(p) = r(q)$, then $p = r_p( \hat{e} ) = r_q( \hat{e} ) = q$. It remains to show the surjectivity of $r$. Let $f : S(\cA^d) \to S(\cA^d)$ be a $G$-flow endomorphism and take $q = f(\hat{e})$. Then $r_q : S(\cA^d) \to S(\cA^d)$ is also a $G$-flow morphism such that $r_q(\hat{e}) = q$, so by Corollary~\ref{cor:mor}, we have $f = r_q$.

\vspace{2mm} \noindent
(iii) $\cong$ (iv) Follows directly from Stone duality.
\end{proof}

\subsection{Image algebras}

In this subsection we continue to assume that $G$ is a group and $\cA 
\le \cP(G)$.

Theorem~\ref{thm:iso} allows us to better understand the structure of ideal subgroups in $E(S(\cA))$. For any $p \in S(\cA^d)$, we consider
\begin{align*}
\Ker d_p = \{ B \in \cA^d : d_p(B) = \varnothing \} && \text{and} && \Im d_p = \{ d_p(B) : B \in \cA^d \}.
\end{align*}

\begin{lemma} \label{lem:ker} If $p \in S(\cA^d)$, then 
\[
\cl( G \cdot p ) = \bigcap \{ [B]^c : B \in \Ker d_p \}.
\]
\end{lemma}
\begin{proof}
For $B \in \cA^d$, we have that
\[
B \in \Ker d_p \iff (\forall g \in G) \, gp \notin [B] \iff [B] \cap \cl( G \cdot p ) = \varnothing.
\]
Given that for any closed subset $F \subseteq S(\cA^d)$ we have
\[
F = \bigcap \{ [B]^c : [B] \cap F = \varnothing \},
\]
the conclusion follows.
\end{proof}

\begin{lemma} \label{lem:im} For any set $X$, assume $S \subseteq X^X$ is a semigroup with function composition and let $I \trianglelefteqslant_m S$, $u \in J(I)$. Then
\begin{enumerate}[label=(\roman{*})]
\item $u(y) = y$ for all $y \in \Im u$;
\item $qu = q$ for any $q \in I$.
\end{enumerate}
\end{lemma}
\begin{proof}
(i) Follows directly from $u \circ u = u$.

(ii) Since $q \in I = I \circ u$, there is $\alpha \in I$ such that $q = \alpha \circ u$. Therefore 
\[
qu = \alpha u u = \alpha u = q. \qedhere
\]
\end{proof}

\begin{theorem} Assume $I \trianglelefteqslant_m \End_G(\cA^d)$. 
\begin{enumerate}[label=(\roman{*})]
\item All $\varphi \in I$ share a common kernel $K_I$, which uniquely determines $I$.
\item Endomorphisms $\varphi, \psi \in I$ have the same image if and only if they belong to the same ideal subgroup $uI$, $u \in J(I)$.
\item Let $\mathcal{R}_I = \{ \Im \varphi : \varphi \in I \}$. Then $\mathcal{R}_I = \mathcal{R}_J$ for any $J \trianglelefteqslant_m \End_G(\cA^d)$. 
\item Let $\mathcal{K} = \{ K_I : I \trianglelefteqslant_m \End_G(\cA^d) \}$ and write $\mathcal{R}$ for the common $\mathcal{R}_I$ for all $I \trianglelefteqslant_m \End_G(\cA^d)$. Then the function 
\[
\{ uI : I \trianglelefteqslant_m \End_G(\cA^d), u \in J(I) \} \to \mathcal{K} \times \mathcal{R}
\]
assigning to any $uI$ the common $(\Ker \varphi, \Im \varphi)$ for all $\varphi \in uI$, is a bijection.
\end{enumerate}
\end{theorem}

\begin{proof} 
(i) By Theorem~\ref{thm:iso}, for any $J \trianglelefteqslant_m \End_G(\cA^d)$ we can write $I = \{ d_p : p \in I' \}$ and $J = \{ d_q : q \in J' \}$, where $I', J' \trianglelefteqslant_m S(\cA^d)$. It suffices to show that for $p \in I', q \in J'$ we have that $\Ker d_p = \Ker d_q$ if and only if $I' = J'$. Indeed, if $\Ker d_p = \Ker d_q$, then by Lemma~\ref{lem:ker} we have  $I' = \cl( G \cdot p) = \cl( G \cdot q ) = J'$. On the other hand, if $I' = J'$, then as in Lemma~\ref{lem:ker},
\[
\Ker d_p = \{ B \in \cA^d : [B] \cap I' = \varnothing \} = \Ker d_q.
\]

(ii) Let $u, v \in J(I)$ and fix $\varphi \in uI, \psi \in vI$. We want to show that $\Im \varphi = \Im \psi$ if and only if $uI = vI$. On the one hand, if $uI = vI$, then by simple group properties there is $\alpha \in uI$ such that $\varphi = \psi \circ \alpha$. Consequently, $\Im \varphi \subseteq \Im \psi$ and $\Im \varphi = \Im \psi$ by symmetry. On the other hand, assume that $\Im \varphi = \Im \psi$. Using the already proved implication we get $\Im u = \Im \varphi = \Im \psi = \Im v$. From Lemma~\ref{lem:im} (i) it follows that $uv = v$ (because $u$ equals identity on $\Im v = \Im u$) and from Lemma~\ref{lem:im} (ii) that $uv = u$, hence $u = uv = v$ and $uI = vI$.

(iii) Take $R \in \mathcal{R}_I$ and pick $u \in J(I)$ so that $R = \Im u$. By Theorem~\ref{thm:ell}, there is $v \in J(J)$ such that $uv = v$ and $vu = u$. It follows that $\Im v = \Im u = R$, so $R \in \mathcal{R}_J$.

(iv) Follows from (i) -- (iii).
\end{proof}

The elements of $\mathcal{R}$ are called \emph{image algebras}. By Theorem~\ref{thm:iso}, they are exactly the algebras of the form $R = d_q[\cA^d]$, where $q \in S(\cA^d)$ is almost periodic. The image algebras determine any minimal ideal (or subflow) of $S(\cA^d)$ up to $G$-flow isomorphism.

\begin{fact} \label{fact:imalg} Assume $R \in \mathcal{R}$ and $I \trianglelefteqslant_m S(\cA^d)$. Then $I \cong S(R)$ as $G$-flows. Moreover, one such $G$-flow isomorphism between $I$ and $S(R)$ is $\pi \restriction I$, where $\pi : S(\cA^d) \to S(R)$ denotes the restriction.
\end{fact}

\begin{proof}
Take $u \in J(I)$ such that $R = \Im d_u$ and consider the dual diagrams

\begin{displaymath}
\begin{minipage}{0.5 \textwidth}
\begin{center}
\begin{tikzcd}[column sep=1.5em]
& \cA^d \arrow{dr}{d_u} \\
R \arrow{ur}{\iota} \arrow{rr}{\id} & & R
\end{tikzcd}
\end{center}
\end{minipage}
\begin{minipage}{0.5 \textwidth}
\begin{center}
\begin{tikzcd}[column sep=1em]
& S(\cA^d) \arrow{dl}{\pi} \\
S(R) & & S(R) \arrow{ll}{\id} \arrow{ul}{r_u}
\end{tikzcd}
\end{center}
\end{minipage}
\end{displaymath}
Since $d_u \circ d_u = d_{u \ast u} = d_u$, the diagram on the left is commutative, hence so is the one on the right. Moreover, $d_u$ is an epimorphism of $G$-algebras, so by Fact~\ref{fact:dual-func}, $r_u$ is an injective $G$-flow morphism. Therefore it is an isomorphism onto its image (the continuity of $r_u^{-1}$ follows from compactness). By Fact~\ref{fact:im} and Lemma~\ref{lem:ker}, that image is equal to
\[
\{ p \in S(\cA^d) : p \cap \Ker d_u = \varnothing \} = \cl( G \cdot u ) = I. \qedhere
\]
Hence $I \cong S(R)$. Finally, from the commutativity of the right hand side diagram we get that $\pi \restriction I$ is the inverse of $r_u$, so it is also an isomorphism.
\end{proof}

Now we will explain the relation between image algebras and strongly generic sets. We should note that some similar ideas have already been studied in \cite{BF97} for an arbitrary semigroup $G$, but the study only applies to the universal minimal $G$-flow and has no model-theoretic connotations.

\begin{definition} \label{def:sgen} \leavevmode
\begin{enumerate}[label=(\alph{*})] 
\item A subset $B \subseteq G$ is \emph{generic} if finitely many (left) translates of $B$ cover $G$, i.e.
\[
G = \bigcup_{i=1}^n g_i \cdot B \quad \text{for some } g_1, \ldots, g_n \in G.
\]
\item A $G$-algebra $\cB \le \cP(G)$ is \emph{generic} if every set in $\cB \setminus \{ \varnothing \}$ is generic.
\item A subset $B \subseteq G$ is \emph{strongly generic} if the $G$-algebra generated by $B$ is generic.
\end{enumerate}
\end{definition}

\begin{lemma} \label{lem:gen-min}
Assume $\cB \le \cP(G)$ is a $G$-algebra. Then 
\[
\cB \text{ is generic} \iff S(\cB) \text{ is a minimal subflow of itself.}
\]
\end{lemma}
\begin{proof} $({\implies})$ Assume for contradiction there is a proper non-empty subflow $Y \subseteq S(\cB)$. Pick $p \in S(\cB) \setminus Y$ and find $B \in \cB$ such that $p \in [B]$ and $[B] \cap Y = \varnothing$. Then $B \neq \varnothing$ and $g[B] \cap Y = g ( [B] \cap Y ) = \varnothing$ for any $g \in G$. On the other hand, for some $g_1, \ldots, g_n \in G$ we have that $G = \bigcup_{i=1}^n g_i B$, which implies that $S(\cB) = [G] = \bigcup_{i=1}^n g_i [B]$, a contradiction.

\vspace{2mm} \noindent
$({\impliedby})$ Take any $B \in \cB \setminus \{ \varnothing \}$ and consider the set 
\[
G \cdot [B] := \bigcup_{g \in G} g[B].
\]
Its complement is a proper $G$-subflow of $S(\cB)$, so by assumption it must be empty, i.e. $G \cdot [B] = S(\cB) = [G]$. By compactness, $[G] = \bigcup_{i=1}^n g_i[B]$ for some $g_1, \ldots, g_n \in G$, which means that $G = \bigcup_{i=1}^n g_i B$, so $B$ is generic.
\end{proof}

\begin{corollary} \label{cor:im-gen} Image algebras are generic.
\end{corollary}

\begin{proof} This follows directly from Fact~\ref{fact:imalg} and Lemma~\ref{lem:gen-min}.
\end{proof}

\begin{corollary} If $p \in S(\cA^d)$ is almost periodic and $B \in \cA^d$, then $d_p(B)$ is strongly generic. \noproof
\end{corollary}

\begin{lemma} \label{lem:imalg-eq} Assume $R_1, R_2 \in \mathcal{R}$ and $R_1 \subseteq R_2$. Then $R_1 = R_2$.
\end{lemma}

\begin{proof} Take $I {\trianglelefteqslant_m} S(\cA^d)$ and $u, v \in J(I)$ such that $R_1 = d_u[\cA^d]$, $R_2 = d_v[\cA^d]$. It suffices to show that $u = v$. Take any $B \in \cA^d$. By assumption we can find $B' \in \cA^d$ such that $d_u B = d_v B'$. By Lemma~\ref{lem:im}, $v \ast u = v$, so
\[
d_u B = d_v B' = d_v( d_v B' ) = d_v( d_u B ) = d_{v \ast u} B = d_v B.
\]
In particular,
\[
B \in u \iff e \in d_u B \iff e \in d_v B \iff B \in v,
\]
hence $u = v$, as desired.
\end{proof}

\begin{lemma} \label{lem:gen-imalg} Every generic $G$-subalgebra $\cB \le \cA^d$ can be extended to an image algebra $R \in \mathcal{R}$.
\end{lemma}

\begin{proof} By Fact~\ref{fact:imalg}, the flow $S(\cB)$ is minimal, so the point
\[
q = \hat{e} = \{ B \in \cB : e \in B \} \in S(\cB)
\]
is almost periodic. Applying Fact~\ref{fact:per} (v) to the restriction $r : S(\cA^d) \to S(\cB)$, we can find an almost periodic point $p \in S(\cA^d)$ such that $r(p) = q$. Given $B \in \cB$, Remark~\ref{rem:dlim} gives
\[
d_p(B) = d_q(B) = d_{\hat{e}}(B) = Be^{-1} = B.
\]
Therefore $\cB \subseteq d_p[\cA^d]$ and $R := d_p[\cA^d]$ is an image algebra, since $p \in S(\cA^d)$ is almost periodic.
\end{proof}

\begin{corollary} \label{cor:im-maxgen} Image algebras are precisely maximal generic $G$-subalgebras of $\cA^d$.
\end{corollary}

\begin{proof} First assume $R$ is an image algebra. By Corollary~\ref{cor:im-gen}, $R$ is generic, so it extends to a maximal generic $G$-algebra $R \subseteq \cB \subseteq \cA^d$. By Lemma~\ref{lem:gen-imalg}, $\cB$ further extends to an image algebra $R_2 \supseteq \cB$. It follows that $R \subseteq R_2$, so by Lemma~\ref{lem:imalg-eq}, $R = R_2 = \cB$. Thus $R$ is maximal generic.

Now assume $\cB \subseteq \cA^d$ is a maximal generic $G$-algebra. By Lemma~\ref{lem:gen-imalg}, $\cB$ can be extended to an image algebra $R \supseteq \cB$, which is generic by Corollary~\ref{cor:im-gen}. Hence $\cB = R$ and so $\cB$ is an image algebra.
\end{proof}

The next proposition is a joint result of Newelski and the author. It shows how to recover the Ellis group of $S(\cA)$ from any image algebra.

\begin{proposition} \label{prop:im-ell} Assume $\cB \le \cA^d$ is an image algebra and define
\[
\cE = \{ q \in S(\cB) : d_q[ \cB ] \subseteq \cB \}.
\]
\begin{enumerate}[label=(\roman{*})]
\item The operation 
\[
p \ast q = \{ B \in \cB : d_q B \in p \}
\]
is well defined as an operation $\ast : \cE \times \cE \to S(\cB)$.
\item $\cE$ is closed under $\ast$.
\item $(\cE, \ast)$ is isomorphic to the Ellis group of $S(\cA)$.
\end{enumerate}
\end{proposition}

\begin{proof} 
(i) If $q \in \cE$, then by definition $d_q : \cB \to \cB$, so we have the dual map $r_q : S(\cB) \to S(\cB)$. As usual, for any $p \in S(\cB)$ we have that $p \ast q = r_q(p)$.

\vspace{2mm} \noindent 
(ii) Fix $p, q \in \cE$ and $B \in \cB$. Applying Fact~\ref{fact:dual} to $d_q : \cB \to \cB$, we get
\[
d_{p \ast q}(B) = d_{r_q(p)}(B) = d_p( d_q(B) ) \in \cB.
\]
Hence $p \ast q \in \cE$.

\vspace{2mm} \noindent
(iii) Take $I \trianglelefteqslant_m S(\cA^d)$ and pick $u \in J(I)$ such that $\cB = d_u[\cA^d]$. By Fact~\ref{fact:imalg}, the restriction $\pi : I \to S(\cB)$ is a $G$-flow isomorphism. We first show that $\pi[uI] = \cE$. If $q \in uI$, then $d_q[\cA^d] = d_u[\cA^d] = \cB$, so in particular $d_{\pi(q)}[\cB] = d_q[\cB] \subseteq \cB$, hence $\pi(q) \in \cE$. On the other hand, assume that $q_0 \in \cE$ and take $q \in I$ such that $\pi(q) = q_0$. Then $q \ast u = q$ by Lemma~\ref{lem:im}, so
\[
d_q[\cA^d] = d_{q \ast u}[\cA^d] = d_q[d_u[\cA^d]] = d_q[\cB] = d_{q_0}[\cB] \subseteq \cB = d_u[\cA^d].
\]
By Lemma~\ref{lem:imalg-eq}, these algebras are equal. Hence $q \in uI$ and so $q_0 \in \pi[uI]$.

It follows that $\pi$ is a bijection between $uI$ and $\cE$. It is also a homomorphism since for $p, q \in uI$ and $B \in \cB$,
\begin{align*}
B \in \pi(p \ast q) & \iff B \in p \ast q \iff d_q B \in p \iff d_{\pi(q)} B \in p \\
& \iff d_{\pi(q)} B \in \pi(p) \iff B \in \pi(p) \ast \pi(q).
\end{align*}
Hence $\cE$ is isomorphic to the Ellis group $uI$.
\end{proof}

The results from this section lead to an interesting way of studying the Ellis group of $S(\cA)$ that only refers to the notion of a strongly generic set. Take any strongly generic set $B \in \cA^d$ (e.g. the empty set) so that the $G$-algebra $\cB \le \cA^d$ generated by $B$ is generic. Extend $\cB$ to a maximal generic $G$-subalgebra $R$ of $\cA^d$, which is also an image algebra. Every minimal ideal $I \trianglelefteqslant_m E(S(\cA))$ is isomorphic to $S(R)$ as a $G$-flow and the Ellis group of $S(\cA)$ is isomorphic to $\cE$ as defined in Proposition~\ref{prop:im-ell}. Thus if the strongly generic subsets of $G$ can be understood to the point of characterizing the maximal generic $G$-subalgebras of $\cA^d$, a complete description of the Ellis group of $S(\cA)$ will follow.

\begin{corollary} \label{cor:im-ell}
Assume $R \in \mathcal{R}$ and $R \le \cB \le \cA^d$. Then the Ellis groups of $S(\cA)$ and $S(\cB)$ are isomorphic. 
\end{corollary}
\begin{proof}
By Corollary \ref{cor:im-maxgen}, $R$ is a maximal generic $G$-subalgebra of $\cA^d$. Since $R \le \cB^d \le \cA^d$, $R$ is also a maximal generic $G$-subalgebra of $\cB^d$, hence it is an image algebra in $\cB^d$. By Proposition \ref{prop:im-ell}, the Ellis groups of $S(\cA)$ and $S(\cB)$ can be computed directly from $R$, giving the same result up to isomorphism.
\end{proof}

\section{Strongly generic sets}

In this section we study abstract properties of strongly generic sets. We also provide a way to construct non-trivial examples of such sets in a wide class of groups, namely those that do not have the smallest subgroup of finite index.

\subsection{Basic properties}

Throughout the section assume that $G$ is a group. We regard $2^G$ as a $G$-flow, where for $g \in G, f \in 2^G$, the action is defined as
\[
(g \odot f)(x) = f(xg) \text{ for } x \in G.
\]
Equivalently, $g \odot \chi_A = \chi_{Ag^{-1}}$ for $A \subseteq G$, where $\chi_A$ denotes the characteristic function of $A$. 

The simplest examples of strongly generic sets are periodic sets, defined below:

\begin{definition} Assume $f \in 2^G$.
\begin{enumerate}[label=(\roman{*})]
\item $t \in G$ is a \emph{period} of $f$ if $(\forall x \in G) \, f(x \cdot t) = f(x)$.
\item $\Per(f)$ is the set of periods of $f$:
\[
\Per(f) = \{ t \in G : (\forall x \in G) \, f(x \cdot t) = f(x) \}.
\]
By the next remark, $\Per(f)$ is a subgroup of $G$.
\item $f$ is \emph{periodic} if $\Per(f)$ has finite index in $G$.
\item A subset $A \subseteq G$ is \emph{periodic} if $\chi_A \in 2^G$ is periodic.
\end{enumerate}
\end{definition}

\begin{remark} $\Per(f)$ is the stabilizer of $f$ in the $G$-flow $2^G$. In particular, it is a subgroup of $G$.
\end{remark}

$\Per(f)$ can also be understood as follows: assume $A \subseteq G$ and $H \le G$. Clearly the following are equivalent:
\begin{align*}
& H \subseteq \Per(\chi_A) \iff (\forall h \in H) \, Ah = A \iff AH = A \\
& \iff A \text{ is a union of some left cosets of } H.
\end{align*}
As a consequence:
\begin{remark} \label{rem:per-max}
$\Per(\chi_A)$ is the greatest subgroup $H \le G$ such that $A$ can be written as a union of left cosets of $H$. \noproof
\end{remark}

The following remark is a useful characterization of periodic sets.
\begin{remark} \label{rem:per} For $A \subseteq G$ the following are equivalent:
\begin{enumerate}[label=(\roman{*})]
\item $A$ is periodic;
\item $A$ is a union of left cosets of some subgroup $H \le G$ of finite index;
\item The $G$-algebra $\< A \> \subseteq \cP(G)$ generated by $A$ is finite;
\item $A^{-1} := \{ a^{-1} : a \in A \}$ is periodic.
\end{enumerate}
\end{remark}

\begin{proof}
(i)$\implies$(ii) Follows directly from Remark~\ref{rem:per-max}.

\vspace{2mm} \noindent
(ii)$\implies$(iii) Assume $A$ is a union of left cosets of a subgroup $H \le G$ of finite index. All such unions form a finite $G$-algebra of size $2^{[G:H]}$ in which $\< A \>$ is contained, hence it is also finite.

\vspace{2mm} \noindent
(iii) $\implies$ (iv) From the assumption the set $\{ gA : g \in G \}$ is finite, hence also finite is the set 
\[
\{ \chi_{A^{-1} g^{-1}} : g \in G \} = \{ g \cdot \chi_{A^{-1}}  : g \in G  \},
\]
which is the orbit of $\chi_{A^{-1}}$ in $2^G$. It follows that the stabilizer of $\chi_{A^{-1}}$ has finite index in $G$, so $A^{-1}$ is periodic. 

\vspace{2mm} \noindent
(iv)$\implies$(i) Follows from the converse implication by symmetry.
\end{proof}

\begin{corollary}  \label{cor:per-bc}
Assume $A \subseteq G$ is periodic. Then $\Per(\chi_A)$ can be expressed as an intersection of finitely many left translates of $A$ and $A^c$.
\end{corollary}

\begin{proof} By definition
\[
\Per(\chi_A) = \bigcap_{a \in A} a^{-1} A \cap \bigcap_{b \in G \setminus A} b^{-1} A^c.
\]
By Remark~\ref{rem:per} there are finitely many sets of the form $gA$ where $g \in G$, so the conclusion follows.
\end{proof}

\begin{corollary} \label{cor:per-nbc} Assume $A \subseteq G$ is periodic. Then there is a normal subgroup $N \trianglelefteqslant G$ of finite index such that $N \in \< A \>^d$ and $A$ is a union of cosets of $N$.
\end{corollary}

\begin{proof} Let 
\[
N := \bigcap_{g \in G} g \Per(\chi_A) g^{-1} \trianglelefteqslant G.
\]
Since $\Per(\chi_A)$ has finite index, there are only finitely many distinct sets in this intersection. Each of them belongs to $\< A \>^d$, since by Corollary~\ref{cor:per-bc} we have $\Per(\chi_A) \in \< A \>$ and so $g \Per(\chi_A) g^{-1} \in \< A \>^d$ by Corollary~\ref{cor:dcl-rt}. Hence $N \in \< A \>^d$ and $N$ has finite index. Also $N \subseteq \Per(\chi_A)$, so by Remark~\ref{rem:per-max}, $A$ is a union of cosets of $N$.
\end{proof}

It is now easy to see that every periodic subset $A \subseteq G$ is strongly generic. Indeed, if $A$ is periodic, then by Remark~\ref{rem:per} any non-empty element of $\< A \>$ contains at least one coset of some subgroup $H \le G$ of finite index (e.g. $H = \Per(\chi_A)$), hence it is generic. It was shown in \cite[Proposition 2.8]{New12Feb} that in the context of stable groups these are the only definable strongly generic subsets of $G$. Below is a known natural generalization of that result.

\begin{theorem} \label{thm:stab-sg-per} If $A \subseteq G$ is strongly generic and the formula $\varphi(x; y) \equiv x \in yA$ is stable in $(G, \cdot, A)$, then $A$ is periodic.
\end{theorem}

\begin{proof} Note that the $G$-algebra $\cA \le \cP(G)$ generated by $A$ contains an atom. Indeed, otherwise we can find $\< b_{\eta} : \eta \in 2^{<\omega} \>$ such that for each $\alpha \in 2^{\omega}$ the family $\{ b_{\alpha \restriction n} A^{\alpha(n)} : n \in \omega \}$ has the finite intersection property. It follows that there are $2^{\aleph_0}$-many $\varphi$-types over the countable parameter set $\{ b_{\eta} : \eta \in 2^{<\omega} \}$, which contradicts the stability of $\varphi$.

Choose an atom $B \in \cA$. Since $A$ is strongly generic, $B$ is generic, so $G$ is a union of finitely many left translates of $B$, which are also atoms. Therefore $\cA$ is finite and the conclusion follows from Remark~\ref{rem:per}.
\end{proof}

The stability assumption is essential -- as we will show in the next subsection, strongly generic non-periodic sets exist in general. However, the property of being a strongly generic set turns out to be equivalent to a specific weaker version of periodicity.

\begin{definition} Assume $f \in 2^G$.
\begin{enumerate}[label=(\roman{*})]
\item For any finite $U \subseteq G$, an element $t \in G$ is a $U$\emph{-period} of $f$ if 
\[
(\forall x \in U) \, f(xt) = f(x).
\]
\item $\Per_U(f)$ is the set of $U$-periods of $f$.
\item $f$ is \emph{locally periodic} if for each finite $U \subseteq G$ the set $\Per_U(f)$ is generic.
\item A subset $A \subseteq G$ is \emph{locally periodic} if $\chi_A$ is locally periodic.
\end{enumerate}
\end{definition}

Note that $\Per_U(f)$ is typically not a subgroup of $G$, so in (iii) we can no longer expect it to be a subgroup of finite index; the correct condition here is for it to be a generic set.

We will also utilize the notion of a \emph{self-replicating function}. Although it was originally defined for functions $f \in 2^N$ where $N$ is a $\QQ$-vector space (Definition 3.12 in \cite{New09}), it can be naturally interpreted when $N = G$ is an arbitrary group:

\begin{definition} \label{def:rep} We say that $f \in 2^G$ is self-replicating if
\[
(\forall \underset{\text{finite}}{U \subseteq G})(\exists \underset{\text{finite}}{V \subseteq G})(\forall g \in G)(\exists h \in G) \, h \odot (f \restriction U) \subseteq f \restriction Vg.
\]
The expression $h \odot (f \restriction U)$ here means the partial function defined by the same formula as in the beginning of the section: 
\[
\big(h \odot (f \restriction U) \big)(x) = (f \restriction U)(xh) \quad \text{for } x \in Uh^{-1}.
\]
\end{definition}

In \cite{New09} it was proved that self-replication is a sufficient condition for almost periodicity. The following proposition states that in fact these conditions are equivalent to each other, and also to the property of being locally periodic.

\begin{proposition} \label{prop:ap-char} For $f \in 2^G$ the following are equivalent:
\begin{enumerate}[label=(\roman{*})]
\item $f$ is self-replicating;
\item $f$ is an almost periodic point of $2^G$;
\item $f$ is locally periodic.
\end{enumerate}
\end{proposition}

\begin{proof} (i)$\implies$(ii) Take any basic open neighbourhood of $f$. We can write it as
\[
[\sigma] = \{ f' \in 2^G : \sigma \subseteq f' \},
\]
where $\sigma = f \restriction U$ for some finite $U \subseteq G$. By Remark~\ref{rem:min-gen}, it suffices to show that $\cl(G \odot f) \subseteq W^{-1} \odot [\sigma]$ for some finite $W \subseteq G$. From the assumption there is a finite $V \subseteq G$ such that
\[
(\forall g \in G)(\exists h \in G) \, h \odot (f \restriction U) \subseteq f \restriction Vg.
\]
Let $W = \{ w \in G : Uw \subseteq V \}$ and take any $f' \in \cl(G \odot f)$. We have that $f' \in [f' \restriction V] \cap \cl(G \odot f)$, so there is $g \in G$ such that $g \odot f \in [f' \restriction V]$, or equivalently $g \odot (f \restriction Vg) = f' \restriction V$. Take $h \in H$ such that $h \odot (f \restriction U) \subseteq f \restriction Vg$. It follows that
\[
gh \odot (f \restriction U) \subseteq g \odot (f \restriction Vg) = f' \restriction V.
\]
Let $w := (gh)^{-1}$. By the above, $Uw \subseteq V$, hence $w \in W$ and $w \odot f' \in [\sigma]$, as needed.

\vspace{2mm} \noindent
(ii)$\implies$(iii) Take any finite $U \subseteq G$. By Remark~\ref{rem:min-gen}, we can find a finite set $W \subseteq G$ such that $\cl(G \odot f) \subseteq W^{-1} \odot [f \restriction U]$. It suffices to show that $W^{-1} \cdot \Per_U(f) = G$. For any $g \in G$ we can find $w \in W$ such that $g \odot f \in w^{-1} \odot [f \restriction U]$, i.e. $f \restriction U \subseteq w g \odot f$. It follows that $w g \in \Per_U(f)$. 

\vspace{2mm} \noindent
(iii)$\implies$(i) Fix any finite $U \subseteq G$. By assumption $\Per_U(f)$ is generic, so there is a finite $W \subseteq G$ such that $W^{-1} \cdot \Per_U(f) = G$. We will show that the set $V = U \cdot W$ has the required property. Take any $g \in G$ and write $g = w^{-1} h$, where $w \in W, h \in \Per_U(f)$. Then $Uh = Uwg \subseteq Vg$ and $h^{-1} \odot (f \restriction U) \subseteq f$, so $h^{-1} \odot (f \restriction U) \subseteq f \restriction Vg$, as required.
\end{proof}

\begin{proposition} \label{prop:sgen-lper} A subset $A \subseteq G$ is strongly generic if and only if it is locally periodic.
\end{proposition} 

\begin{proof} Write $A^0 := G \setminus A$ and $A^1 := A$. The following are equivalent:
\begin{itemize}[label=--]
\item $A$ is strongly generic;
\item For any Boolean term $\tau(x_1, \ldots, x_n)$ and any $u_1, \ldots, u_n \in G$, if the set $\tau(u_1^{-1} A, \ldots, u_n^{-1} A)$ is non-empty, then it is generic;
\item For any $\varepsilon_1, \ldots, \varepsilon_n \in \{ 0, 1 \}$ and $u_1, \ldots, u_n \in G$, if the set $\bigcap_{j=1}^n u_j^{-1} A^{\varepsilon_j}$ is non-empty, then it is generic;
\item For any $\varepsilon_1, \ldots, \varepsilon_n \in \{ 0, 1 \}$ and $u_1, \ldots, u_n \in G$, if $e \in \bigcap_{j=1}^n u_j^{-1} A^{\varepsilon_j}$, then this set is generic;
\item For any $u_1, \ldots, u_n \in G$ the set $\bigcap_{j=1}^n u_j^{-1} A^{\chi_A(u_j)}$ is generic.
\end{itemize}

Let $U = \{ u_1, \ldots, u_n \} \subseteq G$ and $t \in G$. We have that
\begin{align*}
t \in \bigcap_{j=1}^n u_j^{-1} A^{\chi_A(u_j)} & \iff \bigwedge_{j=1}^n u_j t \in A^{\chi_A(u_j)} \\ 
& \iff \bigwedge_{j=1}^n \chi_A(u_j t) = \chi_A(u_j) \\ 
& \iff (\forall u \in U) \, \chi_A(u t) = \chi_A(u) \\[1ex]
& \iff t \in \Per_U(\chi_A),
\end{align*}
hence
\[
\bigcap_{j=1}^n u_j^{-1} A^{\chi_A(u_j)} = \Per_U(\chi_A).
\]
By the previous equivalences, the set $A$ is strongly generic if and only if for every finite $U \subseteq G$, the set $\Per_U(\chi_A)$ is generic, i.e. $A$ is locally periodic.
\end{proof}

\begin{theorem} \label{thm:sgen-aper} For $A \subseteq G$,
\[
A \text{ is strongly generic} \iff \chi_A \text{ is an almost periodic point of } 2^G.
\]
\end{theorem}

\begin{proof} Follows directly from the last two propositions.
\end{proof}

Using these results we give an alternative proof of Corollary~\ref{cor:im-gen}.

\begin{proof}[Proof of Corollary~\ref{cor:im-gen}]
Take any almost periodic point $q \in S(\cA^d)$ and let $B = d_q(A) \in \Im d_q$. By Fact~\ref{fact:per} (iv) and Remark~\ref{rem:dlim} (ii), the point $\chi_B$ is almost periodic in $2^G$. Therefore by Theorem~\ref{thm:sgen-aper}, the set $B$ is strongly generic.
\end{proof}

Consider the additive group of integers, $G = (\ZZ, +)$. It is well known that there are almost periodic points in $2^{\ZZ}$ that are not periodic functions. It follows from Theorem~\ref{thm:sgen-aper} that there is a strongly generic subset $A \subseteq \ZZ$ that is not periodic. In the next subsection we explicitly describe a whole class of such subsets. 

\subsection{Concrete examples}

Consider a finitely branching tree $T \subseteq \omega^{< \omega}$ with leaves $T_L \subseteq T$. Pick indexed families $\mathcal{H} = \< H_{\eta} : \eta \in T \>$ of subgroups of $G$ and $d = \< d_{\eta} : \eta \in T \>$ of elements of $G$, satisfying for any $\eta \in T$:
\begin{itemize}
\item $H_{\varnothing} = G$;
\item $H_{\eta^\frown i} = H_{\eta^\frown j}$ whenever $\eta^\frown i, \eta^\frown j \in T$;
\item $H_{\eta^\frown i}$ is a proper subgroup of $H_{\eta}$ of finite index whenever $\eta^\frown i \in T$;
\item $d_{\eta} H_{\eta}$ is a disjoint union of $\{ d_{\eta^\frown i} H_{\eta^\frown i} : \eta^\frown i \in T \}$ if $\eta \notin T_L$.
\end{itemize}
Assume that each $g \in G$ belongs to some (necessarily unique) coset $d_{\eta} H_{\eta}$, where $\eta \in T_L$. Equivalently, the intersection of cosets $d_{\eta} H_{\eta}$ along any infinite branch of $T$ is empty. Thus we can define $f_T : G \to T_L$ so that $f_T(g) = \eta$ for the unique $\eta \in T_L$ such that $g \in d_{\eta} H_{\eta}$.

\begin{definition} \label{def:tree} \leavevmode
\begin{enumerate}[label=(\roman{*})]
\item A \emph{tree of cosets} of $G$ is a tuple $(T, \mathcal{H}, d)$ satisfying the above assumptions.
\item A \emph{valued tree of cosets} of $G$ is a tuple $(T, \mathcal{H}, d, v)$, where $(T, \mathcal{H}, d)$ is a tree of cosets of $G$ and $v : T_L \to \{ 0, 1 \}$.
\item A function $f : G \to \{ 0, 1 \}$ is \emph{founded} on the valued tree of cosets $(T, \mathcal{H}, d, v)$ if $f = v \circ f_T$. It is \emph{founded} on the tree of cosets $(T, \mathcal{H}, d)$ if it is founded on $(T, \mathcal{H}, d, v)$ for some $v$.
\item A function $f : G \to \{ 0, 1 \}$ is \emph{tree-founded} if it is founded on some tree of cosets of $G$.
\item A subset $A \subseteq G$ is \emph{tree-founded} if $\chi_A : G \to \{ 0, 1 \}$ is tree-founded.
\end{enumerate}
\end{definition}

The generality of the definition is convenient, but it is possible to put a simplifying restriction without narrowing down the class of tree-founded functions. Namely, assume $f$ is a function founded on the tree of cosets $(T, \mathcal{H}, d)$. Then without loss of generality we can assume that each $H_{\eta}$ depends only on $|\eta|$, i.e. $H_{\eta} = A_{|\eta|}$ for some decreasing sequence of subgroups of finite index $G = A_0 \ge A_1 \ge A_2 \ge \ldots$. Moreover, we can assume that each $A_n$ is a normal subgroup of $G$ so that the left and right cosets coincide. When $(T, \mathcal{H}, d)$ satisfies these additional assumptions, we shall call it  a \emph{linear tree of cosets}.

To see this, let $A_n$ be the intersection of all conjugates of $H_{\eta}$ with $|\eta| \le n$. It is easy to prove that $A_n$ is a decreasing sequence of normal subgroups of finite index in $G$. The node of $T$ corresponding to each coset of $H_{\eta}$ can be split into finitely many nodes corresponding to all cosets of $A_{|\eta|}$ contained in it and the function $v$ can be modified accordingly.

\begin{example} \label{ex:tf} Take any $2$-adic integer $\alpha \in \ZZ_2 \setminus \ZZ$. Let $H_n = 2^n \ZZ \le \ZZ$ and $d_n = \alpha \bmod{2^n}$ so that $\< d_n + H_n : n \in \NN \>$ is a descending sequence of cosets in $\ZZ$ with empty intersection. Define $f : \ZZ \to \{ 0, 1 \}$ by
\[
\begin{cases}
f \restriction \ZZ \setminus (d_1 + H_1) = 0, \\
f \restriction (d_1 + H_1) \setminus (d_2 + H_2) = 1, \\
f \restriction (d_2 + H_2) \setminus (d_3 + H_3) = 0, \\
f \restriction (d_3 + H_3) \setminus (d_4 + H_4) = 1, \\
\vdots
\end{cases}
\]
and so on. Then $f$ is tree-founded. Figure~\ref{fig:vt-ex} illustrates the construction for
\[
\alpha = 1 \cdot 2^0 + 0 \cdot 2^1 + 1 \cdot 2^2 + 1 \cdot 2^3 + \ldots = 1 + 0 + 4 + 8 + \ldots \in \ZZ_2 \setminus \ZZ.
\]

\begin{figure}
\centering
\begin{tikzpicture}[scale=1.4, coset/.style = {draw, rectangle}, value/.style= { draw, circle } ]

\node[coset] (1) at (0, 0) {$\ZZ$};

\node[coset] (2) at (-1, -1) {$2\ZZ$};
\node[value, below=0 of 2] {$0$};
\draw (1) -- (2);

\node[coset] (3) at (1, -1) {$2\ZZ+1$};
\draw (1) -- (3);

\node[coset] (4) at (0, -2) {$4\ZZ+3$};
\node[value, below=0 of 4] {$1$};
\draw (3) -- (4);

\node[coset] (5) at (2, -2) {$4\ZZ+1$};
\draw (3) -- (5);

\node[coset] (6) at (1, -3) {$8\ZZ+1$};
\node[value, below=0 of 6] {$0$};
\draw (5) -- (6);

\node[coset] (7) at (3, -3) {$8\ZZ+5$};
\draw (5) -- (7);

\node[coset] (8) at (2, -4) {$16\ZZ+5$};
\node[value, below=0 of 8] {$1$};
\draw (7) -- (8);

\node[coset] (9) at (4, -4) {$16\ZZ+13$};
\draw (7) -- (9);

\node [below right] at (9) {$\ddots$};
\end{tikzpicture}
\caption{A valued tree of cosets defining a tree-founded function. }
\label{fig:vt-ex}
\end{figure}
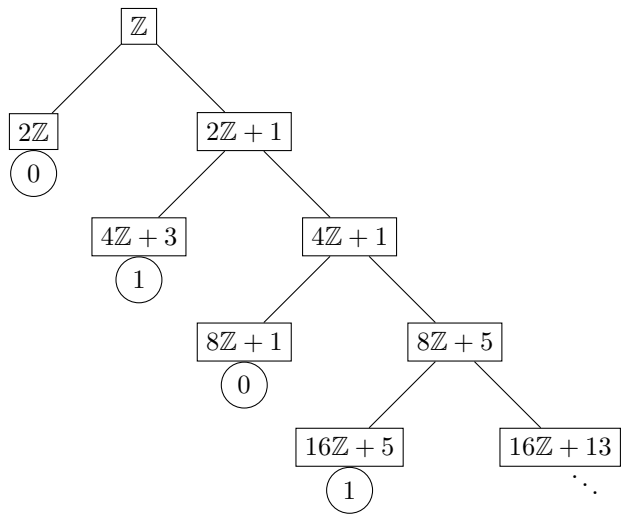
\end{example}

Now we set to prove that every tree-founded set is strongly generic. We will also give a natural sufficient condition for a tree-founded set to be non-periodic. Assuming $G$ has sufficiently many subgroups of finite index, this describes a large class of explicitly definable strongly generic, non-periodic subsets of $G$.

Consider the topology on $G$ generated by cosets of subgroups of $G$ of finite index, which is called \emph{the profinite topology} on $G$. Contrary to what the name can suggest, it need be neither compact nor Hausdorff. For instance, when $G = (\QQ, +)$, the topology consists of just two sets, $\varnothing$ and $\QQ$. On the other hand, when $G = (\ZZ, +)$, the topology is Hausdorff but not compact -- as demonstrated in the example above, where infinitely many disjoint cosets cover $\ZZ$. Finally, let $G = (\hat{\ZZ}, +)$ be the profinite completion of $(\ZZ, +)$. Then the profinite topology on $G$ is the usual topology on $\hat{\ZZ}$ treated as a profinite group, hence it is both Hausdorff and compact. In this case every tree-founded subset of $G$ is founded on a finite tree, as an infinite tree would have an infinite branch, which by compactness would have a non-empty intersection. Therefore in $\hat{\ZZ}$, tree-founded sets are precisely periodic sets.

We will say that a subset $A \subseteq G$ is \emph{pf-clopen} if it is clopen in the profinite topology. We will also say that a function $f : G \to \{ 0, 1 \}$ is \emph{pf-continuous} if it is continuous with respect to the profinite topology on $G$. Clearly $A \subseteq G$ is pf-clopen if and only if $\chi_A : G \to \{ 0, 1 \}$ is pf-continuous.

\begin{remark} \label{rem:tree-clop} Every tree-founded set $A \subseteq G$ is pf-clopen. 
\end{remark}

\begin{proof} Take the tree of cosets $(T, \mathcal{H}, d)$ on which $\chi_A$ is founded. Let $a \in A$ and take $\eta = f_T(a)$ so that $\eta$ is a leaf in $T$ and $a \in d_{\eta} H_{\eta}$. Then $\chi_A$ is constant on $d_{\eta} H_{\eta}$, so $d_{\eta} H_{\eta} \subseteq A$, hence $A$ is open. By the same reasoning if $b \in G \setminus A$, there is a coset $d_{\eta} H_{\eta}$ such that $b \in d_{\eta} H_{\eta} \subseteq G \setminus A$, so $G \setminus A$ is open as well.
\end{proof}

The converse of Remark~\ref{rem:tree-clop} is also true provided that the family of finite index subgroups of $G$ ordered by inverse inclusion has countable cofinality, i.e. if there is a decreasing sequence $(H_n)$ of finite index subgroups of $G$ such that each finite index subgroup of $G$ has some $H_n$ as a subgroup.

\begin{proposition} Every pf-clopen set is strongly generic. 
\end{proposition}

\begin{proof} All pf-clopen sets form a $G$-algebra of subsets of $G$. The algebra is generic, since if a pf-clopen set is non-empty, then it contains a coset of a subgroup of finite index, hence it is generic. The conclusion follows.
\end{proof}

\begin{proposition} \label{prop:tf-sg} Every tree-founded subset of $G$ is strongly generic. \noproof
\end{proposition}

Some valued trees are needlessly complicated with respect to the tree-founded function they define. Figure~\ref{fig:red} shows tree simplifications that do not change the functions founded on them.

\begin{figure}
\centering
\begin{tikzpicture}[scale=0.8]
\def\r{0.4};

\node [above] at (0, 0) {$G$};

\draw (0, 0) -- (-1.5, -0.5);
\draw (0, 0) -- (1.5, -0.5);

\draw (-1.5, -0.5) -- (-2.5, -1.5);
\draw (-1.5, -0.5) -- (-1.5, -1.5);
\draw (-1.5, -0.5) -- (-0.5, -1.5);

\draw (-2.5, -1.5-\r) circle (\r) node {$1$};
\draw (-1.5, -1.5-\r) circle (\r) node {$1$};
\draw (-0.5, -1.5-\r) circle (\r) node {$1$};

\draw (1.5, -0.5) -- (1, -1.5);
\draw (1.5, -0.5) -- (2, -1.5);

\draw (1, -1.5-\r) circle (\r) node {$0$};
\draw (2, -1.5-\r) circle (\r) node {$0$};

\draw (6, -1.1) -- (6.6, -1.1) -- (6.6, -1.25) -- (7, -1) -- (6.6, -0.75) -- (6.6, -0.9) -- (6, -0.9) -- cycle;

\def\x{10};

\node [above] at (\x, 0) {$G$};

\draw (0+\x, 0) -- (-1+\x, -0.5);
\draw (0+\x, 0) -- (1+\x, -0.5);

\draw (-1+\x, -0.5-\r) circle (\r) node {$1$};
\draw (1+\x, -0.5-\r) circle (\r) node {$0$};

\def\y{-4};

\node [above] at (0, 0+\y) {$G$};

\draw (0, 0+\y) -- (-1, -1+\y);
\draw (0, 0+\y) -- (1, -1+\y);

\draw (-1, -1+\y-\r) circle (\r) node {$0$};

\draw (1, -1+\y) -- (0, -2+\y);
\draw (1, -1+\y) -- (1, -2+\y);
\draw (1, -1+\y) -- (2.5, -2+\y);

\draw (0, -2+\y-\r) circle (\r) node {$1$};
\draw (1, -2+\y-\r) circle (\r) node {$1$};

\draw (2.5, -2+\y) -- (1.5, -3+\y);
\draw (2.5, -2+\y) -- (3.5, -3+\y);

\draw (1.5, -3+\y-\r) circle (\r) node {$1$};

\draw (3.5, -3+\y) -- (2.5, -4+\y);
\draw (3.5, -3+\y) -- (3.5, -4+\y);
\draw (3.5, -3+\y) -- (4.5, -4+\y);

\draw (3.5, -4+\y-\r) circle (\r) node {$1$};
\draw (4.5, -4+\y-\r) circle (\r) node {$1$};

\draw (2.5, -4+\y) -- (2, -5+\y);
\draw (2.5, -4+\y) -- (2.5, -5+\y) node [below] {$\vdots$};
\draw (2.5, -4+\y) -- (3, -5+\y);

\draw (6, -1.1+\y) -- (6.6, -1.1+\y) -- (6.6, -1.25+\y) -- (7, -1+\y) -- (6.6, -0.75+\y) -- (6.6, -0.9+\y) -- (6, -0.9+\y) -- cycle;

\node [above] at (\x, \y) {$G$};

\draw (0+\x, 0+\y) -- (-1+\x, -0.5+\y);
\draw (0+\x, 0+\y) -- (1+\x, -0.5+\y);

\draw (-1+\x, -0.5+\y-\r) circle (\r) node {$0$};
\draw (1+\x, -0.5+\y-\r) circle (\r) node {$1$};
\end{tikzpicture}
\caption{Reducible trees and their reductions.}
\label{fig:red}
\end{figure}
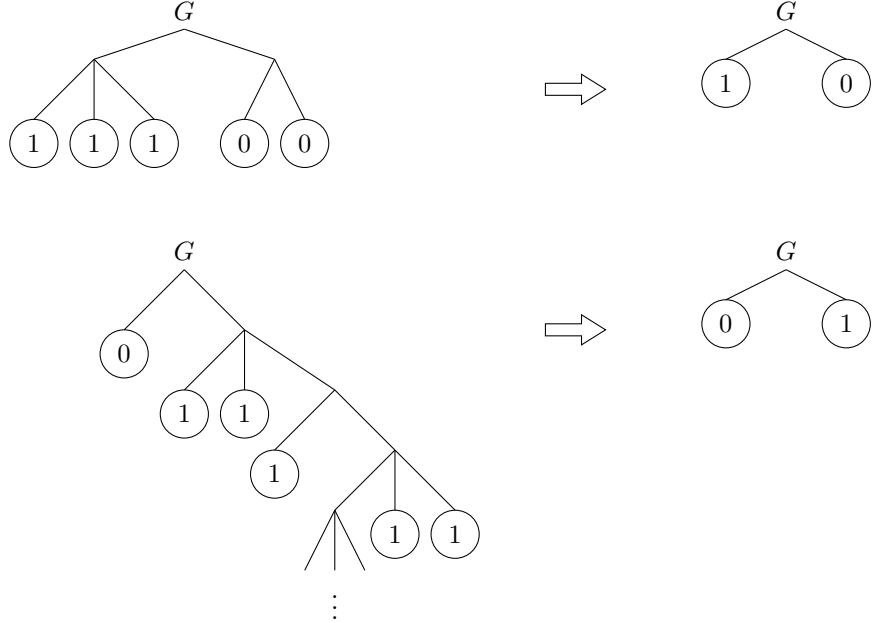

\begin{definition} Assume $(T, \mathcal{H}, d, v)$ is a valued tree of cosets of $G$.
\begin{enumerate}[label=(\roman{*})]
\item A node $\eta \in T$ is \emph{homogeneous} if $v$ is constant on the set of leaves extending $\eta$, or equivalently, if $v \circ f_T$ is constant on $d_{\eta} H_{\eta}$.
\item $(T, \mathcal{H}, d, v)$ is \emph{irreducible} if every homogeneous $\eta \in T$ is a leaf. Otherwise it is \emph{reducible}.
\end{enumerate}
\end{definition}

\begin{remark} Assume $f : G \to \{ 0, 1 \}$ is founded on a valued tree of cosets $(T, \mathcal{H}, d, v)$. Then $f$ is founded on some irreducible valued tree of cosets.
\end{remark}

\begin{proof} Let $T'$ be the tree obtained by removing from $T$ all proper extensions of $\eta$ for each minimal homogeneous $\eta \in T$. Thus each such $\eta$ will become a leaf in $T'$. Define $v'(\eta)$ for such $\eta$ as the common value of $v$ on the leaves extending $\eta$ in $T$ and $v' = v$ elsewhere. Then $(T', \mathcal{H} \restriction T', d \restriction T', v')$ is an irreducible valued tree of cosets of $G$ and $f = v' \circ f_{T'}$.
\end{proof}

For the purpose of stating the next lemma, we introduce some auxiliary notation.
\begin{notation*} Assume $f : G \to \{ 0, 1 \}$.
\begin{enumerate}[label=(\roman{*})]
\item Denote by $\mathcal{F}$ the family of subgroups of $G$ of finite index;
\item For $H \in \mathcal{F}$, let $\chi(f, H) = \bigcup \{ A \in G/H : f \text{ is constant on } A \}$;
\item For $H, H' \in \mathcal{F}$, write $H \sqsubsetneq_f H'$ if $H \le H'$ and $\chi(f, H) \supsetneq \chi(f, H')$.
\end{enumerate}
\end{notation*}

\begin{lemma} \label{lem:per-wf} Assume $f : G \to \{ 0, 1 \}$ is pf-continuous. The following are equivalent:
\begin{enumerate}[label=(\roman{*})]
\item $f$ is periodic;
\item There is $N \in \NN$ such that each sequence $G \sqsupsetneq_f H_1 \sqsupsetneq_f H_2 \sqsupsetneq_f \ldots \sqsupsetneq_f H_n$ has length $n \le N$;
\item There is no infinite sequence $G \sqsupsetneq_f H_1 \sqsupsetneq_f H_2 \sqsupsetneq_f \ldots$, i.e. the strict partial order $(\mathcal{F}, \sqsubsetneq_f)$ is well-founded.
\end{enumerate}
\end{lemma}

\begin{proof}
(i)$\implies$(ii) Take a normal subgroup $T \trianglelefteqslant G$ of finite findex such that $T \le \Per(f)$. Consider any subgroup $H \le G$ such that $f$ is constant on some coset $aH$, $a \in G$. Then $HT \le G$ and $f$ is constant on $aHT$ since for any $h, h' \in H, t, t' \in T$: $f(aht) = f(ah) = f(ah') = f(ah't')$. Since $aH \subseteq aHT$ and $aHT$ is a union of left cosets of $T$, it follows that $\chi(f, H) = \bigcup \mathcal{T}$ for some family $\mathcal{T} \subseteq G/T$. Therefore any sequence $G \sqsupsetneq_f H_1 \sqsupsetneq_f H_2 \sqsupsetneq_f \ldots \sqsupsetneq_f H_n$ has length $n \le N$, where $N = [G:T]$.

\vspace{2mm} \noindent
(ii)$\implies$(iii) is obvious.

\vspace{2mm} \noindent
(iii)$\implies$(i) Assume for contradiction that $f$ is not periodic. It suffices to show that there is no $\sqsubsetneq_f$-minimal element in $\mathcal{F}$. Take any $H \in \mathcal{F}$. Since $f$ is not periodic, we have that $\chi(f, H) \neq G$, so we can find $a, b \in G \setminus \chi(f, H)$ such that $f(a) = 0$ and $f(b) = 1$. By pf-continuity of $f$ there are subgroups $A, B \in \mathcal{F}$ satisfying $f \restriction aA \equiv 0$ and $f \restriction bB \equiv 1$. It follows that $H_0 := H \cap A \cap B \sqsubsetneq_f H$, as required.
\end{proof}

\begin{proposition} \label{prop:tree-nper} Assume $f : G \to \{ 0, 1 \}$ is founded on a valued tree of cosets $(T, \mathcal{H}, d, v)$ which is linear, irreducible and infinite. Then $f$ is not periodic.
\end{proposition}

\begin{proof} Recall that linearity means that $H_{\theta} = H_{\eta}$ whenever $|\theta| = |\eta|$. Take an infinite branch $t \in \omega^{\omega}$ of $T$. For each $n \in \NN$ we have that $t \restriction n$ is not a leaf in $T$, so by irreducibility it is not homogeneous. Hence $f$ is not constant on $d_{t \restriction n} H_{t \restriction n}$ and $\chi(f, H_{t \restriction n}) \cap d_{t \restriction n} H_{t \restriction n} = \varnothing$. Take any $a \in d_{t \restriction n} H_{t \restriction n}$ and find a leaf $\eta$ extending $t \restriction n$ such that $a \in d_{\eta} H_{\eta}$. Then $f$ is constant on $d_{\eta} H_{\eta}$, so $d_{\eta} H_{\eta} \subseteq \chi(f, H_{\eta}) \setminus \chi(f, H_{t \restriction n})$ and therefore $H_{t \restriction |\eta|} = H_{\eta} \sqsubsetneq_f H_{t \restriction n}$. Thus for any $n \in \NN$ there is $m \ge n$ such that $H_{t \restriction m} \sqsubsetneq_f H_{t \restriction n}$, so by Lemma~\ref{lem:per-wf}, $f$ is not periodic.
\end{proof}

\begin{figure}
\centering
\begin{tikzpicture}[scale=1, coset/.style = {draw, rectangle}, value/.style= { draw, circle }]
\def\r{0.4};

\node at (0, 1) {$\alpha = 1 + 3 + 9 + 27 + \ldots = -\frac{1}{2} \in \ZZ_3 \setminus \ZZ$};

\node[coset] (1) at (0, 0) {$\ZZ$};

\node[coset] (2) at (-2, -0.5) {$3\ZZ$};
\node[coset] (3) at (2, -0.5) {$3\ZZ+2$};
\draw (1) -- (2);
\draw (1) -- (3);

\node[coset] (4) at (-3.5, -1.5) {$6\ZZ$};
\node[value, below=0 of 4] {$0$};
\draw (2) -- (4);

\node[coset] (5) at (-2, -1.5) {$6\ZZ+3$};
\node[value, below=0 of 5] {$1$};
\draw (2) -- (5);

\node[coset] (6) at (2, -1.5) {$6\ZZ+2$};
\node[value, below=0 of 6] {$0$};
\draw (3) -- (6);

\node[coset] (7) at (4, -1.5) {$6\ZZ+5$};
\node[value, below=0 of 7] {$1$};
\draw (3) -- (7);

\node[coset] (8) at (0, -2.5) {$3\ZZ+1$};
\draw (1) -- (8);

\node[coset] (9) at (-2, -3.5) {$9\ZZ+1$};
\draw (8) -- (9);

\node[coset] (10) at (2, -3.5) {$9\ZZ+7$};
\draw (8) -- (10);

\node[coset] (11) at (-4, -4.5) {$18\ZZ+10$};
\node[value, below=0 of 11] {$0$};
\draw (9) -- (11);

\node[coset] (12) at (-2, -4.5) {$18\ZZ+1$};
\node[value, below=0 of 12] {$1$};
\draw (9) -- (12);

\node[coset] (13) at (2, -4.5) {$18\ZZ+16$};
\node[value, below=0 of 13] {$0$};
\draw (10) -- (13);

\node[coset] (14) at (4, -4.5) {$18\ZZ+7$};
\node[value, below=0 of 14] {$1$};
\draw (10) -- (14);

\node[coset] (15) at (0, -5.5) {$9\ZZ+4$};
\draw (8) -- (15);

\draw (15) -- (0, -6);
\draw (15) -- (-1.2, -5.8);
\draw (15) -- (1.2, -5.8);

\node at (0, -5.8) [below] {$\substack{\vdots \\[0.5ex] \displaystyle \downarrow \\[0.5ex] \displaystyle \alpha}$};

\end{tikzpicture}
\caption{A periodic function founded on a non-linear infinite irreducible tree.}
\label{fig:irr-per}
\end{figure}
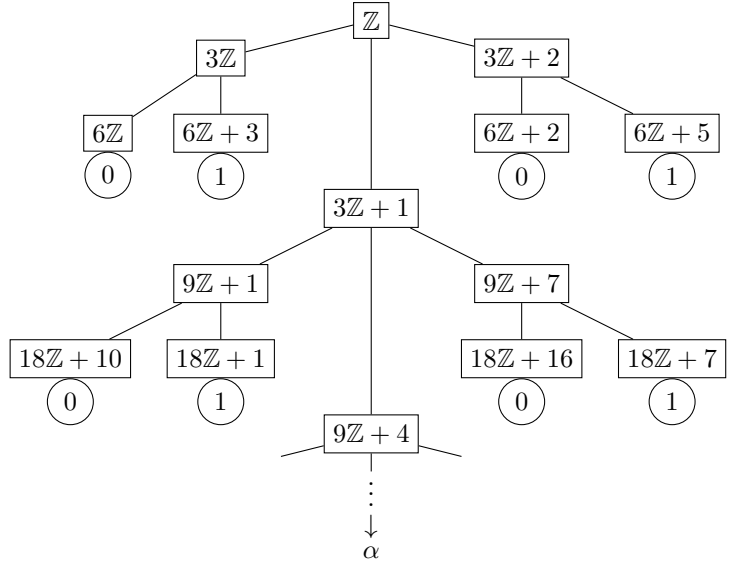

The linearity assumption above is essential. Indeed, consider the function $f : \ZZ \to \{ 0, 1 \}$ founded on the tree shown in Figure~\ref{fig:irr-per}. Then $f$ is periodic as $f \restriction 2\ZZ \equiv 0$ and $f \restriction (2\ZZ+1) \equiv 1$, even though the tree is infinite and irreducible.

This concludes our description of a class of strongly generic, non-periodic subsets of an arbitrary group $G$. However, this is not a complete characterization even in $(\ZZ, +)$, as there are strongly generic subsets which are not even pf-clopen.

\begin{example} \label{ex:val} Let $v_2 : \ZZ \to \NN \cup \{ \infty \}$ denote the $2$-adic valuation, i.e. $v_2(0) = \infty$ and $v_2(k)$ is the highest $n \in \NN$ such that $2^n \mid k$ for $k \neq 0$. Define $A_{\varepsilon} \subseteq \ZZ$, $\varepsilon \in \{ 0, 1 \}$, as follows:
\[
\chi_{A_{\varepsilon}}(k) = \begin{cases} \varepsilon & \text{if } \ k = 0, \\ v_2(k) \bmod 2 & \text{if } \ k \neq 0. \end{cases}
\]
Clearly neither $A_{\varepsilon}$ is pf-clopen, because $\chi_{A_0}$ and $\chi_{A_1}$ are not pf-continuous at $k=0$. However, both $A_0$ and $A_1$ are strongly generic. To see this for e.g. $A_0$, by Proposition~\ref{prop:sgen-lper} it suffices to check it is locally periodic. Fix any finite $U \subseteq \ZZ$ and take an even natural number $N > v_2(u)$ for all $u \in U \setminus \{ 0 \}$. We claim that 
\[
2^N + 2^{N+1} \ZZ = \{ (2k+1) \cdot 2^N : k \in \ZZ \} \subseteq \Per_U(\chi_{A_0}).
\]
Indeed, for any $k \in \ZZ$ and $u \in U \setminus \{ 0 \}$ we have that $v_2 \big( (2k+1) \cdot 2^N \big) = N > v_2(u)$, so $v_2 \big( u + (2k+1) \cdot 2^N \big) = v_2(u)$, hence $\chi_{A_0} \big( u + (2k+1) \cdot 2^N \big) = \chi_{A_0}(u)$. Also $\chi_{A_0} \big( (2k+1) \cdot 2^N \big) = N \bmod 2 = 0 = \chi_{A_0}(0)$. Hence $\Per_U(\chi_{A_0})$ contains the coset $2^N + 2^{N+1} \ZZ$ and therefore is generic.
\end{example}

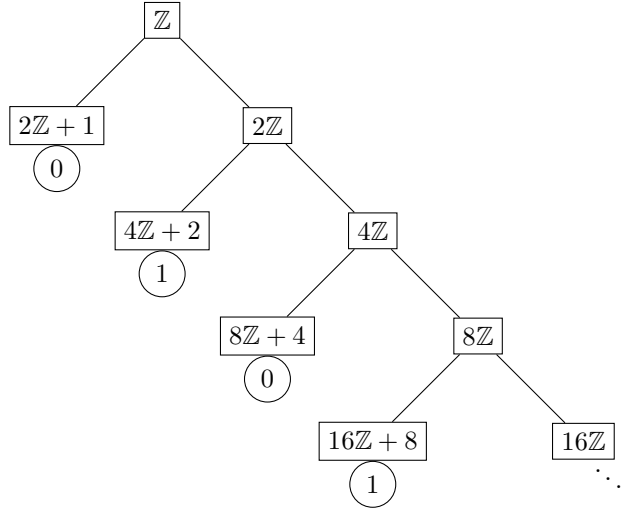
\begin{figure}
\centering
\begin{tikzpicture}[scale=1.4, coset/.style = {draw, rectangle}, value/.style= { draw, circle } ]

\node[coset] (1) at (0, 0) {$\ZZ$};

\node[coset] (2) at (-1, -1) {$2\ZZ+1$};
\node[value, below=0 of 2] {$0$};
\draw (1) -- (2);

\node[coset] (3) at (1, -1) {$2\ZZ$};
\draw (1) -- (3);

\node[coset] (4) at (0, -2) {$4\ZZ+2$};
\node[value, below=0 of 4] {$1$};
\draw (3) -- (4);

\node[coset] (5) at (2, -2) {$4\ZZ$};
\draw (3) -- (5);

\node[coset] (6) at (1, -3) {$8\ZZ+4$};
\node[value, below=0 of 6] {$0$};
\draw (5) -- (6);

\node[coset] (7) at (3, -3) {$8\ZZ$};
\draw (5) -- (7);

\node[coset] (8) at (2, -4) {$16\ZZ+8$};
\node[value, below=0 of 8] {$1$};
\draw (7) -- (8);

\node[coset] (9) at (4, -4) {$16\ZZ$};
\draw (7) -- (9);

\node [below right] at (9) {$\ddots$};
\end{tikzpicture}
\caption{A tree that leaves the value at $k=0$ undefined. }
\label{fig:val-tree}
\end{figure}

Note that both sets considered in the example are ``almost tree-founded'', meaning they are defined in the same way as before by the tree shown in Figure~\ref{fig:val-tree}, except for the value at $k=0$ which must be assigned manually as either $0$ or $1$. We can try to generalize and ask whether the property of being strongly generic will be preserved if instead of one point we allow such arbitrary assignment on some small subset of $G$, e.g. a set of elements corresponding to a nowhere dense set of infinite branches. We will demonstrate that the answer is quite simply negative after we formalize the necessary notions. However, the sets with the described property will be of interest in the next section.

\begin{definition} Let $(T, \mathcal{H}, d, v)$ be as in Definition~\ref{def:tree} (ii), except we no longer assume that each $g \in G$ belongs to some coset corresponding to a leaf of $T$. Thus the set $G_L = \bigcup_{\eta \in T_L} d_{\eta} H_{\eta}$ is not necessarily equal to $G$, but we still define $f_T : G_L \to T_L$ as before. Assume that the set of infinite branches of $T$ is nowhere dense, i.e. every $\eta \in T$ extends to a leaf.
\begin{enumerate}[label=(\roman{*})]
\item A function $f : G \to \{ 0, 1 \}$ is \emph{almost founded} on $(T, \mathcal{H}, d, v)$ if $f \restriction G_L = v \circ f_T$.
\item A function $f : G \to \{ 0, 1 \}$ is \emph{almost tree-founded} if it is almost founded on some $(T, \mathcal{H}, d, v)$ as above.
\item A subset $A \subseteq G$ is \emph{almost tree-founded} if $\chi_A$ is almost tree-founded.
\end{enumerate}
\end{definition}

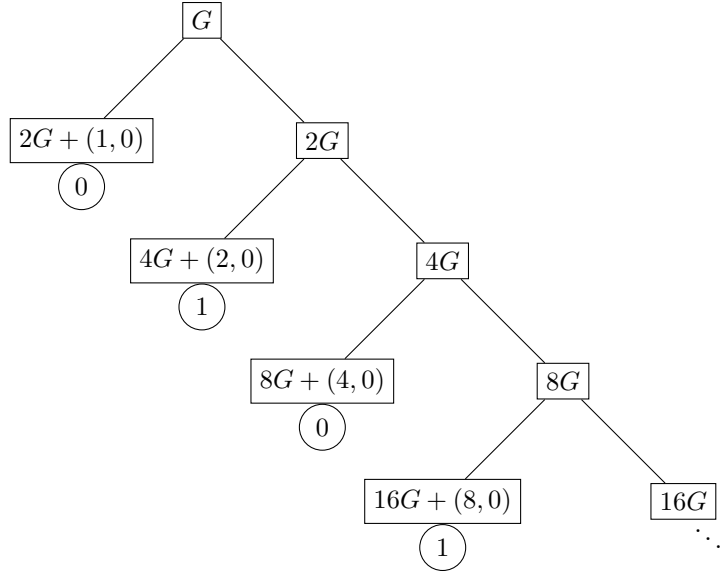
\begin{figure} 
\centering
\begin{tikzpicture}[scale=1.6, coset/.style = {draw, rectangle}, value/.style= { draw, circle } ]

\node[coset] (1) at (0, 0) {$G$};

\node[coset] (2) at (-1, -1) {$2G+(1, 0)$};
\node[value, below=0 of 2] {$0$};
\draw (1) -- (2);

\node[coset] (3) at (1, -1) {$2G$};
\draw (1) -- (3);

\node[coset] (4) at (0, -2) {$4G + (2, 0)$};
\node[value, below=0 of 4] {$1$};
\draw (3) -- (4);

\node[coset] (5) at (2, -2) {$4G$};
\draw (3) -- (5);

\node[coset] (6) at (1, -3) {$8G + (4, 0)$};
\node[value, below=0 of 6] {$0$};
\draw (5) -- (6);

\node[coset] (7) at (3, -3) {$8G$};
\draw (5) -- (7);

\node[coset] (8) at (2, -4) {$16G + (8, 0)$};
\node[value, below=0 of 8] {$1$};
\draw (7) -- (8);

\node[coset] (9) at (4, -4) {$16G$};
\draw (7) -- (9);

\node [below right] at (9) {$\ddots$};
\end{tikzpicture}
\caption[]{An almost tree-founded set that need not be strongly generic.}
\label{fig:padxq}
\end{figure}

\begin{example} Assume $A \subseteq G = \ZZ \times \QQ$ is a subset such that $\chi_A$ is almost founded on the tree in Figure~\ref{fig:padxq}. Then $A$ is strongly generic if and only if $\chi_A$ is constant on $\{ 0 \} \times \QQ$. Indeed: if $\chi_A$ is constant on $\{ 0 \} \times \QQ$, then the argument is the same as in Example~\ref{ex:val} with the additional remark that $\{ 0 \} \times \QQ \subseteq \Per(\chi_A)$. On the other hand, if $\chi_A$ is not constant on $\{ 0 \} \times \QQ$, take $a, b \in \QQ$ such that $(0, a) \in A$ and $(0, b) \notin A$. Then the set $A \setminus \big( (0, a-b) + A \big)$ is non-empty and contained in $\{ 0 \} \times \QQ$, hence not generic.
\end{example}

The following variation on Example~\ref{ex:val} will be used later:
\begin{example} \label{ex:gval} Assume $G$ is an infinite profinite group. Take a sequence
\[
G = F_0 \gne F_1 \gne F_2 \gne \ldots
\]
of clopen subgroups of $G$, thus of finite index, and let $F_{\infty} := \bigcap_{n=0}^{\infty} F_n$. We define $v : G \to \NN \cup \{ \infty \}$ so that $v(x)$ is the highest $n \in \NN \cup \{ \infty \}$ such that $x \in F_n$. Similarly as for valuations, we have that $v(xy) \ge \min \{ v(x), v(y) \}$ and these values are equal if $v(x) \neq v(y)$. 

Consider $A = \bigcup_{n=1}^{\infty} (F_{2n-1} \setminus F_{2n}) \subseteq G$, so that
\[
\chi_{A}(x) = \begin{cases} 0 & \text{if } \ x \in F_{\infty}, \\ v(x) \bmod 2 & \text{if } \ x \in G \setminus F_{\infty}. \end{cases}
\]
We will prove that $A$ is strongly generic. By Proposition~\ref{prop:sgen-lper}, it suffices to check that $A$ is locally periodic. Fix a finite $U \subseteq G$ and take an even natural number $N > v(u)$ for all $u \in U \setminus F_{\infty}$. We claim that $F_N \setminus F_{N+1} \subseteq \Per_U(\chi_A)$. Indeed, fix $t \in F_N \setminus F_{N+1}$ and $u \in U$. If $u \in U \setminus F_{\infty}$, then $v(t) = N > v(u)$, so $v(ut) = v(u)$ and so $\chi_A(ut) = \chi_A(u)$. Furthermore, if $u \in F_{\infty}$, then $\chi_A(ut) = N \bmod 2 = 0 = \chi_A(u)$. Hence $\Per_U(\chi_A)$ contains $F_N \setminus F_{N+1}$, so it is generic.

We also prove that $A$ is not periodic. In fact, we claim that $\Per(\chi_A) = F_{\infty}$, which clearly has infinite index. Note that
\[
\Per(\chi_A) = \{ t \in G : At^{-1} = A \}
\]
and take any $t \in F_{\infty}$. Then $F_n t^{-1} = F_n$ for any $n \in \NN$, thus
\[
At^{-1} = \bigcup_{n=1}^{\infty} (F_{2n-1} \setminus F_{2n}) t^{-1} = \bigcup_{n=1}^{\infty} (F_{2n-1} \setminus F_{2n}) = A
\]
and so $t \in \Per(\chi_A)$. Now take any $t \in G \setminus F_{\infty}$ and let $n = v(t) \in \NN$ so that $t \in F_n \setminus F_{n+1}$. If $n$ is even, then 
\begin{align*}
A & = (F_1 \setminus F_2) \cup \ldots \cup (F_{n-1} \setminus F_n) \cup (F_{n+1}\phantom{t^{-1}} \setminus F_{n+2}\phantom{t^{-1}}) \cup \ldots \\
At^{-1} &= (F_1 \setminus F_2) \cup \ldots \cup (F_{n-1} \setminus F_n) \cup (F_{n+1}t^{-1} \setminus F_{n+2}t^{-1}) \cup \ldots
\end{align*}
It follows that $A \cap F_n \subseteq F_{n+1}$ and $At^{-1} \cap F_n \subseteq F_{n+1}t^{-1}$. In particular, the sets on the left side of the inclusions are disjoint. But they are also non-empty, hence $A \neq At^{-1}$. If $n$ is odd, we show in a similar manner that $A \neq At^{-1}$. In either case $t \notin \Per(\chi_A)$, as desired.
\end{example}

\subsection{Uniformly strongly generic sets}

In this subsection we investigate a particular strengthening of strong genericity. Assume $G$ is a group definable in a first-order structure $M$ and let $M \preccurlyeq N$. Recall that one of the central ideas of the thesis is to relate the Ellis groups of $S_{\ext, G}(M)$ and $S_{\ext, G}(N)$ using image algebras. This could be particularly effective in the following case: assume some image algebra consists of definable sets. Then it can be naturally lifted to $N$. If the lifted sets still generate a generic $G^N$-algebra, we can extend it to an image algebra and possibly conclude some relation between the corresponding Ellis groups.

However, given a strongly generic definable subset $A \subseteq G$, it may not be the case that $A^N \subseteq G^N$ is also strongly generic. For instance, assume $N$ is $\aleph_0$-saturated. Then $A^N$ is strongly generic if and only if $A$ satisfies the following property:
\begin{equation} \label{eq:usg}
\parbox{11cm}{\emph{For each Boolean term $\tau(X_1, \ldots, X_n)$ there is $m_{\tau} \in \NN$ such that for any $g_1, \ldots, g_n \in G$ the set $\tau(g_1A, \ldots, g_nA)$ is either empty or $m_{\tau}$-generic,}}
\end{equation}
where a set $B \subseteq G$ is said to be $m$-generic if $G = h_1 B \cup \ldots \cup h_m B$ for some $h_1, \ldots, h_m \in G$. In fact, (\ref{eq:usg}) holds if and only if $A^N$ is strongly generic in $G^N$ for every $N \succcurlyeq M$. For this reason we devote some attention to investigating the property.

\begin{definition} Assume $G$ is an arbitrary group and $\cA$ is a $G$-algebra.
\begin{enumerate}[label=(\alph{*})]
\item When $m \in \NN$, an element $B \in \cA$ is said to be \emph{$m$-generic} provided that $\one = h_1 B \vee \ldots \vee h_m B$ for some $h_1, \ldots, h_m \in G$.
\item An element $A \in \cA$ is called \emph{uniformly strongly generic}, abbreviated USG, if it satisfies the property (\ref{eq:usg}).
\end{enumerate}
\end{definition}

Although the generality of the definition will be useful at times, we mainly focus on uniformly strongly generic subsets of $G$, that is, sets $A \subseteq G$ that are USG as elements of the $G$-algebra $\cP(G)$. The next two remarks are straightforward to prove:

\begin{remark} An element $A \in \cA$ is uniformly strongly generic if and only if it satisfies (\ref{eq:usg}) restricted to terms $\tau$ of the form
\[
\tau(X_1, \ldots, X_I, Y_1, \ldots, Y_J) = X_1 \wedge \ldots X_I \wedge (Y_1)^c \wedge \ldots \wedge (Y_J)^c.
\]
\end{remark}

\begin{remark} \label{rem:hom-usg} Assume $\varphi: \cA \to \cB$ is a homomorphism of $G$-algebras and $A \in \cA$ is USG. Then $\varphi(A) \in \cB$ is USG.
\end{remark}

Propositions~\ref{prop:ap-char} and \ref{prop:sgen-lper} and the corresponding notions have a natural counterpart in the uniform setting.

\begin{definition} \label{def:uni} Assume $f \in 2^G$.
\begin{enumerate}[label=(\alph{*})]
\item $f$ is \emph{uniformly self-replicating} if it satisfies the condition
\[
(\forall \underset{\text{finite}}{U \subseteq G})(\exists \underset{\text{finite}}{V \subseteq G})(\forall g \in G)(\exists h \in G) \, h \odot (f \restriction U) \subseteq f \restriction Vg,
\]
where additionally $|V|$ depends only on $|U|$.
\item $f$ is \emph{uniformly almost periodic} in $2^G$ if for each finite $U \subseteq G$ there is a finite $W \subseteq G$ such that $\cl(G \odot f) \subseteq W^{-1} \odot [f \restriction U]$, and $|W|$ depends only on $|U|$.
\item $f$ is \emph{uniformly locally periodic} if for each finite $U \subseteq G$ there is $m \in \NN$ such that $\Per_U(f)$ is $m$-generic, where $m$ depends only on $|U|$.
\end{enumerate}
\end{definition}

\begin{remark} \label{rem:usg-char} For any $A \subseteq G$, the following are equivalent:
\begin{enumerate}[label=(\roman{*})]
\item $A$ is uniformly strongly generic;
\item $\chi_A$ is uniformly self-replicating;
\item $\chi_A$ is uniformly almost periodic in $2^G$;
\item $\chi_A$ is uniformly locally periodic.
\end{enumerate}
\end{remark}

\begin{proof} Essentially repeat the reasoning from Propositions~\ref{prop:ap-char} and \ref{prop:sgen-lper}.
\end{proof}

If $A \subseteq G$, we have the following implications:
\[
A \text{ is periodic} \implies A \text{ is USG} \implies A \text{ is strongly generic}.
\]
None of these can be reversed without extra assumptions. We later show (see Corollary~\ref{cor:tf-sg-nusg} and Theorem~\ref{thm:comp-nusg}) that in many groups there are strongly generic sets that are not USG. Now let us show examples of uniformly strongly generic sets that are not periodic.

\begin{example} \label{ex:usg-nper1} For any finite group $\cG \neq \{ e \}$ let $G = \cG^{\omega} \rtimes_{\varphi} \Sym(\omega)$, where the underlying action of the semidirect product is $\varphi_{\sigma}(s) = s \circ \sigma^{-1}$ for $\sigma \in \Sym(\omega)$, $s \in \cG^{\omega}$. Consider the subset $A \subseteq G$ defined as
\[
A = \{ s \in \cG^{\omega} : s(0) = e \} \times \Sym(\omega).
\]
We will show that $A$ is uniformly strongly generic but not periodic. Given $n \in \omega$, $g \in \cG$, let
\[
A_n^g = \{ s \in \cG^{\omega} : s(n) = g \} \times \Sym(\omega),
\]
so that $A = A_0^e$. Then for any $\< s, \sigma \> \in G$ we have that
\begin{align*}
\< s, \sigma \> A & = \{ \< s, \sigma \> \< t, \tau \> : t(0) = e \} \\
& = \{ \< s \varphi_{\sigma}(t), \sigma \tau \> : t(0) = e \} \\
& = \{ \< s \varphi_{\sigma}(t), \sigma \tau \> : ( s \varphi_{\sigma}(t) ) \big( \sigma(0) \big) = s(\sigma(0)) \} \\
& = A_{\sigma(0)}^{s(\sigma(0))}.
\end{align*}
It follows that the $G$-algebra $\cA \le \cP(G)$ generated by $A$ is 
\[
\cA = \{ C \times \Sym(\omega) : C \subseteq \cG^{\omega} \text{ is clopen} \},
\]
where the topology on $\cG^{\omega}$ is that of a product of finite discrete groups. In particular, $\cA$ is infinite, so $A$ is not periodic. On the other hand, consider any non-empty set $B \in \cA$ of the form
\[
B = a_1 A \cap \ldots \cap a_I A \cap b_1 (G \setminus A) \cap \ldots \cap b_J (G \setminus A),
\]
where $a_i, b_j \in G$, and write it as
\[
B = A_{n_1}^{g_1} \cap \ldots \cap A_{n_I}^{g_I} \cap (G \setminus A_{m_1}^{h_1}) \cap \ldots \cap (G \setminus A_{m_J}^{h_J}).
\]
Then $B$ is $|\cG|^I \cdot 2^J$-generic. Therefore $A$ is uniformly strongly generic.

We show in Appendix~\ref{app:semi} that the Ellis group of $S(\cA)$ is isomorphic to $\cG$. 
\end{example}

\begin{example} \label{ex:usg-nper2}
Let $G = \mathbb{F}_{x, y}$ denote the free group on generators $x$, $y$. Write $S = \{ x, y, x^{-1}, y^{-1} \}$ and define $S^*$ to be the set of all words over $S$, i.e. finite sequences of elements of $S$. For any non-empty irreducible word $w \in S^*$ let $T_w \subseteq G$ denote the set of elements represented by irreducible words beginning with $w$. Consider the subset $A := T_x \subseteq G$. Clearly it is not periodic because 
\[
\{ x^n A : n \in \omega \} = \{ T_{x^{n+1}} : n \in \omega \}
\]
is infinite. We will check that $A$ is uniformly strongly generic.

We claim that for any $a \in A$, the set $a^{-1} A$ contains $T_{y^n} \cup T_{y^{-n}}$ for sufficiently large $n \in \omega$. Indeed, let $xuy^k$ be the irreducible word representing $a$, where $k \in \ZZ$ and $u \in S^*$ does not end with $y$ nor $y^{-1}$. Now let $n > |k|$ and take any $g \in T_{y^n} \cup T_{y^{-n}}$. If $g \in T_{y^n}$, we can write it as an irreducible word $y^n v$, where $v \in S^*$ does not begin with $y^{-1}$. Then $ag = x u y^{n+k} v$ is an irreducible representation, hence $ag \in A$. If $g \in T_{y^{-n}}$, we show in a similar way that $ag \in A$. Hence $T_{y^n} \cup T_{y^{-n}} \subseteq a^{-1} A$ whenever $n > |k|$.

Also for any $b \notin A$, the set $b^{-1} (G \setminus A)$ contains $T_{y^n} \cup T_{y^{-n}}$ for sufficiently large $n \in \omega$. To see this, let $u \in S^*$ be an irreducible word representing $b$. Then $u$ does not begin with $x$, hence $u^{-1} x$ is irreducible and $b^{-1} A = T_{u^{-1} x}$. Clearly $T_{y^n} \cup T_{y^{-n}}$ is disjoint from this set when $n \in \omega$ is sufficiently large.

It suffices to show that any non-empty set $B$ in the $G$-algebra generated by $A$ is $2$-generic. Take any such $B$ and assume without loss of generality that $e \in B$ and $B$ is of the form
\[
B = a_1^{-1} A \cap \ldots \cap a_I^{-1} A \cap b_1^{-1} (G \setminus A) \cap \ldots \cap b_J^{-1} (G \setminus A),
\]
so that $a_i \in A$ and $b_j \notin A$. It then follows from the previous observations that $T_{y^n} \cup T_{y^{-n}} \subseteq B$ for some $n \in \omega$. But the last set is $2$-generic because 
\[
G = T_{y^n} \cup y^{2n-1} T_{y^{-n}},
\]
thus also $B$ is $2$-generic. Therefore $A$ is uniformly strongly generic.
\end{example}

For the sake of the next example we introduce some notation. When a group $G$ acts on a set $X$ and $x \in X$, let
\[
d_x A = \{ g \in G : gx \in A \}.
\]
The set $\cP(X)$ carries a natural structure of a $G$-algebra such that the map $d_x : \cP(X) \to \cP(G)$ is a $G$-algebra homomorphism.

\begin{remark} \label{rem:tran-usg} Assume $G$ acts transitively on $X$. Assume $A \subseteq X$ is uniformly strongly generic and the set $\{ gA : g \in G \}$ is infinite. Then for any $x \in X$ the set $d_x A$ is uniformly strongly generic and not periodic.
\end{remark}

\begin{proof} By Remark~\ref{rem:hom-usg}, $d_x A$ is USG. Furthermore, because of the transitivity of the action, $d_x$ is a monomorphism, so the family $\{ g \cdot d_x A : g \in G \}$ is infinite. Hence $d_x A$ is not periodic.
\end{proof}

\begin{example} \label{ex:usg-nper3}
Let $G = \text{PSL}(2, \RR)$ and consider the usual action of $G$ on the projective line $\mathbf{P}(\RR) = \RR \cup \{ \infty \}$. Let $A = [0, 1) \subseteq \mathbf{P}(\RR)$ and
\[
B = d_0 A = \{ M \in G : M \cdot 0 \in [0, 1) \}.
\]
We will check that $B$ is uniformly strongly generic and not periodic. Using Remark~\ref{rem:tran-usg}, it suffices to show that $A$ is uniformly strongly generic and has infinite orbit. Given distinct $y, z \in \mathbf{P}(\RR)$, let $[y, z)$ denote the usual interval if $y < z$ and $\mathbf{P}(\RR) \setminus [z, y)$ otherwise, where $\RR < \infty$. Note that the action of $G$ on $\mathbf{P}(\RR)$ is $2$-transitive, thus
\[
\{ M \cdot [0, 1) : M \in G \} = \{ [y, z) : y, z \in \mathbf{P}(\RR), y \neq z \}.
\]
In particular, $A = [0, 1)$ has infinite orbit. 

Furthermore, take any non-empty set $C$ of the form 
\[
C = a_1 A \cap \ldots a_I A \cap b_1 (X \setminus A) \cap \ldots \cap b_J (X \setminus A),
\]
where $a_1, \ldots, a_I, b_1, \ldots, b_J \in G$. Then $C$ contains an interval of the form $[y, z)$ for some $y < z$. We can find $M \in G$ such that $M \cdot [y, z) = [z, y)$, so $C$ is $2$-generic. Therefore $A$ is uniformly strongly generic.
\end{example}

The last example is based on \cite[Remark 5.2 (iv)]{HPP06} where $G$ was an example of a group definable in an o-minimal structure that is not definably amenable. The lack of definable amenability is not a coincidence, as made clear by the next theorem due to Newelski.

\begin{theorem} \label{thm:nip-usg-per} Assume $A \subseteq G$ is uniformly strongly generic and the formula $\varphi(x; y) \equiv x \in yA$ does not have the independence property. Moreover, assume that the structure $(G, \cdot, A)$ is definably amenable. Then $A$ is periodic.
\end{theorem}

\begin{proof}
Assume for contradiction that $A$ is not periodic and take a translation invariant Keisler measure $\mu$ on $G$. By Theorem~\ref{thm:stab-sg-per}, $\varphi(x; y)$ is unstable, so by \cite[Theorem 4.7 (2)]{She90}, for some $n \in \NN$ and $\eta \in \{ 0, 1 \}^n$ the formula
\[
\psi(x; y_1, \ldots, y_n) := \bigwedge_{i=1}^n \varphi(x; y_i)^{\eta_i}
\]
has the strict order property. Thus for any $K < \omega$ we can find $\bar{b_0}, \ldots, \bar{b_K} \subseteq G$ such that 
\[
\psi(G; \bar{b}_0) \supsetneq \psi(G; \bar{b}_1) \supsetneq \ldots \supsetneq \psi(G; \bar{b}_K).
\]
Since $A$ is USG, we can find $m \in \NN$ such that $\psi(G; \bar{b}) \setminus \psi(G; \bar{c})$ is either empty or $m$-generic for any $\bar{b}, \bar{c} \subseteq G$. It follows that $\mu \big( \psi(G; \bar{b}_k) \setminus \psi(G; \bar{b}_{k+1}) \big) \ge \frac{1}{m}$ for $k < K$, which is a contradiction when $K > m$. 
\end{proof}

We have shown three examples of a group with a subset that is uniformly strongly generic but not periodic. However, given a concrete group $G$, it is not always easy to verify whether it has a subset with this property. In particular, we were unable to answer the following question:

\begin{question} \label{q:usg} Does the group $(\ZZ, +)$ have a uniformly strongly generic subset which is not periodic?
\end{question}

\noindent
Note that any such subset would have to be quite complicated; for instance, when $A \subseteq \ZZ$ is such a set, by the last theorem the formula $\varphi(x; y) \equiv x \in y+A$ must have the independence property. We share some ideas related to the problem in Appendix~\ref{app:usg}.

\section{The Ellis group across models}

Consider a group $G$ definable in a first-order structure $M$. We wish to investigate the relation between Ellis groups computed in $M$ and its extension $M^*$ in various scenarios. Although the ultimate goal is to relate the Ellis groups of the ``full'' flows $S_{\ext, G}(M)$ and $S_{\ext, G}(M^*)$, the task seems very difficult in general. Therefore we focus on some variants of the problem which we find more tractable.

Our approach is based on the following well known observation, which resembles the idea of localization from stability theory. Consider any $G$-algebra $\cA \le \cP(G)$. By Theorem~\ref{thm:iso}, $E(S(\cA))$ is isomorphic to $S(\cA^d)$ as a $G$-flow and semigroup, so we may compute the Ellis group of $S(\cA)$ directly from $S(\cA^d)$. Consider any direct system $\< \cA_i : i \in I \>$ of $G$-subalgebras $\cA_i \le \cA$ with inclusions such that $\cA = \bigcup_{i \in I} \cA_i$, e.g. the system of all finitely generated $G$-subalgebras. Then $\cA^d$ is the union and the direct limit of the system $\< \cA_i^d : i \in I \>$. By Stone duality, the $G$-flow $S(\cA^d)$ is the projective limit of the inverse system $\< S(\cA_i^d) : i \in I \>$ with restrictions. By Remark~\ref{rem:ast-hom}, these restrictions are also semigroup epimorphisms with respect to $\ast$. 

Let $\pi_j : S(\cA^d) \to S(\cA_j^d)$ denote the restriction. Choose a minimal ideal $I \trianglelefteqslant_m S(\cA^d)$ with an idempotent $u \in J(I)$. Then $I_j := \pi_j[I] \trianglelefteqslant_m S(\cA_j^d)$ and $u_j := \pi_j(u) \in J(I_j)$ and $\pi_j[uI] = u_j I_j$ for each $j \in I$. From here it is not hard to show that the group $uI$ is isomorphic to the projective limit of the system $\< u_j I_j : j \in I \>$. 

It follows that the Ellis group of $S(\cA)$ can be retrieved from the Ellis groups of $S(\cA_j)$. Thus a good approach to understanding the ``full'' Ellis group may be to identify and study some tractable $G$-subalgebras of the target algebra together with their ``partial'' Ellis groups. In this section we focus on a particular algebra $\cSBP \le \cP(G)$, introduced below. This algebra is strictly related to the algebra of all externally definable subsets of $G$ at least in the case when $G$ is definable in a densely ordered o-minimal structure, as shown in Corollary~\ref{cor:ext-sbp}.

\begin{definition} Let $X$ be a topological space. We say that $A \subseteq X$ has the \emph{strong Baire property} (abbreviated SBP) if it satisfies any of the following equivalent conditions:
\begin{itemize}
\setlength\itemsep{0.2em}
\item $A = U \Delta M$ for some open $U \subseteq X$ and nowhere dense $M \subseteq X$;
\item $\bd(A)$ has empty interior;
\item $\int(A) \cup \int(X \setminus A)$ is dense.
\end{itemize}
\end{definition}

Throughout the section $G$ is a topological group (meaning that inversion and multiplication are continuous) satisfying various assumptions specified individually in subsections. The family of all subsets of $G$ having the strong Baire property is a $G$-algebra, which we denote as $\cSBP$. The plan is as follows: given an arbitrary $\cA \le \cSBP$, we first describe the Ellis group of $S(\cA)$. Next we set up $G$ as a group definable in a model $M$, take an elementary extension $M \preccurlyeq M^*$ and find a natural $G^*$-subalgebra $\cA^* \le \cSBP^*$ corresponding to $\cA$. Finally, we apply our description to relate the Ellis groups of $S(\cA)$ and $S(\cA^*)$.

\subsection{Groups with profinite topology}

In this subsection we consider an arbitrary group $G$ equipped with the profinite topology, which makes it a topological group. We first show that $\cSBP$ includes virtually all sets considered in the previous section, namely:

\begin{proposition} \label{prop:atf-sbp} If $A \subseteq G$ is almost tree-founded, then $A$ has SBP with respect to the profinite topology.
\end{proposition}

Although for an almost tree-founded $A \subseteq G$ the set $\int(A) \cup \int(A^c)$ clearly contains the ``tree-founded part'' of $A$, which is dense in the tree, it does not directly follow that it is dense in the profinite topology, so the proposition needs a proper proof. We first state a useful lemma:

\begin{lemma} \label{lem:ind} Assume $H$ is a group and $H_0, F_0 \le H$ are its subgroups. Then $[H_0 : H_0 \cap F_0] \le [H : F_0]$. If additionally $F_0$ has finite index and some coset of $H_0$ is disjoint from some coset of $F_0$, then the inequality is strict.
\end{lemma}

\begin{proof} Consider the map $\tau : H_0/(H_0 \cap F_0) \to H/F_0$ given by $\tau \big( x(H_0 \cap F_0) \big) = xF_0$. It is injective since if $x, y \in H_0$ and $xF_0 = yF_0$, then $y^{-1} x \in H_0 \cap F_0$ and $x(H_0 \cap F_0) = y(H_0 \cap F_0)$. It follows that $|H_0/(H_0 \cap F_0)| \le |H/F_0|$, which proves the first part. Under the additional assumption, $\tau$ is not onto: take $y, z \in H$ with $yH_0 \cap zF_0 = \varnothing$, equivalently $H_0 \cap y^{-1}zF_0 = \varnothing$. Then clearly $y^{-1}zF_0$ is not in the image of $\tau$. Consequently, the inequality is strict.
\end{proof}

\begin{proof}[Proof of Proposition~\ref{prop:atf-sbp}.] Take a tree $(T, \mathcal{H}, d, v)$ on which $\chi_A$ is almost founded. Assume for contradiction that there is a subgroup $F \le G$ of finite index with a coset $aF$ disjoint from $\int A \cup \int A^c$. Let $Q$ denote the set of all $\eta \in T$ satisfying $d_{\eta} H_{\eta} \cap aF \neq \varnothing$. In particular, $T_L \cap Q = \varnothing$.

We claim that there are $\eta \in Q$ and $i, j \in \omega$ such that $\eta^\frown i \notin Q$, $\eta^\frown j \in Q$ and 
\begin{equation} \tag{$\dagger$}
[H_{\eta} : H_{\eta} \cap F] \text{ is the smallest among all } \eta \in Q.
\end{equation}
Indeed, take any $\eta_0 \in Q$ satisfying $(\dagger)$ and note that by Lemma~\ref{lem:ind}, passing to any extension $\eta_0 \subseteq \eta \in Q$ preserves the minimality. By assumption, there is a leaf $\theta \in T_L$ extending $\eta_0$ and then $\theta \notin Q$. Pick $\eta \in T$ and $i \in \omega$ such that $\eta_0 \subseteq \eta \subseteq \eta^\frown i \subseteq \theta$ and $\eta \in Q$, $\eta^\frown i \notin Q$. On the other hand, $\eta^\frown j \in Q$ for some $j \in \omega$. Then $\eta, i, j$ satisfy the desired properties.

Since $\eta \in Q$, we can assume that $a \in d_{\eta} H_{\eta}$, so that $a^{-1} d_{\eta^\frown i} \in H_{\eta}$. It follows that the group $H := H_{\eta}$, its subgroup $H_0 := H_{\eta^\frown i} = H_{\eta^\frown j} \le H$ and the subgroup $F_0 := H \cap F \le H$ of finite index satisfy all assumptions of Lemma~\ref{lem:ind}, since $a^{-1} d_{\eta^\frown i} H_0 \cap F_0 = \varnothing$. Therefore $[H_{\eta^\frown j} : H_{\eta^\frown j} \cap F] < [H_{\eta} : H_{\eta} \cap F]$, contradicting $(\dagger)$.
\end{proof}

Given $\cA \le \cSBP$, computing the Ellis group of $S(\cA)$ is most convenient when $\cA$ is d-closed because of Theorem \ref{thm:iso}. This situation is not uncommon, since the next proposition shows that $\cSBP^d = \cSBP$. It follows that when $\cA \le \cSBP$ is arbitrary, the d-closure of $\cA$ is still contained in $\cSBP$.

\begin{proposition} \label{prop:pf-sbp-dcl} $\cSBP$ is a d-closed $G$-algebra.
\end{proposition}

\begin{proof} By Remark~\ref{rem:dlim} (iv), it suffices the check the following: if $A \in \cSBP$ and $B \subseteq G$ is in the pointwise closure of $\{ Ag^{-1} : g \in G \}$, then $B \in \cSBP$. Write $H \lef G$ when $H$ is a subgroup of $G$ of finite index. Assume for contradiction that $B \notin \cSBP$, so there is a coset $Fc$ of a subgroup $F \lef G$ in which both $B$ and $B^c$ are dense. Shifting $B$ if necessary, we can assume that $c = e$. 

Let $Fc_1, \ldots, Fc_n$ be a list of all right cosets of $F$. Since $A \in \cSBP$, for each $1 \le i \le n$ we can find a coset $H_i d_i \subseteq F c_i$ of some $H_i \lef G$ on which $\chi_A$ is constant. By taking the intersection, we can assume that all $H_i$ equal some $H \lef G$ and clearly $H \subseteq F$. Let $Hf_1, \ldots, Hf_m$ denote all right cosets of $H$ in $F$. By assumption, we can find $a_j, b_j \in Hf_j$ such that $a_j \in B, b_j \notin B$.

Since $B$ is in the pointwise closure of $\{ Ag^{-1} : g \in G \}$, we can find $g \in G$ such that $a_j \in Ag^{-1}, b_j \notin Ag^{-1}$ for $j = 1, \ldots, m$. Pick $i \in \{ 1, \ldots, n \}$ so that $g \in F c_i$. Then $d_i g^{-1} \in F c_i g^{-1} = F$, so we can find $j \in \{ 1, \ldots, n \}$ such that $d_i g^{-1} \in Hf_j$. It follows that $a_j g, b_j g \in H f_j g = H d_i$ and $a_j g \in A, b_j g \notin A$, contradicting the fact that $\chi_A$ is constant on $H d_i$.
\end{proof}

The same proof yields the following generalization of Proposition~\ref{prop:pf-sbp-dcl}:
\begin{remark} \label{rem:pf-sbp-dcl} Assume $\cN_0$ is a family of finite index normal subgroups of $G$ closed under finite intersection. Let $\cSBP(\cN_0)$ denote the $G$-algebra of subsets of $G$ having SBP with respect to the topology generated by cosets of subgroups from $\cN_0$. Then $\cSBP(\cN_0)$ is d-closed.
\end{remark}

Take any $G$-subalgebra $\cA \le \cSBP$. We are going to explicitly compute the Ellis group of $S(\cA)$. Let
\[
\cN := \{ N \le G : N \text{ is a normal subgroup of finite index} \}.
\]
The set $\cN_{\cA} = \cN \cap \cA^d$ with reverse inclusion is a directed set. Consider the projective system of groups $(G/N)_{N \in \cN_{\cA}}$, where for $N, N' \in \cN_{\cA}$, $N \subseteq N'$ we have the natural quotient map $\pi_{N, N'} : G/N \to G/N'$.

\begin{theorem} \label{thm:pf-ell} Assume $\cA \le \cSBP$. The Ellis group of $S(\cA)$ is isomorphic to the projective limit 
\[
\cG := \projlim_{N \in \cN_{\cA}} G/N.
\]
\end{theorem}

Let $\< \cN_{\cA} \>$ denote the $G$-algebra generated by $\cN_{\cA}$. In order to prove the theorem, we will need two lemmas:

\begin{lemma} \label{lem:per-dcl} Any $G$-algebra $\cB \le \cP(G)$ consisting of periodic sets is d-closed.
\end{lemma} 

\begin{proof} Given any $N \in \cB$ and $q \in \beta G$, we need to show that $d_q N \in \cB$. Using Corollary~\ref{cor:per-nbc}, we can assume that $N \in \cN$. Pick $a \in G$ such that $aN \in q$. We have that $Ng^{-1} = Na^{-1}$ for $g \in aN$. Thus $Ng^{-1} = Na^{-1}$ holds ultimately as $\hat{g} \to q$, so by Remark~\ref{rem:dlim} (iii),
\[
d_q N = \lim_{\hat{g} \to q} Ng^{-1} = Na^{-1} = a^{-1} N \in \cB. \qedhere
\] 
\end{proof}
In particular, $\< \cN_{\cA} \>$ is d-closed, so $S(\< \cN_{\cA} \>)$ equipped with $\ast$ is a semigroup.

\begin{lemma} \label{lem:proj-iso} The semigroup $\big( S( \< \cN_{\cA} \> ), \ast \big)$ is a group isomorphic to $\cG$.
\end{lemma}

\begin{proof} For $N \in \cN_{\cA}$, define $\varphi_N : S( \< \cN_{\cA} \> ) \to G/N$ so that $\varphi_N(q)$ is the unique coset in $G/N$ that belongs to $q$. Each $\varphi_N$ is a homomorphism by Lemma~\ref{lem:ast-prod} and is clearly continuous, where $G/N$ is treated as a discrete group. Applying the universal property of $\cG$ to the family $\< \varphi_N : N \in \cN_{\cA} \>$, we get a continuous homomorphism $\varphi : S( \< \cN_{\cA} \> ) \to \cG$ satisfying the usual diagram commutativity. It is easy to see that $\varphi$ is injective. It is also surjective, because it has dense image and $S( \< \cN_{\cA} \> )$ is compact. Hence $\varphi$ is an isomorphism of semigroups, thus in particular, $S( \< \cN_{\cA} \> )$ is a group.
\end{proof}

\begin{proof}[Proof of Theorem~\ref{thm:pf-ell}.]
Without loss of generality we can assume that $\cA$ is d-closed, since in general by Theorem~\ref{thm:iso}, we have the semigroup isomorphism $E(S(\cA)) \cong S(\cA^d)$ and $E(S(\cA^d)) \cong S((\cA^d)^d) = S(\cA^d)$, hence the Ellis group of $S(\cA)$ is the same as that of $S(\cA^d)$. Also $\cN_{\cA^d} = \cN_{\cA}$.

We first note that for every $A \in \cA$:
\[
A \text{ is nowhere dense} \iff A \text{ is not generic} \iff A \text{ has empty interior}.
\]
Indeed, the left-to-right implications are easy and the third one to close the circle follows from SBP. In particular, the non-generic sets in $\cA$ are precisely the nowhere dense sets, so they form a proper $G$-ideal in $\cA$. Therefore we can find $p \in S(\cA)$ such that every $A \in p$ is generic. Such $p$ is a generic point of the flow $S(\cA)$, so by Fact~\ref{fact:gen-min}, there is a unique minimal subflow $I \trianglelefteqslant_m S(\cA)$ consisting of all such ultrafilters $p$.

By Remark~\ref{rem:ast-hom}, the restriction $\pi : S(\cA) \to S( \< \cN_{\cA} \> )$ is a semigroup epimorphism. Therefore by Lemma~\ref{lem:proj-iso} and Corollary~\ref{cor:ell-epi}, it remains to show that $I \cap \pi^{-1}[ \{ \hat{e} \} ] \subseteq J(I)$. So take any $u \in I$ such that $\pi(u) = \hat{e}$, which means that $\cN_{\cA} \subseteq u$. Assume for contradiction that $u \ast u \neq u$ and take $A \in \cA$ such that $A \in u$ and $A \notin u \ast u$, i.e. $A \setminus d_u A \in u$. Since $u \in I$, the set $A \setminus d_u A$ has non-empty interior, so $aF \subseteq A \setminus d_u A$ for some $a \in G$ and $F \in \cN$. Letting $B := a^{-1} A \in \cA$ we get $F \subseteq B \setminus d_u B$.

If $F \in \cA$, the rest of the proof is easy: we have that $F \in u$, so $B \in u$. But $e \in F$ so $e \notin d_u B$, which means $B \notin u$, a contradiction. In general we will replace $B$ with some $C \in \cA$ and $F$ with some $N \in \cN_{\cA}$ with similar properties, so that this argument will work. 

\begin{claim*} There is $C \in \cA$ of the form $\tau(f_1B, \ldots, f_mB)$, where $f_1, \ldots, f_m \in F$ and $\tau(X_1, \ldots, X_m)$ is a nonconstant positive Boolean term, such that $C$ is a union of cosets of $F$.
\end{claim*}

\noindent
\emph{Proof of Claim.}
Suppose not and take $C = \tau(f_1B, \ldots, f_mB) \in \cA$, where $f_1, \ldots, f_m \in F$ and $\tau(X_1, \ldots, X_m)$ is as above, such that the number of cosets of $F$ on which $\chi_C$ is constant is maximal.  Take a coset $Fb$ on which $\chi_C$ is not constant. Since $C$ has SBP, there is a coset $Hc \subseteq Fb$ of a subgroup $H \le G$ of finite index on which $\chi_C$ is constant. We can assume that $b = c$ and so $Hb \subseteq Fb$, hence $H \subseteq F$. Let $g_1H, \ldots, g_nH$, where $g_1, \ldots, g_n \in F$, be a list of all left cosets of $H$ in $F$. 

Two cases are possible: either $\chi_C \restriction Hb \equiv 0$ or $\chi_C \restriction Hb \equiv 1$. In the first case consider the set 
\[
C' := \bigcap_{i=1}^n g_i C \in \cA.
\] 
Any coset of $F$ contained in or disjoint from $C$ will remain so with respect to $C'$. Moreover, $\chi_{C'} \restriction Fb \equiv 0$, which contradicts maximality in the definition of $C$. In the second case we consider
\[
C' := \bigcup_{i=1}^n g_i C \in \cA
\]
and arrive at a contradiction in a similar manner. This proves the claim. \qed

Take $C \in \cA$ as in the claim. Since $F \cap d_u B = \varnothing$, for any $f \in F$ we have that $fB \notin u$. Also $F \subseteq B$, so since $C$ is a nontrivial positive Boolean combination of sets $fB$, where $f \in F$, it follows that $F \subseteq C$ and $C \notin u$. By definition $C$ is periodic, hence by Corollary~\ref{cor:per-nbc}, it can be expressed as a union of cosets of some $N \in \cN_{\cA}$. Since $e \in C$, we have that $N \subseteq C$ and therefore $N \notin u$, contradicting the fact that $\cN_{\cA} \subseteq u$.
\end{proof}

\begin{example} Let $G = (\ZZ, +)$ and consider the set $A_0 \subseteq \ZZ$ from Example~\ref{ex:val}, defined by
\[
\chi_{A_0}(k) = \begin{cases} 0 & \text{if } \ k = 0, \\ v_2(k) \bmod 2 & \text{if } \ k \neq 0. \end{cases}
\]
Let $\cA = \< A_0 \> \le \cP(\ZZ)$. We will show that the Ellis group of $S(\cA)$ is $\ZZ_2$, the additive group of the $2$-adic integers.

First note that for any $n \in \NN$ we have $2^n \ZZ \in \cA$. Indeed, the set 
\[
B := A_0 \cap (4^n + A_0) = \bigcup_{k=0}^{n-1} 2^{2k+1} (2\ZZ+1) \in \cA
\]
is a union of an odd number of cosets of $4^n \ZZ$. Therefore it is periodic and $\Per(B) = 4^n \ZZ$, so $4^n \ZZ \in \cA$ by Corollary~\ref{cor:per-bc} and therefore $2^n \ZZ \in \cA$ as well. 

So $\{ 2^n \ZZ : n \in \NN \} \subseteq \cN \cap \cA^d$. It remains to show the reverse inclusion and apply Theorem~\ref{thm:pf-ell}, because $\projlim \ZZ / 2^n \ZZ = \ZZ_2$. Assume for contradiction that there is some subgroup $N \in \cN \cap \cA^d$ of the form $N = (2^n s) \ZZ$ where $s > 1$ is odd. Using Fact~\ref{fact:dcl-gen}, it is easy to see that $\cA^d$ is generated as a Boolean algebra by $\{ d_q (k+A_0) : q \in S(\beta \ZZ), k \in \ZZ \}$, so $N$ can be expressed as
\[
N = \tau(d_{q_1}(k_1+A_0), \ldots, d_{q_m}(k_m+A_0)),
\]
where $k_i \in \ZZ$, $q_i \in \beta \ZZ$ and $\tau(X_1, \ldots, X_m)$ is a Boolean term. On the one hand, for any $j \in \NN$ we have that $A_0 \symdif (2^j + A_0) \subseteq 2^j \ZZ$, so
\begin{align*}
N \symdif (2^j + N) & \subseteq \bigcup_{i=1}^m d_{q_i}(k_i+A_0) \symdif \big( 2^j + d_{q_i}(k_i+A_0) \big) \\[1ex]
& = \bigcup_{i=1}^m \big( k_i + d_{q_i}(A_0) \big) \symdif \big( k_i + d_{q_i} ( 2^j + A_0 ) \big) \\[1ex]
& = \bigcup_{i=1}^m k_i + d_{q_i} \big( A_0 \symdif (2^j + A_0) \big) \\[1ex]
& \subseteq \bigcup_{i=1}^m k_i + d_{q_i}( 2^j \ZZ ).
\end{align*}
Since each $d_{q_i}( 2^j \ZZ )$ is a coset of $2^j \ZZ$ (see Lemma~\ref{lem:per-dcl}), the set $N \symdif (2^j + N)$ is contained in a union of $m$ cosets of $2^j \ZZ$. On the other hand,
\[
N \symdif (2^j + N) = (2^n s) \ZZ \symdif (2^j + (2^n s) \ZZ) \supseteq (2^n s) \ZZ.
\]
When $j \in \NN$ is sufficiently big, it is a contradiction.
\end{example}

\begin{example} Assume $(G, \cdot)$ is a group definable in a stable structure $M$ and let $\cA = \Def_G(M)$ denote the $G$-algebra of subsets of $G$ definable in $M$. By stability, $\cA$ is equal to $\Def_{\ext, G}(M)$, the $G$-algebra of subsets of $G$ externally definable in $M$. It follows that $\cA$ is d-closed. It is known that in this case the Ellis group of $S(\cA) = S_G(M)$ is isomorphic to $G/G^0$ computed in a monster model. We will reprove this fact using the results from this subsection. Note that it does not follow immediately from Theorem~\ref{thm:pf-ell} since $\cA$ need not be contained in $\cSBP$.

Observe that $\cB = \< \cN_{\cA} \>$ is the family of all strongly generic sets in $\cA$. Indeed, one inclusion follows from the fact that periodic sets are strongly generic and the other one from Theorem~\ref{thm:stab-sg-per} and Corollary~\ref{cor:per-nbc}. It follows that $\cB$ is a (unique) maximal generic subalgebra of $\cA$, so by Corollary~\ref{cor:im-maxgen}, it is an image algebra. 

By Corollary~\ref{cor:im-ell}, the Ellis group of $S(\cA)$ is isomorphic to the Ellis group of $S(\cB)$. Since $\cN_{\cB} = \cN_{\cA}$ and $\cB \le \cSBP$, we get from Theorem~\ref{thm:pf-ell} that the Ellis group of $S(\cB)$ is isomorphic to $\projlim_{N \in \cN_{\cA}} G/N$, which is clearly isomorphic to $G/G^0$ computed in a monster model.
\end{example}

The same reasoning proves that more generally:
\begin{corollary} \label{cor:ell-pf} Assume $\cA \le \cP(G)$ is d-closed and all strongly generic sets in $\cA$ are periodic. Then the Ellis group of $S(\cA)$ is profinite and isomorphic to $\projlim_{N \in \cN_{\cA}} G/N$.
\end{corollary}

Now we consider two scenarios of making $G$ a group definable in a model $M$ and lifting a $G$-algebra to an elementary extension $M^* \succcurlyeq M$. 

1. Fix a subfamily $\cN_0 \subseteq \cN$ closed under finite intersection and taking supergroups from $\cN$ and denote by $\cSBP(\cN_0)$ the $G$-algebra of subsets of $G$ having SBP with respect to the topology generated by cosets of subgroups from $\cN_0$. Take any $G$-algebra $\cN_0 \subseteq \cA \le \cSBP(\cN_0)$. By Remark~\ref{rem:pf-sbp-dcl}, we have that $\cSBP(\cN_0)$ is d-closed, so $\cA^d \le \cSBP(\cN_0)$. 

\begin{lemma} We have $\cSBP(\cN_0) \cap \cN = \cN_0$, so in particular $\cA^d \cap \cN = \cN_0$.
\end{lemma}

\begin{proof} The right-to-left inclusion is obvious. In order to prove the other one, fix $N \in \cSBP(\cN_0) \cap \cN$ and let
\[
\cC = \{ C \in G/N : (\exists A \in \< \cN_0 \>) \, C \subseteq A \subseteq G \setminus N \}.
\]
Clearly $\cC$ is finite, so we can find $B \in \< \cN_0 \>$ containing the union of $\cC$ and disjoint from $N$. Since $B$ is clopen and $N$ has SBP, we can find a coset $K \subseteq G \setminus B$ of a subgroup in $\cN_0$ contained in $N$ or $G \setminus N$. We claim that the first possibility must hold.

Assume for contradiction that $K \subseteq G \setminus N$. Then $K \cap C \neq \varnothing$ for some $N \neq C \in G/N$. Let $A$ denote a union of finitely many translates of $K$ by elements of $N$ covering $C$. Then $A \in \< \cN_0 \>$ and $C \subseteq A \subseteq G \setminus N$, hence $C \in \cC$, contradicting the fact that $K \subseteq G \setminus B$.

It follows that $K \subseteq N$. Let $K$ be a coset of $F \in \cN_0$. Then $F \le N$ and the conclusion follows because $\cN_0$ is closed under taking supergroups.
\end{proof}

Consider the structure $M = (G, \cdot, A)_{A \in \cA}$ and take an elementary extension $M \preccurlyeq M^* = (G^*, \cdot, A^*)_{A \in \cA}$. The family $\{ A^* : A \in \cA \}$ is a $G$-algebra, but not necessarily a $G^*$-algebra, so let $\cA^*$ denote the $G^*$-algebra generated by that family. Also let $\cN_0^* = \{ N^* : N \in \cN_0 \}$, which makes sense as $\cN_0 \subseteq \cA$, and
\[
\cN^* := \{ N \le G^* : N \text{ is a normal subgroup of finite index } \}.
\]
Clearly $\cN_0^* \subseteq \cN^*$ is a family closed under finite intersection and taking supergroups from $\cN^*$ and $\cN_0^* \subseteq \cA^* \le \cSBP(\cN_0^*)$, hence as before $(\cA^*)^d \le \cSBP(\cN_0^*)$ and $(\cA^*)^d \cap \cN^* = \cN_0^*$. Let $\cE(M)$ and $\cE(M^*)$ denote the Ellis group of $S(\cA)$ and $S(\cA^*)$, respectively. By Theorem~\ref{thm:pf-ell}, we have that
\begin{align*}
\cE(M) \cong \projlim_{N \in \cN_0} G/N & & \text{and} & & \cE(M^*) \cong \projlim_{N^* \in \cN_0^*} G^*/N^*.
\end{align*}
The inverse systems $(G/N)_{N \in \cN_0}$ and $(G^*/N^*)_{N^* \in \cN_0^*}$ are isomorphic. Therefore $\cE(M) \cong \cE(M^*)$, so here the Ellis group is preserved between models.

In the next example we assume that the reader is familiar with the relation $\preccurlyeq^*$ between models, defined in the beginning of Section 2 in \cite{New12Aug}.
\begin{example} Consider a group $G$ definable in an arbitrary structure $M$. Let $\Def_{\ext, G}(M)$ denote the $G$-algebra of subsets of $G$ externally definable in $M$ and
\begin{align*}
\cN_0 := \Def_{\ext, G}(M) \cap \cN, && \cA := \Def_{\ext, G}(M) \cap \cSBP(\cN_0).
\end{align*}
Now let $M \preccurlyeq^* M^*$. Since $\cA \subseteq \Def_{\ext, G}(M)$, for each $A \in \cA$, the structure $M^*$ provides an interpretation $A^{M^*} \in \Def_{\ext, G}(M^*)$. Define $\cA^* \le \Def_{\ext, G}(M^*)$ as the $G^*$-algebra generated by the family $\{ A^{M^*} : A \in \cA \}$. We also define $\cN_0^*$ and $\cN^*$ in the same way as in the last scenario.

Since $\cN_0 \subseteq \cA \le \cSBP(\cN_0)$, the conclusions from the last scenario remain true here by the same reasoning. In particular, for the two corresponding algebras $\cA \le \Def_{\ext, G}(M)$ and $\cA^* \le \Def_{\ext, G}(M^*)$ we have that the Ellis groups of $S(\cA)$ and $S(\cA^*)$ are isomorphic.
\end{example}

2. Consider any structure $M = (G, \cdot, C, N)$, where $(G, \cdot)$ is a group, $C \subseteq G$ and $\cN_0 := \{ N(c, G) : c \in C \}$ is a family of normal subgroups of $G$ of uniformly bounded index. Take an elementary extension $M \preccurlyeq M^* = (G^*, \cdot, C^*, N^*)$ and let $\cN_0^* := \{ N^*(c, G^*) : c \in C^* \}$. Define $\overline{\cN_0}$ as the closure of $\cN_0$ under finite intersection and supergroups from $\cN$, and $\overline{\cN_0^*}$ in the same way. For any $G$-algebra $\cA \le \cSBP(\overline{\cN_0})$ containing $\cN_0$ and $G^*$-algebra $\cA^* \le \cSBP(\overline{\cN_0^*})$ containing $\cN_0^*$, denote by $\cE(M)$ the Ellis group of $S(\cA)$ and by $\cE(M^*)$ the Ellis group of $S(\cA^*)$. Reasoning as in the previous scenario, we get
\begin{align*}
\cE(M) \cong \projlim_{N \in \overline{\cN_0}} G/N & & \text{and} & & \cE(M^*) \cong \projlim_{N^* \in \overline{\cN_0^*}} G^*/N^*.
\end{align*}
This time the Ellis groups need not be isomorphic because the set $A^*$ may be essentially bigger than $A$. However, $\cE(M)$ is still a homomorphic image of $\cE(M^*)$ since the system $(G/N)_{N \in \overline{\cN_0}}$ is isomorphic to a subsystem of $(G^*/N^*)_{N^* \in \overline{\cN_0^*}}$.

Because of Corollary~\ref{cor:ell-pf}, given a $G$-algebra, it is reasonable to ask whether all of its strongly generic sets are periodic. Below we prove that a weaker variant of this property holds in $\cSBP$ (with respect to the profinite topology).

\begin{proposition} \label{prop:pf-usg-per} Assume $A \in \cSBP$ is USG. Then $A$ is periodic.
\end{proposition}

\begin{proof} Consider the structure $M = (G, \cdot, A, F)_{F \in \cN}$, where $\cN$ denotes the family of all normal subgroups of $G$ of finite index. Assume for contradiction that $A$ is not periodic. Then the following is a consistent $2$-type in $M$:
\[
q(x, y) = \{ x \in A, y \notin A \} \cup \{ xF = yF : F \in \cN \}.
\]
Take $N \succcurlyeq M$ such that $N \models q(\alpha, \beta)$ for some $\alpha, \beta \in N$. On the one hand, $A^N$ is strongly generic in $G^N$. On the other hand, we will show that the set
\[
B := \alpha^{-1} A^N \setminus \beta^{-1} A^N
\]
is nowhere dense with respect to the topology on $G^N$ generated by cosets of subgroups $F^N$, where $F \in \cN$. Since $e \in B$, it follows that $B$ is neither empty nor generic, which is a contradiction.

Consider any basic open set $F^N c$, where $F \in \cN$ and $c \in G^N$. In fact, we may assume that $c \in G$. Take $a \in G$ such that $aF^N = \alpha F^N$. Since $A$ has SBP, $\chi_A$ is constant on some coset $Hd \subseteq aFc$, where $H \in \cN$ and $d \in G$. It follows that $H^N d \subseteq \alpha F^N c$, so $\alpha^{-1} H^N d \subseteq F^N c$. Since $N \models q(\alpha, \beta)$, we have the equality $\alpha^{-1} H^N d = \beta^{-1} H^N d$. This basic open set is either contained in $\alpha^{-1} A^N$ and $\beta^{-1} A^N$, or disjoint from both. Therefore it is disjoint from $B$, as needed.
\end{proof}

\begin{corollary} \label{cor:tf-sg-nusg} Assume $A \subseteq G$ is tree-founded where the tree is linear, irreducible and infinite. Then $A$ is strongly generic but not uniformly strongly generic.
\end{corollary}

\begin{proof} Follows from Propositions~\ref{prop:tf-sg}, \ref{prop:tree-nper}, \ref{prop:atf-sbp} and \ref{prop:pf-usg-per}.
\end{proof}

\subsection{Compact groups} \label{sub:comp}

In this subsection we assume $G$ is a compact topological group. We begin by proving that just as in the previous subsection, the $G$-algebra $\cSBP$ is d-closed. This result is joint with Newelski. For any $G$-algebra $\cA \le \cP(G)$, let $\cA^s$ denote the smallest Boolean algebra of subsets of $G$ containing $\cA$ and closed under both left and right translation.

\begin{lemma} \label{lem:dcl-nwd} 
Assume $\cA \le \cSBP$ is a $G$-algebra. Then for any $C \in \cA^d$ there is $B \in \cA^s$ such that $C \symdif B$ is nowhere dense.
\end{lemma}

\begin{proof} First assume that $C = d_q A$ for some $A \in \cA, q \in \beta G$. Since $G$ is compact, there is $b \in G$ such that each open neighbourhood of $b$ belongs to $q$. For any $h \in \int (Ab^{-1})$, we have that $b \in \int (h^{-1} A)$, thus $h^{-1} A \in q$ and so $h \in d_q A$. Hence $\int (Ab^{-1}) \subseteq d_q A$. Repeating the argument for $G \setminus A$ in place of $A$, we get $\int (G \setminus Ab^{-1}) \subseteq d_q(G \setminus A) = G \setminus d_q A$. It follows that $d_q A \symdif Ab^{-1} \subseteq \bd(Ab^{-1})$, so $B := Ab^{-1} \in \cA^s$ has the desired property.

Now assume that $C \in \cA^d$ is arbitrary. Using Fact~\ref{fact:dcl-gen}, we can write $C = \tau(d_{q_1} A_1, \ldots, d_{q_n} A_n)$, where $\tau$ is a Boolean term and $A_1, \ldots, A_n \in \cA$, $q_1, \ldots, q_n \in \beta G$. For $i = 1, \ldots, n$, we use the first part to find $B_i \in \cA^s$ such that $d_{q_i} A_i \symdif B_i$ is nowhere dense. Then $C \symdif B$ is nowhere dense, where $B = \tau(B_1, \ldots, B_n) \in \cA^s$.
\end{proof}

\begin{proposition} \label{prop:comp-sbp-dcl} $\cSBP$ is a d-closed $G$-algebra.
\end{proposition}

\begin{proof} Take any $C \in \cSBP^d$ and using Lemma~\ref{lem:dcl-nwd}, find $B \in \cSBP^s = \cSBP$ such that $C \symdif B$ is nowhere dense. Then clearly $C \in \cSBP$.
\end{proof}

Consider any $G$-algebra $\cA \le \cSBP$. Again we are going to describe the Ellis group of the flow $S(\cA)$. Given $A \in \cSBP$, let $\varrho(A)$ denote the set $\int(\cl(A))$, which is the unique regular open set such that $A \symdif \varrho(A)$ is nowhere dense. For an overview of basic properties of the operation $\varrho$, see Fact~\ref{fact:reg}. 

For $x, y \in G$, let
\[
x \sim_{\cA} y \qquad \text{if} \qquad (\forall B \in \cA^d) ( x \in \varrho(B) \iff y \in \varrho(B) )
\]
and $N_{\cA} := [e]_{\sim_{\cA}}$. Clearly $\sim_{\cA}$ is an equivalence relation on $G$ which is closed under left and right translation by Corollary~\ref{cor:dcl-rt}. It follows that $N_{\cA}$ is a normal subgroup of $G$ and $G/{\sim_{\cA}} = G/N_{\cA}$.

\begin{theorem} \label{thm:comp-ell} Assume $\cA \le \cSBP$ is a $G$-algebra. The Ellis group of $S(\cA)$ is isomorphic to $G/N_{\cA}$.
\end{theorem}

First we need the following lemma:

\begin{lemma} \label{lem:comp-sep} Assume $D \in \cSBP$ and $a \in \varrho(D)$, $b \in G \setminus \varrho(D)$. Then there is $g \in G$ such that $a \in \varrho(gD)$ and $b \in \varrho(G \setminus gD)$.
\end{lemma}

\begin{proof} Since $\varrho(D)$ is an open neighbourhood of $a$, we have that $ga \in \varrho(D)$ when $g$ is sufficiently close to $e$. Moreover, since $b \in G \setminus \varrho(D) = \cl(\int(G \setminus D))$, there are $g$ arbitrarily close to $e$ such that $gb \in \int(G \setminus D)$. Combining the two statements, we get that there is $g \in G$ such that $ga \in \varrho(D)$ and $gb \in \int(G \setminus D) \subseteq \varrho(G \setminus D)$, hence $a \in \varrho(g^{-1}D)$ and $b \in \varrho(G \setminus g^{-1}D)$.
\end{proof}

\begin{proof}[Proof of Theorem~\ref{thm:comp-ell}.] Repeating the argument from the proof of Theorem~\ref{thm:pf-ell}, we get that there is a unique minimal ideal $I \trianglelefteqslant_m S(\cA^d)$, which consists of generic points of $S(\cA^d)$, and the generic sets in $\cA^d$ are precisely the ones with non-empty interior. Define $\pi : I \to G/N_{\cA}$ so that 
\[
\pi(q) = \{ x \in G : (\forall B \in \cA^d) ( x \in \varrho(B) \implies B \in q ) \}.
\]
First we show that $\pi$ is well defined, that is, each $\pi(q)$ is an equivalence class of $\sim_{\cA}$. Fix $q \in I$ and note that the set $\pi(q)$ is non-empty. Indeed, if $\pi(q) = \varnothing$, then for each $a \in G$ there is $B \in \cA^d$ such that $a \in \varrho(B)$ and $B \notin q$. By the compactness of $G$, we can find $B_1, \ldots, B_n \in \cA^d$ satisfying $\varrho(B_1) \cup \ldots \cup \varrho(B_n) = G$ and $B := B_1 \cup \ldots \cup B_n \notin q$. It follows that the set $G \setminus B$ is nowhere dense, since
\[
G \setminus B = B \symdif G = \left( \bigcup_{i=1}^n B_i \right) \symdif \left( \bigcup_{i=1}^n \varrho(B_i) \right) \subseteq \bigcup_{i=1}^n \big( B_i \symdif \varrho(B_i) \big).
\]
But $G \setminus B \in q$, which contradicts the fact that $q$ is generic.

So there exists $a \in \pi(q)$. We will prove that $\pi(q) = [a]_{\sim_{\cA}}$. The right-to-left inclusion is obvious and for the other one, take $b \notin [a]_{\sim_{\cA}}$. By Lemma~\ref{lem:comp-sep}, there is $B \in \cA^d$ such that $a \in \varrho(B)$ and $b \in \varrho(G \setminus B)$ (or the other way, in which case we replace $B$ with $G \setminus B$). Then $B \in q$, so $G \setminus B$ witnesses that $b \notin \pi(q)$.

By Lemma~\ref{lem:ell-epi}, the proof will be complete once we show that $\pi$ is a semigroup epimorphism and $\pi^{-1}[ \{ N_{\cA} \} ] \subseteq J(I)$. To check that it is a homomorphism, fix $p, q \in I$ and take any $a \in \pi(p), b \in \pi(q)$. It suffices to show that $ab \in \pi(p \ast q)$. Fix any $W \in \cA^d$ such that $ab \in \varrho(W)$ and pick open neighbourhoods $U \subseteq G$ of $a$ and $V \subseteq G$ of $b$ satisfying $U \cdot V \subseteq \varrho(W)$. For each $x \in U$, we have that $b \in V \subseteq \varrho(x^{-1}W)$, so $x^{-1} W \in q$ and therefore $x \in d_q W$. It follows that $U \subseteq d_q W$, hence $a \in \varrho(d_q W)$ and so $d_q W \in p$, which means that $W \in p \ast q$.

For surjectivity, fix $K \in G/N_{\cA}$ and write $K = aN_{\cA}$, where $a \in G$. The family
\[
\{ B \in \cA^d : a \in \varrho(B) \}
\]
is a filter of $\cA^d$ consisting of generic sets, so it extends to a generic ultrafilter $q \in S(\cA^d)$. Then $q \in I$ and $a \in \pi(q)$, hence $\pi(q) = [a]_{\sim_{\cA}} = K$.

Finally, we check that $\pi^{-1}[ \{ N_{\cA} \} ] \subseteq J(I)$. Assume for contradiction that $\pi(q) = N_{\cA}$ for some $q \in I \setminus J(I)$. Then $q \ast q \neq q$, so there is $C \in \cA^d$ such that $C \in q$ and $C \notin q \ast q$, which implies $C \setminus d_q C \in q$. Since $q$ omits nowhere dense sets, there is a non-empty open subset $U \subseteq C \setminus d_q C$. Pick any $g \in U$. Then $e \in \varrho(g^{-1} C)$, so $g^{-1} C \in q$. Thus $g \in U \cap d_q C$, which is a contradiction.
\end{proof}

The description of the relation $\sim_{\cA}$, and thus of the Ellis group of $S(\cA)$, can be simplified in the following way:

\begin{remark} For every $x, y \in G$,
\[
x \sim_{\cA} y \iff (\forall C \in \cA^s) ( x \in \varrho(C) \iff y \in \varrho(C) ).
\]
\end{remark}

\begin{proof} The left-to-right implication is trivial since $\cA^s \subseteq \cA^d$ by Corollary~\ref{cor:dcl-rt}. For the other one take any $B \in \cA^d$. By Lemma~\ref{lem:dcl-nwd}, we can write $B = C \symdif M$ for some $C \in \cA^s$ and nowhere dense $M \subseteq G$. Then $\varrho(B) = \varrho(C)$, so the conclusion follows.
\end{proof}

By Proposition~\ref{prop:pf-usg-per}, when a group is equipped with the profinite topology, the algebra $\cSBP$ does not contain non-periodic uniformly strongly generic sets. The same is true for compact groups, as we prove below. Recall a basic fact:

\begin{fact} Every clopen $C \subseteq G$ is periodic.
\end{fact}

\begin{proof} We first find an open neighbourhood $W \subseteq G$ of identity such that $CW = C$. For any $a \in C$ choose an open neighbourhood $W_a \subseteq G$ of identity satisfying $(W_a)^2 \subseteq a^{-1} C$. The family $\{ a \cdot W_a : a \in C \}$ is an open cover of $C$, so it has a finite subcover $\{ a_1 W_{a_1}, \ldots, a_n W_{a_n} \}$. We claim that $W := W_{a_1} \cap \ldots \cap W_{a_n}$ has the desired property. Indeed, take any $b \in C$ and pick $i \in \{ 1, \ldots, n \}$ such that $b \in a_i W_{a_i}$. Then 
\[
bW \subseteq bW_{a_i} \subseteq a_i (W_{a_i})^2 \subseteq C.
\]
Hence $CW \subseteq C$ and the other inclusion is obvious.

Therefore we have that $W \cap W^{-1} \subseteq \Per( \chi_C )$. Since $W \cap W^{-1}$ is generic by the compactness of $G$, it follows that $\Per( \chi_C )$ has finite index and so $C$ is periodic.
\end{proof}

\begin{lemma} \label{lem:usg-clop} Assume $U_1, U_2, W \subseteq G$ are open, $U_1 \cup U_2$ is dense and $W$ is non-empty. Then there are open neighbourhoods $W'$ and $Z$ of identity such that for each $g \in G$ there is $h \in G$ such that $hW'$ is contained in $U_1$ or $U_2$ and $ZhW' \subseteq gW$.
\end{lemma}

\begin{proof}

To each $g \in G$ we assign an element $h_g$, an open neighbourhood $O_g$ of $g$ and open neighbourhoods $W'_g, Z_g$ of identity in the following way: we choose $j \in \{ 1, 2 \}$ so that $U_j \cap gW \neq \varnothing$ and we let $h_g$ denote any element of this intersection. We then have $g^{-1} e h_g e = g^{-1} h_g \in g^{-1} U_j \cap W$, so by continuity, we can find an open neighbourhood $O_g$ of $g$ and open neighbourhoods $Z_g, W'_g$ of identity satisfying $O_g^{-1} Z_g h_g W'_g \subseteq g^{-1} U_j \cap W$.

By compactness, there are $g_1, \ldots, g_n \in G$ such that $G = O_{g_1} \cup \ldots \cup O_{g_n}$. We will prove that $W' := W'_{g_1} \cap \ldots \cap W'_{g_n}$ and $Z := Z_{g_1} \cap \ldots \cap Z_{g_n}$ are as desired. Take any $g \in G$, choose $i$ so that $g \in O_{g_i}$ and let $h = h_{g_i}$. Then $O_{g_i}^{-1} Z h W' \subseteq g_i^{-1} U_j \cap W$ for some $j \in \{ 1, 2 \}$. In particular, $hW' \subseteq U_j$ (because $g_i \in O_{g_i}$ and $e \in Z$) and $ZhW' \subseteq gW$ (as $g \in O_{g_i}$).
\end{proof}

\begin{proposition} Assume $A \in \cSBP$ is USG. Then $A$ is clopen.
\end{proposition}

\begin{proof} Assume for contradiction that $A$ is not clopen, so that without loss of generality $e \in \bd A$. Take $N \in \NN$ such that $A \setminus gA$ is always either empty or $N$-generic. 

We will inductively construct non-empty open sets $W_0, \ldots, W_N \subseteq G$ and neighbourhoods of identity $Z_0, \ldots, Z_N \subseteq G$, such that for $0 \le n \le N$ and any $z \in Z_n$ and $g_1, \ldots, g_n \in G$, some left translation of $W_n$ is disjoint from $g_1 (A \setminus zA) \cup \ldots \cup g_n (A \setminus zA)$.

Let $W_0 = Z_0 = G$. Fix $n < N$ and assume that $W_i, Z_i$ have already been defined for $i \le n$. By Lemma~\ref{lem:usg-clop}, we can find open neighbourhoods of identity $W_{n+1} \subseteq G$ and $Z_{n+1} \subseteq Z_n$ such that for each $g \in G$ there is $h \in G$ such that $h W_{n+1}$ is contained in $\int(A)$ or $\int (G \setminus A)$ and $Z_{n+1} h W_{n+1} \subseteq gW_n$. Fix $z \in Z_{n+1}$ and $g_1, \ldots, g_{n+1} \in G$. By the induction hypothesis, some left translate $t W_n$, where $t \in G$, is disjoint from $g_1 (A \setminus zA) \cup \ldots \cup g_n (A \setminus zA)$. Take $h \in G$ as above corresponding to $g := g_{n+1}^{-1} t$.

If $h W_{n+1} \subseteq \int A$, then $zhW_{n+1}$ is disjoint from $A \setminus zA$, so $g_{n+1} z h W_{n+1}$ is disjoint from $g_{n+1} (A \setminus zA)$. At the same time it is a subset of $tW_n$, hence $g_{n+1} z h W_{n+1}$ is disjoint from $g_1 (A \setminus zA) \cup \ldots \cup g_{n+1} (A \setminus zA)$.

On the other hand, if $h W_{n+1} \subseteq \int(G \setminus A)$, then $h W_{n+1}$ is disjoint from $A \setminus zA$, so $g_{n+1} h W_{n+1}$ is disjoint from $g_{n+1} (A \setminus zA)$. At the same time it is a subset of $tW_n$, hence $g_{n+1} h W_{n+1}$ is disjoint from $g_1 (A \setminus zA) \cup \ldots \cup g_{n+1} (A \setminus zA)$.

This ends the construction. It follows that for any $z \in Z_N$ the set $A \setminus zA$ is not $N$-generic. Now take an open neighbourhood $Z$ of identity satisfying $ZZ^{-1} \subseteq Z_N$, choose $a \in Z \cap A, b \in Z \setminus A$ and let $z = ab^{-1} \in Z_N$. Then $a \in A \setminus zA$. Hence $A \setminus zA$ is neither empty nor $N$-generic, which contradicts the choice of $N$.
\end{proof}

\begin{corollary} \label{cor:comp-usg} If $A \in \cSBP$ is USG, then it is periodic. 
\end{corollary} 

\begin{theorem} \label{thm:comp-nusg} Assume $G$ is compact Hausdorff and infinite. Then there is a strongly generic set in $\cSBP$ that is not uniformly strongly generic.
\end{theorem}

\begin{proof} Assume for the contrary that every strongly generic set in $\cSBP$ is uniformly strongly generic, thus periodic by Corollary~\ref{cor:comp-usg}. Let us first prove that $G$ is profinite as a topological group. We provide a sketch and leave out the technical details. Let $\cE$ denote the Ellis group of $S(\cSBP)$. By Theorem~\ref{thm:comp-ell}, the groups $\cE$ and $G$ are algebraically isomorphic, since $G$ is $T_2$ and so $N_{\cSBP} = \{ e \}$. In fact, they are isomorphic as topological groups via the same isomorphism, where the topology on $\cE$ is induced from $S(\cSBP)$. On the other hand, by Corollary~\ref{cor:ell-pf} and Proposition~\ref{prop:comp-sbp-dcl}, $\cE$ is algebraically isomorphic to an appropriate projective limit of finite groups. Again it can be checked that these are isomorphic as topological groups. Hence $G$ is profinite as a topological group.

Consider the set $A \subseteq G$ constructed in Example~\ref{ex:gval}. Clearly $A \in \cSBP$ because $\bd A \subseteq F_{\infty}$ is not generic, hence nowhere dense. It was shown that $A$ is strongly generic and not periodic. This is a contradiction.
\end{proof}

\subsection{Precompact groups}

We begin by recalling some classical notions related to topological groups. For an extended study on the subject see e.g. \cite{AT08}.

A topological group $G$ is defined to be:
\begin{itemize}
\item \emph{totally bounded} (or \emph{precompact}), if every open neighbourhood $U$ of identity is generic;
\item \emph{(Ra\u{\i}kov) complete}, if every Cauchy filter on $G$ converges.
\end{itemize}
Moreover, it is known to be
\begin{itemize}
\item compact if and only if it is totally bounded and complete;
\item totally bounded if and only if it is a dense subgroup of a compact topological group.
\end{itemize} 

We aim to show that completeness does not essentially contribute to most of the results about compact groups from the previous subsection. Hence in this subsection we assume that $G$ is just totally bounded, so it is a dense subgroup of a compact topological group $H$. Let $\cSBP(G)$ and $\cSBP(H)$ denote the algebras of sets having SBP in $G$ and $H$, respectively.

Let us first see that the other path of generalization, that is, the complete groups, does not seem promising in terms of characterizing the Ellis group of non-trivial algebras.

\begin{proposition} \label{prop:sbp-ndcl} Consider the complete topological group $(\RR, +)$. There exists $A \subseteq \RR$ with SBP such that $\< A \>^d = \cP(\RR)$, where $\< A \>$ denotes the $\RR$-algebra generated by $\{ A \}$. In particular, $\cSBP(\RR)$ is not d-closed.
\end{proposition}
\begin{proof}
Let $\mathcal{R}$ be the family of finite unions of bounded open intervals with rational endpoints. Since $\mathcal{R}$ is countable, we can write $\mathcal{R} = \{ R_n : n \in \NN \}$. Pick a sequence $(k_n)$ of integers such that $k_n + R_n < k_{n+1} + R_{n+1}$ for $n \in \NN$, where by $X < Y$ we mean $(\forall x \in X)(\forall y \in Y) \, x < y$. Let 
\[
A = \bigcup_{n=1}^{\infty} (k_n + R_n).
\]
Then $A$ is open and for each finite disjoint $U, V \subseteq \RR$ there is $r \in \RR$ such that $U \subseteq A-r \subseteq \RR \setminus V$. Indeed, take any such $U, V$ and take a closed interval $[a, b]$ such that $U \cup V \subseteq [a, b]$. Then there is $R_n \in \mathcal{R}$ satisfying $U \subseteq R_n \subseteq \RR \setminus V$ and containing some open intervals $I, J$ such that $I < [a, b] < J$. It follows that $(A - k_n) \cap [a, b] = R_n$, so $U \subseteq A - k_n \subseteq \RR \setminus V$, as required.

Therefore each subset $B \subseteq \RR$ is in the pointwise closure of $\{ A-r : r \in \RR \}$, so by Remark~\ref{rem:dlim} (iv), it can be written as $B = d_q A$ for some $q \in S( \< A \>)$. Hence $\< A \>^d = \cP(\RR)$.
\end{proof}

\begin{remark} \label{rem:pre-her} Assume $M, A \subseteq G$.
\begin{enumerate}[label=(\roman{*})]
\item $M \subseteq G$ is nowhere dense if and only if $M = M' \cap G$ for some nowhere dense $M' \subseteq H$.
\item $A \in \cSBP(G)$ if and only if $A = A' \cap G$ for some $A' \in \cSBP(H)$. \noproof
\end{enumerate} 
\end{remark}

\begin{proposition} $\cSBP(G)$ is d-closed.
\end{proposition}

\begin{proof} Fix any $A \in \cSBP(G)$ and $q \in \beta G$. Take $A' \in \cSBP(H)$ satisfying $A = A' \cap G$ and the unique extension $q' \in \beta H$ of $q$. By Proposition~\ref{prop:comp-sbp-dcl}, we have that $d_{q'}(A') \in \cSBP(H)$, so $d_q A = d_{q'}(A') \cap G \in \cSBP(G)$.
\end{proof}

Take any $\cA \le \cSBP(G)$. Yet again we wish to describe the Ellis group of $S(\cA)$. The outline of the proof is essentially the same as that of Theorem~\ref{thm:comp-ell} and only some technical details are adjusted to work in the new setting. For $D \in \cSBP(H)$, let $\gamma(D) = D \cap G$, so $\gamma : \cSBP(H) \to \cSBP(G)$ is a $G$-algebra epimomorphism. Denote by $\cC$ the smallest Boolean algebra of subsets of $H$ containing $\gamma^{-1}[\cA^d]$ and closed under left and right translation by elements of $H$. For $x, y \in H$, let
\[
x \sim_{\cA} y \qquad \text{if} \qquad (\forall C \in \cC) ( x \in \varrho(C) \iff y \in \varrho(C) )
\]
and $N_{\cA} := [e]_{\sim_{\cA}}$. Clearly $\sim_{\cA}$ is an equivalence relation on $H$ closed under left and right translation, so $N_{\cA}$ is a normal subgroup of $H$ and $H/{\sim_{\cA}} = H/N_{\cA}$.

\begin{theorem} \label{thm:pre-ell} For any $\cA \le \cSBP(G)$, the Ellis group of $S(\cA)$ is isomorphic to $H/N_{\cA}$. 
\end{theorem}

The following lemma will be used:

\begin{lemma} \label{lem:pre-sep} Assume $a, b \in H$ and $a \not \sim_{\cA} b$. Then there is $B \in \gamma^{-1}[\cA^d]$ such that $a \in \varrho(B)$ and $b \in \varrho(H \setminus B)$.
\end{lemma}

\begin{proof} Without loss of generality there is $C \in \cC$ such that $a \in \varrho(C)$ and $b \notin \varrho(C)$. Let
\begin{align*}
\mathcal{Q} & = \{ h_{\ell} \cdot B \cdot h_r : h_{\ell}, h_r \in H, B \in \gamma^{-1}[\cA^d]  \}, \\
\mathcal{R} & = \{ Q_1 \cup \ldots \cup Q_n : Q_1, \ldots, Q_n \in \mathcal{Q} \},
\end{align*}
so that clearly
\[
\cC = \{ R_1 \cap \ldots \cap R_m : R_1, \ldots, R_m \in \mathcal{R} \}.
\]
Thus we can write $C = R_1 \cap \ldots \cap R_m$ for some $R_1, \ldots, R_m \in \mathcal{R}$. Given the identity $\varrho(C) = \varrho(R_1) \cap \ldots \cap \varrho(R_m)$, we have that
\begin{align*}
a \in \varrho(R_1) \cap \ldots \cap \varrho(R_m) \qquad \text{and} \qquad b \notin \varrho(R_1) \cap \ldots \varrho(R_m),
\end{align*}
so $a \in \varrho(R_i)$ and $b \notin \varrho(R_i)$ for some $i \in \{ 1, \ldots, m \}$. By Lemma~\ref{lem:comp-sep}, there is $R \in \mathcal{R}$ such that $a \in \varrho(R)$ and $b \in \varrho(H \setminus R)$ because $\mathcal{R}$ is closed under left translation.

Write $R = Q_1 \cup \ldots \cup Q_n$, where $Q_1, \ldots, Q_n \in \mathcal{Q}$. Then 
\[
b \in \varrho \left( \bigcap_{j=1}^n H \setminus Q_j \right) = \bigcap_{j=1}^n \varrho(H \setminus Q_j).
\]
On the other hand, $a \notin \varrho(H \setminus Q_j)$ for some $j \in \{ 1, \ldots, n \}$, as otherwise 
\[
a \in \varrho(R) \cap \bigcap_{j=1}^n \varrho(H \setminus Q_j) = \varrho(\varnothing) = \varnothing.
\]
Consequently, $b \in \varrho(H \setminus Q_j)$ and $a \notin \varrho(H \setminus Q_j)$ for some $j \in \{ 1, \ldots, n \}$. Since $\mathcal{Q}$ is closed under left translation, by Lemma~\ref{lem:comp-sep}, we can find $Q \in \mathcal{Q}$ such that $a \in \varrho(Q)$ and $b \in \varrho(H \setminus Q)$.

Finally, let $Q = h_{\ell} \cdot B \cdot h_r$ for some $h_{\ell}, h_r \in H$ and $B \in \gamma^{-1}[\cA^d]$. Then 
\begin{align*}
h_{\ell}^{-1} a h_r^{-1} \in \varrho(B) \qquad \text{and} \qquad h_{\ell}^{-1} b h_r^{-1} \in \varrho(H \setminus B).
\end{align*}
Since $\varrho(B)$ and $\varrho(H \setminus B)$ are open, we have that 
\begin{align*}
g_{\ell}^{-1} a g_r^{-1} \in \varrho(B) \qquad \text{and} \qquad g_{\ell}^{-1} b g_r^{-1} \in \varrho(H \setminus B),
\end{align*}
when $g_{\ell}$ is sufficiently close to $h_{\ell}$ and $g_r$ is sufficiently close to $h_r$. Thus we can find such $g_{\ell}, g_r$ in $G$ so that
\begin{align*}
a \in \varrho(g_{\ell} B g_r) \qquad \text{and} \qquad b \in \varrho(H \setminus g_{\ell} B g_r).
\end{align*}
Clearly $g_{\ell} B g_r \in \gamma^{-1}[\cA^d]$ by Corollary~\ref{cor:dcl-rt}.
\end{proof}

\begin{proof}[Proof of Theorem~\ref{thm:pre-ell}.] Once more we repeat the argument from the proof of Theorem~\ref{thm:pf-ell} and conclude that there is a unique minimal ideal $I \trianglelefteqslant_m S(\cA^d)$, which consists of generic points of $S(\cA^d)$, and the generic sets in $\cA^d$ are precisely those with non-empty interior. Define $\pi : I \to H/N_{\cA}$ as
\[
\pi(q) = \{ x \in H : (\forall B \in \gamma^{-1}[\cA^d]) ( x \in \varrho(B) \implies \gamma(B) \in q \}.
\]
First we show that $\pi$ is well defined. Fix $q \in I$ and note that the set $\pi(q)$ is non-empty. Indeed, if $\pi(q) = \varnothing$, then for each $a \in H$ we can find $B \in \gamma^{-1}[\cA^d]$ such that $a \in \varrho(B)$ but $\gamma(B) \notin q$. By the compactness of $H$, we can find $B_1, \ldots, B_n \in \gamma^{-1}[\cA^d]$ satisfying $\varrho(B_1) \cup \ldots \cup \varrho(B_n) = H$ and $\gamma(B) \notin q$, where $B := B_1 \cup \ldots \cup B_n$. The set $H \setminus B$ is nowhere dense in $H$, since
\[
H \setminus B = B \symdif H = \left( \bigcup_{i=1}^n B_i \right) \symdif \left( \bigcup_{i=1}^n \varrho(B_i) \right) \subseteq \bigcup_{i=1}^n \big( B_i \symdif \varrho(B_i) \big).
\]
It follows that the set $G \setminus \gamma(B) = \gamma(H \setminus B)$ is nowhere dense in $G$. But $G \setminus \gamma(B) \in q$, which contradicts the fact that $q$ is generic.

So there exists $a \in \pi(q)$. We will prove that $\pi(q) = [a]_{\sim_{\cA}}$. The right-to-left inclusion is obvious and for the other one, take $b \notin [a]_{\sim_{\cA}}$. By Lemma~\ref{lem:pre-sep}, there is $B \in \gamma^{-1}[\cA^d]$ such that $a \in \varrho(B)$ and $b \in \varrho(H \setminus B)$. Then $\gamma(B) \in q$, so $H \setminus B$ witnesses that $b \notin \pi(q)$.

By Lemma~\ref{lem:ell-epi}, the proof will be complete once we show that $\pi$ is a semigroup epimorphism and $\pi^{-1}[ \{ N_{\cA} \} ] \subseteq J(I)$. To check that it is a homomorphism, fix $p, q \in I$ and take any $a \in \pi(p), b \in \pi(q)$. It suffices to show that $ab \in \pi(p \ast q)$. Fix any $W \in \gamma^{-1}[\cA^d]$ such that $ab \in \varrho(W)$ and pick open neighbourhoods $U \subseteq H$ of $a$ and $V \subseteq H$ of $b$ satisfying $U \cdot V \subseteq \varrho(W)$. For each $x \in U \cap G$ we have that $x^{-1} W \in \gamma^{-1}[\cA^d]$ and $b \in V \subseteq \varrho(x^{-1}W)$, so $\gamma(x^{-1} W) \in q$ and therefore $x \in d_q \gamma(W)$. It follows that $U \cap G \subseteq d_q \gamma(W)$. By Remark~\ref{rem:pre-her}, there is $D \in \cSBP(H)$ such that $d_q \gamma(W) = \gamma(D)$. So we have that $U \cap G \subseteq D \cap G$, hence $\cl(U) = \cl(U \cap G) \subseteq \cl(D \cap G) \subseteq \cl(D)$ and so $U \subseteq \varrho(U) \subseteq \varrho(D)$. It follows that $a \in \varrho(D)$ and $D \in \gamma^{-1}[\cA^d]$, thus $d_q \gamma(W) = \gamma(D) \in p$, which means that $\gamma(W) \in p \ast q$.

For surjectivity, fix $K \in H/N_{\cA}$ and write $K = aN_{\cA}$, where $a \in H$. The family
\[
\{ \gamma(B) : B \in \gamma^{-1}[\cA^d] \ \& \ a \in \varrho(B) \}
\]
is closed under finite intersection and consists of generic sets. (By a slightly more elaborate argument it is in fact a filter of $\cA^d$, but we do not need it.) Thus the family extends to a generic ultrafilter $q \in S(\cA^d)$. Then $q \in I$ and $a \in \pi(q)$, hence $\pi(q) = [a]_{\sim_{\cA}} = K$.

Finally, we check that $\pi^{-1}[ \{ N_{\cA} \} ] \subseteq J(I)$. Assume for contradiction that $\pi(q) = N_{\cA}$ for some $q \in I \setminus J(I)$. Then $q \ast q \neq q$, so there is $B' \in \cA^d$ such that $B' \in q$ and $B' \notin q \ast q$, which implies $B' \setminus d_q B' \in q$. Since $q$ is generic, there is a non-empty subset $U' \subseteq B' \setminus d_q B'$ open in $G$. Pick any $g \in U'$ and write $B' = \gamma(B)$ and $U' = \gamma(U)$, where $B \in \cSBP(H)$ and $U \subseteq H$ is open. We have $\gamma(U) \subseteq \gamma(B)$ and consequently, $U \subseteq \varrho(B)$. It follows that $e \in g^{-1} U \subseteq \varrho(g^{-1} B)$ and $g^{-1} B \in \gamma^{-1}[\cA^d]$, so $g^{-1}B' = \gamma(g^{-1} B) \in q$. Thus $g \in U' \cap d_q B'$, which is a contradiction.
\end{proof}

It follows directly from Lemma~\ref{lem:pre-sep} that
\begin{remark} \label{rem:pre-sim} For every $x, y \in G$,
\[
x \sim_{\cA} y \iff (\forall B \in \gamma^{-1}[\cA^d]) ( x \in \varrho(B) \iff y \in \varrho(B) ).
\]
\end{remark}
\noindent
In the spirit of further simplification, using techniques as in the proofs of Lemmas~\ref{lem:dcl-nwd} and \ref{lem:pre-sep} it is possible to prove that in fact for $x, y \in G$,
\[
x \sim_{\cA} y \iff (\forall B \in \gamma^{-1}[\cA^s]) ( x \in \varrho(B) \iff y \in \varrho(B) ).
\]

\subsubsection{Applications in o-minimal structures} 

So far the results of this section have been purely abstract. Therefore now we take a moment to show how they can be applied in a classical model-theoretic setting. Consider an o-minimal structure $M = (M, \le, \ldots)$, where $\le$ is a dense linear order without endpoints. Assume $G$ is a group definable in $M$ with $\dim G = d$ and $M \preccurlyeq N$. We write $G^N$ for the interpretation of $G$ in $N$. The models $M$ and $N$ are equipped with the order topologies and their powers with the product topologies. 

By a fundamental result of Pillay, $G$ can be given a structure of a definable manifold making it a Hausdorff topological group:

\begin{proposition}[{\cite[Proposition 2.5]{Pil88}}] \label{prop:def-man}
Assume $G$ is a group definable in an o-minimal structure $M$ with $\dim G = d$. There is a topology $\tau$ on $G$ and a large (hence generic) definable subset $V \subseteq G$ such that
\begin{itemize}
\item $G$ with the topology $\tau$ is a topological group;
\item $V$ is a union of disjoint definable subsets $U_1, \ldots, U_r \subseteq G$ such that for each $i = 1, \ldots, r$, $U_i$ is $\tau$-open in $G$ and there is a definable (in $M$) homeomorphism between $U_i$ with $\tau$ and some open subset $V_i \subseteq M^d$.
\end{itemize}
\end{proposition}

Although the following results hold true in this setting, for simplicity of the argument we will assume that $G \subseteq M^d$, $V$ is open in $M^d$ and $\tau \restriction V$ already agrees with the topology induced from $M^d$, and also $e \in V$. 

The generic subset $V^N \subseteq G^N$ equipped with the topology induced from $N^d$ gives rise to a topology on $G^N$ (via finitely many translations) which also makes $G^N$ a topological group. Assume that $G^N$ is compact with respect to this topology. Applying the tools developed in the last two subsections we will explicitly compute and relate the Ellis groups of $S_{\ext, G}(M)$ and $S_{\ext, G}(N)$. As usual, we let $\Def_{\ext, G}(M)$ denote the $G$-algebra of subsets of $G$ externally definable in $M$.

\begin{proposition} \label{prop:ext-sbp} Externally definable subsets of $M^n$ have SBP with respect to the product topology.
\end{proposition}

\begin{proof} We proceed by induction on $n$. For $n = 0$ the claim is trivial, so fix $n \in \NN$ and assume the claim holds for $n$. We will prove it holds for $n+1$. Take any externally definable $X \subseteq M^{n+1}$ and write $X = Y \cap M^{n+1}$, where $Y \subseteq \fC^{n+1}$ is definable in some $\fC \succcurlyeq M$. By the cell decomposition theorem and since $\cSBP$ is a Boolean algebra, we can assume that $Y$ is of the form
\[
Y = \{ (y, z) \in Y_0 \times \fC : z < f(y) \}
\]
for some definable $Y_0 \subseteq \fC^n$ and $f : Y_0 \to \fC$. It suffices to show that for any open box $B \subseteq M^{n+1}$ there is an open box $C \subseteq B$ that is either contained in or disjoint from $X$. 

Take any open box $B = B_0 \times (a, c) \subseteq M^n \times M$. By the induction hypothesis, the externally definable set $Y_0 \cap M^n$ has SBP in $M^n$, so we can find a box $C_0 \subseteq B_0$ contained in or disjoint from $Y_0 \cap M^n$. In the second case clearly $C_0 \times (a, c) \subseteq B$ is disjoint from $X$, so assume $C_0 \subseteq Y_0 \cap M^n$. Pick $b \in (a, c)$ and consider the definable set
\[
Y_1 = \{ y \in Y_0 : f(y) < b \} \subseteq \fC^n.
\]
By the induction hypothesis, we can find a box $D_0 \subseteq C_0$ contained in or disjoint from $Y_1 \cap M^n$. If $D_0 \subseteq Y_1 \cap M^n$, then $D_0 \times (b, c) \subseteq B$ is a box disjoint from $X$. Otherwise $D_0 \cap (Y_1 \cap M^n) = \varnothing$ and $D_0 \times (a, b) \subseteq B$ is a box contained in $X$. In either case the proof is complete.
\end{proof}

\begin{corollary} \label{cor:ext-sbp} $\Def_{\ext, G}(M) \le \cSBP(G)$ and $\Def_{\ext, G}(N) \le \cSBP(G^N)$.
\end{corollary}

If a box $B \subseteq N^n$ is a product of open intervals with endpoints in $M$, it will be called an open $M$-box. Repeating the proof of Proposition~\ref{prop:ext-sbp}, we get the following:
\begin{remark} \label{rem:ext-sbp0} Externally definable subsets of $N^n$ have SBP with respect to the topology on $N^n$ generated by open $M$-boxes. \noproof
\end{remark}

\begin{lemma} \label{lem:omin-den} $G$ is dense in $G^N$.
\end{lemma}

\begin{proof} If not, we can find a non-empty open subset $U \subseteq G^N$ disjoint from $G$. Without loss of generality it is of the form $U = tB$ for some open box $B \subseteq V^N$ and $t \in G^N$. By compactness, finitely many left translates of $B$ cover $G^N$ and by Remark~\ref{rem:ext-sbp0}, their intersections with $V^N$ have SBP with respect to the topology on $V^N$ generated by open $M$-boxes. Thus at least one such intersection $sB \cap V^N$, where $s \in G^N$, contains an open $M$-box $B_0 \subseteq V^N$. Again, finitely many right translates of $B_0$ cover $G^N$. Since $M \preccurlyeq N$, we can assume that they are of the form $B_0 r$, where $r \in G$. Pick $r \in G$ such that $st^{-1} \in B_0 r$. It follows that $r^{-1} \in t s^{-1} B_0 \subseteq tB$, which is a contradiction.
\end{proof}

\begin{corollary} \label{cor:omin-sub} $G$ is a topological subspace of $G^N$.
\end{corollary}

\begin{proof} It suffices to show that $V$ is a topological subspace of $V^N$. Clearly any open box $B \subseteq V$ is open in the subspace topology as it is of the form $B^N \cap V$. On the other hand, given an open box $C \subseteq V^N$ we will show that $C \cap V$ is open in $V$. Take any $b \in C \cap V$ and write $C = \prod (a_i, c_i)$, where $a_i, c_i \in N$. By Lemma~\ref{lem:omin-den}, we can find $x \in \prod (a_i, b_i) \cap V$ and $y \in \prod (b_i, c_i) \cap V$. Then $B = \prod (x_i, y_i)$ is an open box in $V$ satisfying $b \in B \subseteq C \cap V$. It follows that $C \cap V$ is open in $V$.
\end{proof}

Now we are ready to compute the Ellis groups of the $G$-flows $S_{\ext, G}(M)$ and $S_{\ext, G}(N)$. First we prove that the Ellis group of $S_{\ext, G}(N)$ is isomorphic to $G^N$. Since $G^N$ is compact, by Corollary~\ref{cor:ext-sbp} and Theorem~\ref{thm:comp-ell}, it suffices to show that $N_{\Def_{\ext, G}(N)} = \{ e \}$. Take any $x \in G^N \setminus \{ e \}$. Clearly we can find an open box $B \subseteq V^N$ such that $e \in B$ and $x \notin \cl B$. Then $B \in \Def_{\ext, G}(N)^d$ and $B$ witnesses that $x \not \sim_{\Def_{\ext, G}(N)} e$, hence $x \notin N_{\Def_{\ext, G}(N)}$.

Now we show that the Ellis group of $S_{\ext, G}(M)$ is also isomorphic to $G^N$. Using Lemma~\ref{lem:omin-den} and Corollary~\ref{cor:omin-sub}, we see that $G$ is a dense topological subgroup of the compact group $H = G^N$, hence it is precompact. By Corollary~\ref{cor:ext-sbp} and Theorem~\ref{thm:pre-ell}, it suffices to show that $N_{\Def_{\ext, G}(M)} = \{ e \}$. This is done as in the previous paragraph, since for any open box $B$ in $V^N$, $B$ has SBP in $G^N$ and $B \cap G$ is externally definable in $G$.

Thus we see that the Ellis groups of $S_{\ext, G}(M)$ and $S_{\ext, G}(N)$ are both isomorphic to $G^N$, hence to each other.

\subsection{Types at infinity}

\label{ss:infty}

In Subsection~\ref{sub:comp} we gave a description of the Ellis group of an arbitrary $G$-subalgebra $\cA \le \cSBP$ when $G$ is compact. In this subsection we aim to identify the difficulty in generalizing the argument to arbitrary topological groups.

Assume $G$ is any topological group and $\cA \le \cSBP$ is d-closed.\footnote{$\cSBP$ itself need not be d-closed, see Proposition \ref{prop:sbp-ndcl}. } Let $\cNWD$ denote the family of all nowhere dense sets, so that clearly $\cNWD \subseteq \cSBP$. Since $\cA \cap \cNWD$ is a proper $G$-ideal of $\cA$, we can find a minimal subflow $I \trianglelefteqslant_m S(\cA)$ such that $p \cap \cNWD = \varnothing$ for each $p \in I$. A good approach to describing the Ellis group is by characterizing all idempotents $v \in J(I)$, as the following basic fact shows.

\begin{fact} \label{fact:inf-ell} The relation on $I$ defined by
\[
p \sim q \iff (\exists v \in J(I)) \, q = v \ast p
\]
is a congruence and $I/{\sim}$ is isomorphic to the Ellis group of $S(\cA)$.
\end{fact}

\begin{proof} Fix $u \in J(I)$ and define $\varphi : I \to uI$ by $\varphi(p) = u \ast p$. It suffices to show that $\varphi$ is a semigroup epimorphism such that $p \sim q \iff \varphi(p) = \varphi(q)$ for $p, q \in I$. It is a homomorphism since $p \ast u = p$ for any $p \in I$ by Lemma~\ref{lem:im}. Furthermore, $\varphi(p) = p$ for $p \in uI$, so $\varphi$ is an epimorphism. Now fix $p, q \in I$. If $p \sim q$, i.e. $q = v \ast p$ for some $v \in J(I)$, then $uq = uvp = up$. Conversely, assume $up = uq$ and let $q \in vI$, where $v \in J(I)$. Then $vp = vup = vuq = vq = q$, hence $p \sim q$.
\end{proof}

Thus, roughly speaking, the Ellis group is obtained by dividing the minimal ideal by its idempotents. In the case of compact groups we have essentially proved the following characterization of idempotents in $I$:
\begin{equation} \label{eq:idem}
q \in J(I) \iff (\forall B \in \cA) \, \big( e \in \varrho(B) \Rightarrow B \in q \big).
\end{equation}
A more technical variant of the condition was also proved for precompact groups. If $q \in S(\cA)$ satisfies the right hand side of the equivalence (\ref{eq:idem}), we shall say that it is \emph{concentrated around the identity} and the set of such types is denoted by $S_e(\cA)$. The condition (\ref{eq:idem}) (or its variant) was at the core of describing the Ellis group in compact (or precompact) groups as the quotient of $G$ (or its compactification) by a normal subgroup obtained as an intersection of certain open neighbourhoods of identity. Fact~\ref{fact:inf-ell} expands the intuition behind that description.

In the general setting the right-to-left implication of (\ref{eq:idem}) still holds,\footnote{To see this, repeat the proof of the inclusion $\pi^{-1}[ \{ N_{\cA} \} ] \subseteq J(I)$ in Theorem~\ref{thm:comp-ell}.} while the other implication may fail, as we are about to show. However, any $q$ contradicting this implication must satisfy the following property:

\begin{proposition} \label{prop:idem-inf} Let $q \in J(I)$ and assume there is $B \in \cA$ such that $e \in \varrho(B)$ and $B \notin q$. Then there is a non-empty open subset $V \subseteq G$ such that for any $g \in G$ and $A \in \cA$, if $A \subseteq gV$, then $A \notin q$.
\end{proposition}

\begin{proof} Pick an open neighbourhood $V \subseteq G$ of $e$ such that $V^{-1} V \subseteq \varrho(B)$. Assume for contradiction that $g \in G$ and $A \in \cA$ satisfy $A \subseteq gV$ and $A \in q$. Since $q = q \ast q$, we have that $A \cap d_q A \in q$ and we can find some $h \in A \cap d_q A$. Then $h^{-1} A \subseteq (gV)^{-1} gV = V^{-1} V \subseteq \varrho(B)$, hence $h^{-1} A \setminus B \in \cNWD$. But $h^{-1} A \in q$ and $B \notin q$, which contradicts the fact that $q \cap \cNWD = \varnothing$.
\end{proof}

When $\cA = \cSBP$ (assuming that it is d-closed), the conclusion of Proposition~\ref{prop:idem-inf} says that $q$ lies outside every left translate of a certain open set. If $G$ is compact or even precompact, this is clearly impossible. More generally, such a type $q$ must lie outside of any $B \in \cA$ such that $\cl(B) \subseteq G$ is compact. Hence we may think of it as lying \emph{at infinity}, following the topological idea of ``approaching infinity'' as ``ultimately lying outside every compact subset''.

\begin{definition} We say that a type $q \in S(\cA)$ is \emph{at infinity} when there is a non-empty open subset $V \subseteq G$ such that whenever $A \in \cA$ and $A \subseteq gV$ for some $g \in G$, we have that $A \notin q$. The set of such types is denoted $S_{\infty}(\cA)$.
\end{definition}

\begin{corollary} \label{cor:idem-dich} If $q \in I$ is an idempotent, then $q$ is concentrated at identity or at infinity. \noproof
\end{corollary}

The following example suggests the existence of ``smaller'' and ``bigger'' infinities, although the notion seems difficult to capture in general.
\begin{example} Let $M = (R, +, \cdot, \le)$ be a real closed, non-archimedean field and $G = (R, +)$ a group definable in $M$. Pick any $\varepsilon > 0$ (e.g. infinitesimal) and let $q_{\varepsilon} \in S_{\ext, G}(M)$ be the type corresponding to the left half of the cut $(A_{\varepsilon}, R \setminus A_{\varepsilon})$, where
\[
A_{\varepsilon} = \bigcup_{n \in \NN} (-\infty, n \cdot \varepsilon).
\]
More precisely, $q_{\varepsilon}$ is the unique type containing the family 
\[
\{ A_{\varepsilon} \} \cup \{ (a, \infty) : a \in A_{\varepsilon} \}.
\]
Clearly $q_{\varepsilon}$ is an idempotent and $q_{\varepsilon} \in S_{\infty}(\Def_{\ext, G}(M))$, as witnessed by the open set $(0, \varepsilon)$. However, $q_{\varepsilon}$ need not be a counterexample to (\ref{eq:idem}), since it is not almost periodic unless $A_{\varepsilon} = R$.

Furthermore, let $q_{\infty} \in S_{\ext, G}(M)$ be the unique type containing the family $\{ (a, \infty) : a \in R \}$. Then $q_{\infty} \cap \cNWD = \varnothing$, $I := \{ q_{\infty} \}$ is a minimal ideal of $S_{\infty}(\Def_{\ext, G}(M))$ and $q_{\infty} \in J(I)$. Now the type $q_{\infty}$ is at infinity, witnessed by $V = (0, 1)$, and it is a counterexample to (\ref{eq:idem}).
\end{example}

The idea of a type at infinity becomes distorted if the algebra $\cA$ is not sufficiently rich. For instance, it may happen that a type is both concentrated at identity and at infinity. In the extreme case of $\cA = \{ \varnothing, G \}$, any non-empty, proper open subset $V \subseteq G$ witnesses that the unique type in $S(\cA)$ lies at infinity. Below is another example.

\begin{example} Consider the compact group
\[
G = \{ z \in \CC : |z| = 1 \} \le \CC^{\times}
\]
equipped with the standard topology. For $a, b \in G$, $a \neq b$, let $(a, b)$ denote the open arc going from $a$ to $b$ in the positive direction. Let $\cA$ denote the family of all subsets of $G$ of the form $S \symdif N$, where $S$ is a union of finitely many open arcs satisfying $(\forall z \in S) \, {-z} \in S$ and $N$ is nowhere dense. Clearly $\cA \le \cSBP$ and it follows from Lemma~\ref{lem:dcl-nwd} that $\cA$ is d-closed.

The types concentrated at identity are precisely $\hat{1}^-$ and $\hat{1}^+$, where $\hat{1}^-$ is the unique extension of the family
\[
\{ (g, 1) \cup (-g, -1) : g \in G \setminus \{ 1 \} \} \cup \{ G \setminus N : N \in \cNWD \}
\]
and $\hat{1}^+$ is the unique extension of 
\[
\{ (1, g) \cup (-1, -g) : g \in G \setminus \{ 1 \} \} \cup \{ G \setminus N : N \in \cNWD \}.
\]
There is a unique minimal ideal of $S(\cA)$ and it consists of all types $q \in S(\cA)$ such that $q \cap \cNWD = \varnothing$. All these types are at infinity, as witnessed by the open arc $(1, i)$. In particular, $\hat{1}^-$ and $\hat{1}^+$ are almost periodic types that are both concentrated at identity and at infinity.
\end{example}

The following remark gives a technical condition regarding the set $S_{\infty}(\cA)$ and its relation to $S_e(\cA)$.

\begin{remark} \label{rem:inf-ex} \leavevmode
\begin{enumerate}[label=(\roman{*})]
\item $S_{\infty}(\cA) \cap S_e(\cA) = \varnothing$ if and only if every non-empty open $V \subseteq G$ contains some $A \in \cA$ with non-empty interior.
\item $S_{\infty}(\cA) = \varnothing$ if and only if every non-empty open $V \subseteq G$ contains some generic $A \in \cA$.
\end{enumerate}
\end{remark}

\begin{proof} We prove all implications by contraposition.

\vspace{2mm} \noindent 
(i) $({\implies})$ Take a non-empty open $V \subseteq G$ such that every $A \in \cA$ contained in $V$ has empty interior. The family
\[
\mathcal{R} = \{ A \in \cA : e \in \varrho(A) \} \cup \{ G \setminus N : N \in \cNWD \}
\]
has the finite intersection property. Indeed, take any $A_1, \ldots, A_n \in \cA$ with $e \in \varrho(A_i)$ and $N \in \cNWD$. Then $e \in \varrho(A_1) \cap \ldots \cap \varrho(A_n) = \varrho(A_1 \cap \ldots \cap A_n)$, hence $A_1 \cap \ldots \cap A_n \notin \cNWD$ and so its intersection with $G \setminus N$ is not empty.

Take any $q \in S(\cA)$ extending $\mathcal{R}$. Then clearly $q \in S_e(\cA)$. Moreover, $V$ witnesses that $q \in S_{\infty}(\cA)$, because if $A \in \cA$ is contained in $gV$ for some $g \in G$, then $A$ is nowhere dense and so $A \notin q$. Hence $S_{\infty}(\cA) \cap S_e(\cA) \neq \varnothing$.

\vspace{2mm} \noindent 
$({\impliedby})$ Take $q \in S_{\infty}(\cA) \cap S_e(\cA)$. Since $q \in S_{\infty}(\cA)$, we can find a non-empty open $V \subseteq G$ such that $A \notin q$ whenever $A \in \cA$ and $A \subseteq gV$ for some $g \in G$. It remains to show that each $A \in \cA$ contained in $V$ has empty interior. Indeed, if $A \in \cA$ is contained in $V$ and $g \in \int(A)$, then $e \in \varrho(g^{-1} A)$. Thus also $g^{-1} A \in q$ and $g^{-1} A \subseteq g^{-1} V$, contradicting the choice of $V$.

\vspace{2mm} \noindent 
(ii) $({\implies})$ Take a non-empty open subset $V \subseteq G$ such that any $B \in \cA$ contained in $V$ is not generic. The following family has the finite intersection property:
\[
\mathcal{R} = \{ G \setminus gA : A \in \cA, A \subseteq V \text{ and } g \in G \}.
\]
Indeed, take any $A_1, \ldots, A_n \in \cA$ contained in $V$ and $g_1, \ldots, g_n \in G$. Then $B := A_1 \cup \ldots \cup \ldots A_n \subseteq V$ is not generic and so $g_1 A_1 \cup \ldots \cup g_n A_n \neq G$.

Now any $q \in S(\cA)$ extending $\mathcal{R}$ clearly belongs to $S_{\infty}(\cA)$.

\vspace{2mm} \noindent 
$({\impliedby})$ Take $q \in S_{\infty}(\cA)$ and non-empty open $V \subseteq G$ such that $A \notin q$ whenever $A \in \cA$ and $A \subseteq gV$ for some $g \in G$. If $A \in \cA$ is a subset of $V$ and $g_1, \ldots, g_n \in G$, then $g_i A \subseteq g_i V$, hence $g_i A \notin q$ and so $g_1 A \cup \ldots \cup g_n A \neq G$.
\end{proof}

We call the algebra $\cA$ \emph{reflective} if every non-empty open $V \subseteq G$ contains a subset $A \in \cA$ with non-empty interior. Remark~\ref{rem:inf-ex} says that $\cA$ is reflective if and only if the types concentrated at identity are different from those at infinity. For that reason we consider irreflective algebras degenerate and of secondary importance. It also quickly follows from Remark~\ref{rem:inf-ex} that when $\cA$ is reflective, then types at infinity exist precisely when $G$ is not precompact:

\begin{corollary} \label{cor:pre-inf} \leavevmode
\begin{enumerate}[label=(\alph{*})]
\item If $G$ is not precompact, then $S_{\infty}(\cA) \neq \varnothing$.
\item Assume $\cA$ is reflective. Then the converse of (a) holds.
\end{enumerate}
\end{corollary}

\begin{proof} (a) Follows directly from Remark~\ref{rem:inf-ex} (ii).

\vspace{2mm} \noindent
(b) Assume that $G$ is precompact. Since $\cA$ is reflective, every non-empty open $V \subseteq G$ contains some $A \in \cA$ with non-empty interior, hence generic. By Remark~\ref{rem:inf-ex} (ii), it follows that $S_{\infty}(\cA) = \varnothing$.
\end{proof}

It is natural to think that if infinity exists, then being at infinity is a more generic condition that not. Below we make this thought precise and prove it. 

\begin{proposition} \label{prop:inf-per} If $S_{\infty}(\cA) \neq \varnothing$, then every almost periodic point of $S(\cA)$ lies at infinity.
\end{proposition}

\begin{proof} First we prove that there is a non-empty open $W \subseteq G$ such that for each $A \in \cA$ contained in $W$ and $g_1, \ldots, g_n \in G$, the set $g_1 A \cup \ldots \cup g_n A$ is not dense. For this purpose we consider two cases.
\begin{enumerate}[leftmargin=7mm, label=(\roman{*})]
\item $\cA$ is reflective.

By Corollary~\ref{cor:pre-inf}, we get that $G$ is not precompact, so we can find a non-empty open neighbourhood $V \subseteq G$ of $e$ that is not generic. Let $W \subseteq G$ be an open neighbourhood of $e$ such that $W W^{-1} \subseteq V$. It suffices to show that for any $g_1, \ldots, g_n \in G$ the set $U := g_1 W \cup \ldots \cup g_n W$ is not dense. Indeed, given $g_1, \ldots, g_n \in G$, we can find $a \notin g_1 V \cup \ldots \cup g_n V$. In particular, $a \notin UW^{-1}$, thus $aW \cap U = \varnothing$.

\item $\cA$ is not reflective.

Take a non-empty open $W \subseteq G$ such that every $A \in \cA$ contained in $W$ has empty interior, and so is nowhere dense. Then $W$ is as desired.
\end{enumerate}

Now assume for contradiction that $p \in S(\cA) \setminus S_{\infty}(\cA)$ is almost periodic and pick $A \in p$ such that $A \subseteq gW$ for some $g \in G$. Since $\cl(G \cdot p)$ is a minimal subflow and $\cl(G \cdot p) \cap [A] \neq \varnothing$, by Remark \ref{rem:min-gen}, there are $g_1, \ldots, g_n \in G$ such that $\cl(G \cdot p) \subseteq g_1[A] \cup \ldots \cup g_n[A]$. Let $B = G \setminus (g_1A \cup \ldots \cup g_nA)$. Then $\int(B) \neq \varnothing$, so we can find $A' \in p$ such that $A' \subseteq hB$ for some $h \in G$. But then $\cl(G \cdot p) \cap [B] \neq \varnothing$, which is a contradiction.
\end{proof}

A consequence of the last result is the following dichotomy. If $G$ is precompact, all idempotents of $I$ are concentrated around the identity and an explicit description of the Ellis group follows. On the other hand, if $G$ is not precompact, then by Corollary~\ref{cor:pre-inf} and Proposition~\ref{prop:inf-per}, all minimal subflows of $S(\cA)$, along with their idempotents, are at infinity. It is thus likely that a completely different approach is necessary to describe the Ellis group.

\subsection{Regular open sets}
\label{sub:reg}

In the discussion after Proposition~\ref{prop:im-ell} we explained the significance of the problem of characterizing maximal generic $G$-subalgebras of a given d-closed $G$-algebra $\cA \le \cP(G)$. Following the spirit of the section, we continue to focus on the more tractable variant of the problem where $\cA \le \cSBP$. Even this case appears difficult due to the mysterious nature of strong genericity combined with the complexity of arbitrary nowhere dense sets. In this subsection we aim to reduce the problem to its variant in the realm of regular open sets, which we consider a non-trivial simplification. All the results from this subsection are joint with Newelski.

Let us recall some definitions related to $G$-algebras of subsets of $G$ while naturally extending them to arbitrary $G$-algebras. 

\begin{definition} \label{def:sgen-ext} Assume $\bA$ is a $G$-algebra.
\begin{enumerate}[label=(\roman{*})]
\item $A \in \bA$ is \emph{generic} if $\one = g_1 A \vee \ldots \vee g_n A$ for some $g_1, \ldots, g_n \in G$.
\item $\bA$ is \emph{generic} if every $A \in \bA \setminus \{ \zero \}$ is generic.
\item $A \in \bA$ is \emph{strongly generic} if the generated $G$-algebra $\< A \> \le \bA$ is generic.
\end{enumerate}
\end{definition}

\begin{definition} Assume $\bA \le \bB$ are $G$-algebras. For $A \in \bA$ and $q \in S(\bB)$, let
\[
d_q(A) = \{ h \in G : h^{-1} A \in q \}.
\]
Clearly $d_q : \bA \to \cP(G)$ is a $G$-algebra homomorphism.
\end{definition}

\begin{definition} Assume $\bA$ is a $G$-algebra and $\cA \le \cP(G)$ contains $\bigcup_{q \in S(\bA)} d_q[\bA]$. For $p \in S(\cA), q \in S(\bA)$, let
\[
p \ast q = \{ A \in \bA : d_q(A) \in p \}.
\]
Equivalently, $p \ast q = r_q(p)$, where $r_q : S(\cA) \to S(\bA)$ is the dual of $d_q : \bA \to \cA$.
\end{definition}

\begin{remark} \label{rem:hom-gen} Assume $\bA$ is generic and $\varphi : \bA \to \bB$ is a $G$-algebra homomorphism. Then $\varphi[\bA]$ is generic and $\varphi$ is a monomorphism. \noproof
\end{remark}

Throughout the subsection $G$ is an arbitrary topological group. Given any $A \in \cSBP$, define $\varrho(A) := \int(\cl(A))$. A set $A \in \cSBP$ is called \emph{regular open} if $\varrho(A) = A$. Let $\cNWD$ and $\bRO$ denote the families of all nowhere dense subsets of $G$ and all regular open subsets of $G$, respectively. We now recall some basic properties of regular open sets.

\begin{fact} \label{fact:reg} \leavevmode
\begin{enumerate}[label=(\alph{*})]
\item $\cNWD$ is a $G$-ideal of $\cSBP$.
\item For any $A \in \cSBP$, the set $U = \varrho(A)$ is a unique regular open set $U \subseteq G$ satisfying $U \symdif A \in \cNWD$. In particular, each coset of $\cNWD \trianglelefteqslant \cSBP$ contains exactly one $U \in \bRO$.
\item The bijection $\varphi : \bRO \to \cSBP / \cNWD$ defined by $\varphi(U) = U \symdif \cNWD$ gives rise to a $G$-algebra structure on $\bRO$ such that $\varphi$ becomes a $G$-algebra isomorphism. The structure can be described explicitly by
\[
\setlength{\arraycolsep}{1mm}
\begin{array}{lllll}
U \vee V & := & \varrho(U \cup V), \\[1ex]
U \wedge V & := & \varrho(U \cap V) & = & U \cap V, \\[1ex]
U^{\bot} & := & \varrho(G \setminus U) & = & \int(G \setminus U).
\end{array}
\]
In particular, typically $\bRO \not \le \cSBP$, even though $\bRO \subseteq \cSBP$.
\item $\varrho : \cSBP \to \bRO$ is a $G$-algebra epimorphism such that $\Ker \varrho = \cNWD$. Consequently, $\varrho(A) = \varrho(B)$ if and only if $A \symdif B \in \cNWD$. \noproof
\end{enumerate}
\end{fact}

The idea of the subsection is based on the following observation, which establishes a certain correspondence between generic $G$-subalgebras of $\cSBP$ and $\bRO$.

\begin{remark} \label{rem:gen-reg} \leavevmode
\begin{enumerate}[label=(\roman{*})]
\item Assume $\bB \le \bRO$ is generic and $q \in S(\bRO)$. Then $\cB := d_q[\bB] \le \cP(G)$ is generic.
\item Assume $\cB \le \cSBP$ is generic. Then there is a generic $\bB \le \bRO$ and some $q \in S(\bRO)$ such that $\varrho : \cB \to \bB$ and $d_q : \bB \to \cB$ are mutually inverse $G$-algebra isomorphisms.
\end{enumerate}
\end{remark}

\begin{proof} (i) Follows from Remark~\ref{rem:hom-gen}.

\vspace{2mm} \noindent
(ii) Let $\bB := \varrho[\cB] \le \bRO$. Then $\bB$ is generic and $\varrho : \cB \to \bB$ is an isomorphism. Therefore $\varrho^* : S(\bB) \to S(\cB)$ is a bijection, so we can find $u \in S(\bB)$ such that $\varrho^*(u) = \hat{e} \in S(\cB)$. Let $q \in S(\bRO)$ be an arbitrary extension of $u$. It remains to show that $d_q \circ \varrho = \id_{\cB}$. Take $B \in \cB$. Then $d_q(\varrho(B)) = d_u(\varrho(B)) = d_{\hat{e}}(B) = B$, where the second equality follows from Fact~\ref{fact:dual}.
\end{proof}

Now fix $\cA \le \cSBP$ that is d-closed. Recall that by Proposition~\ref{prop:im-ell}, if $\cB \le \cA$ is maximal generic, the Ellis group of $S(\cA)$ can be recovered from $\cB$ as
\[
\cE = \{ q \in S(\cB) : d_q[ \cB ] \subseteq \cB \}.
\]
The general goal is to describe maximal generic $G$-subalgebras $\cB \le \cA$ and then the group $\cE$. We propose the following approach: let $\bA := \varrho[\cA] \le \bRO$ and assume we managed to find a maximal generic $G$-subalgebra $\bB \le \bA$. Then we pick $q \in S(\bRO)$ and let $\cB := d_q[\bB]$, which in a perfect world will be a maximal generic $G$-subalgebra of $\cA$ isomorphic to $\bB$. Finally, we recover $\cE$ from $\cB$. 

Although Remark~\ref{rem:gen-reg} guarantess that every generic $\cB \le \cA$ can be obtained this way from some generic $\bB \le \bA$, there are two issues with this approach:
\begin{enumerate}
\item It is not evident whether $\cB := d_q[\bB] \le \cA$ is always maximal generic.
\item The condition $d_q[\cB] \subseteq \cB$ is not invariant under $G$-algebra isomorphism. Therefore in general it is not possible to retrieve $\cE$ directly from an isomorphic copy of $\cB$, such as $\bB$. Thus it is necessary to work through the isomorphism $d_q$, which may be complicated.
\end{enumerate}
These issues are addressed below, starting with the second one. For this purpose we define a subset $\fU \subseteq S(\bRO)$ such that any isomorphism $d_v$, where $v \in \fU$, may be considered relatively simple. We also identify which generic $\cB \le \cA$ can be obtained from a generic $\bB \le \bA$ via such an isomorphism.

\begin{definition} \leavevmode
\begin{enumerate}[label=(\alph{*})]
\item Let $\fU = \{ v \in S(\bRO) : V \in v \text{ for every } V \in \bRO \text{ with } e \in V \}$.
\item $A \subseteq G$ is \emph{tidy} if $V \subseteq A \subseteq \cl(V)$ for some $V \in \bRO$.
\item $\cA \le \cSBP$ is \emph{tidy} if each $A \in \cA$ is tidy.
\end{enumerate}
\end{definition}
The next proposition refines Remark~\ref{rem:gen-reg} by saying that generic algebras $\cB \le \cSBP$ that can be obtained from some generic $\bB \le \bRO$ via an ultrafilter $v \in \fU$ are precisely those that are tidy.
\begin{proposition} \label{prop:gen-reg} \leavevmode
\begin{enumerate}[label=(\roman{*})]
\item Assume $\bB \le \bRO$ is generic and $v \in \fU$. Then $\cB := d_v[\bB] \le \cSBP$ is tidy and generic. Moreover, $\varrho : \cB \to \bB$ and $d_v : \bB \to \cB$ are mutually inverse $G$-algebra isomorphisms.
\item Assume $\cB \le \cSBP$ is tidy and generic. Then there is a generic $\bB \le \bRO$ and some $v \in \fU$ such that $\varrho : \cB \to \bB$ and $d_v : \bB \to \cB$ are mutually inverse $G$-algebra isomorphisms.
\end{enumerate}
\end{proposition}

\begin{proof} \leavevmode \\
(i) By Remark~\ref{rem:hom-gen}, $\cB \le \cP(G)$ is generic and $d_v : \bB \to \cB$ is an isomorphism. Now we will check that every $V \in \bB$ satisfies $V \subseteq d_v(V) \subseteq \cl(V)$. Fix $V \in \bB$. For any $h \in V$ we have that $h^{-1}V \in \bB$ is a neighbourhood of $e$, so $h^{-1} V \in v$ and so $h \in d_v(V)$. Thus $V \subseteq d_v(V)$. Repeating the argument for $V^{\bot}$ in place of $V$, we get $V^{\bot} \subseteq d_v(V^{\bot}) = G \setminus d_v(V)$, hence $V \subseteq d_v(V) \subseteq \cl(V)$.

It follows that $d_v(V)$ is tidy and $d_v(V) \symdif V \in \cNWD$, hence $d_v(V) \in \cSBP$ and $\varrho(d_v(V)) = V$. Consequently, $\cB \le \cSBP$ and $\varrho \circ d_v = \id_{\bB}$.

\vspace{2mm} \noindent
(ii) The outline of the proof is as in Remark~\ref{rem:gen-reg}. We adjust it by showing that the extension $v = q \in S(\bRO)$ of $u \in S(\bB)$ can be found in $\fU$. It suffices to show that given $U \in u$ and $V \in \bRO$ such that $e \in V$, we have that $U \wedge V \neq \varnothing$. Take $B \in \cB$ such that $U = \varrho(B)$. Since $B$ is tidy, there is $U' \in \bRO$ such that $U' \subseteq B \subseteq \cl(U')$. It follows that $U = U'$. Furthermore, $e \in B$ because $\varrho(B) = U \in u$ and so $B \in \varrho^*(u) = \hat{e}$. Thus $e \in \cl(U)$ and $e \in V$, which implies that $U \wedge V = U \cap V \neq \varnothing$.
\end{proof}

Assume $\cB \le \cSBP$ is generic and tidy and take $v \in \fU$ and generic $\bB \le \bRO$ as in Proposition~\ref{prop:gen-reg}. The isomorphism $d_v : \bB \to \cB$ can be understood as follows. Any non-empty $V \in \bB$ is strongly generic as an element of $\bB$. The operation $d_v$ ``corrects'' $V$ on a nowhere dense set so that it becomes strongly generic in the usual sense. Namely, $v$ is an ultrafilter concentrated around identity. For each point $g$ we consider the left translate $gv$ as a template for a new set $B \subseteq G$ at $g$: if $V \in gv$, we include $g$ into $B$, and otherwise we do not. Then formally $B = d_v(V)$ and $V \subseteq B \subseteq \cl(V)$ as in the proof of Proposition~\ref{prop:gen-reg}. It follows that $B$ is a strongly generic set in $\cB$ that differs from $V$ on a nowhere dense set contained in $\bd(V)$. The above description justifies our view that $d_v$ is simple when $v \in \fU$.

Proposition~\ref{prop:gen-reg} is particularly strong when $G$ is a precompact topological group because in this case every generic algebra $\cB \le \cSBP$ is tidy, as the next proposition shows.
\begin{proposition} Assume $G$ is precompact and $A \in \cSBP$ is strongly generic. Then $A$ is tidy.
\end{proposition}

\begin{proof} Assume for contradiction that $A$ is not tidy. In particular, we have that $\varrho(A) \not \subseteq A$ or $A \not \subseteq \cl(\varrho(A))$. Because $\varrho(G \setminus A) \subseteq G \setminus A \iff A \subseteq \cl(\varrho(A))$, by replacing $A$ with $G \setminus A$ if necessary, we can assume that $A \not \subseteq \cl(\varrho(A))$. Take $a \in A \setminus \cl(\varrho(A))$ and pick open neighbourhoods $U, V \subseteq G$ of identity such that $\varrho(A) \cap aU = \varnothing$ and $V^{-1} V \subseteq U$. Then $A \cap aU$ is nowhere dense. Since $G$ is precompact, there are $t_1, \ldots, t_n \in G$ such that $G = t_1 V \cup \ldots \cup t_n V$.

Take $B \subseteq G$ that is a non-empty intersection of left translates of $A$ which minimizes the number of $i \in \{ 1, \ldots, n \}$ such that $B \cap t_i V \notin \cNWD$. We will prove that this number is zero. Indeed, otherwise we can find $i \in \{ 1, \ldots, n \}$ such that $B \cap t_i V \notin \cNWD$. Pick $b \in B \cap t_i V$ and let $B' = B \cap ba^{-1} A$. Then $B' \neq \varnothing$ because $b \in B'$. Clearly for each $j \in \{ 1, \ldots, n \}$ if $B \cap t_j V \in \cNWD$, then also $B' \cap t_j V \in \cNWD$. It remains to show that $B' \cap t_i V$ is nowhere dense, contradicting the choice of $B$. Using $b \in t_i V$, it is easy to check that $t_i V \subseteq bU$. It follows that
\[
B' \cap t_i V \subseteq ba^{-1} A \cap bU = ba^{-1} (A \cap aU) \in \cNWD,
\]
as desired.

Therefore $B \cap t_i V \in \cNWD$ for each $i \in \{ 1, \ldots, n \}$. It follows that $B$ is nowhere dense, hence not generic. This contradicts the assumption that $A$ is strongly generic.
\end{proof}

Now we address the first issue, that is, we show that $\cB := d_v[\bB]$ is always a maximal generic $G$-subalgebra of $\cA$ when $v \in \fU$. We also justify that tidy $G$-subalgebras $\cB \le \cA$ are enough to focus on even when $G$ is not precompact. As a consequence, we obtain an essentially complete correspondence between maximal generic $G$-subalgebras of $\cA$ and of $\bA$. 

\begin{notation*} Assume $\cB \le \cSBP$, $\bB \le \bRO$ and $q \in S(\bRO)$. We write $\cB \cong_q \bB$ to denote the property that $\varrho : \cB \to \bB$ and $d_q : \bB \to \cB$ are mutually inverse isomorphisms.
\end{notation*}

\begin{lemma} \label{lem:dRO-sub} Assume $\cA \le \cSBP$ is d-closed, $\bA = \varrho[\cA] \le \bRO$ and $q \in S(\bRO)$. Then $d_q[\bA] \subseteq \cA$.
\end{lemma}

\begin{proof} Fix $V \in \bA$ and pick $A \in \cA$ such that $V = \varrho(A)$. Using Fact \ref{fact:dual},
\[
d_q(V) = d_q(\varrho(A)) = d_{\varrho^*(q)}(A) \in \cA
\]
because $\varrho^*(q) \in S(\cSBP)$ and $\cA$ is d-closed.
\end{proof}

\begin{theorem} \label{thm:max-reg} Assume $\cA \le \cSBP$ is d-closed, $\bA = \varrho[\cA] \le \bRO$ and $v \in \fU$.
\begin{enumerate}[label=(\roman{*})]
\item If $\bB \le \bA$ is maximal generic, then $\cB \cong_v \bB$ for some tidy, maximal generic $\cB \le \cA$.
\item If $\cB \le \cA$ is tidy and maximal generic, then $\cB \cong_u \bB$ for some $u \in \fU$ and maximal generic $\bB \le \bA$.
\item If $\cB_0 \le \cA$ is maximal generic, then there are a maximal generic $\bB \le \bA$, a tidy, maximal generic $\cB \le \cA$ and an isomorphism $\varphi : \cB_0 \to \cB$ such that $\cB \cong_v \bB$ and $\varphi(B) \symdif B \in \cNWD$ for $B \in \cB_0$.
\end{enumerate}
\end{theorem}

\begin{proof} (i) By Lemma~\ref{lem:dRO-sub} and Proposition~\ref{prop:gen-reg} (i), $\cB := d_v[\bB] \le \cA$ is tidy and generic and satisfies $\cB \cong_v \bB$. To prove maximality, assume $\cB \le \cB_1 \le \cA$ is generic. By Remark~\ref{rem:hom-gen}, $\varrho \restriction \cB_1$ is injective and $\bB \le \varrho[\cB_1] \le \bA$ is generic. It follows that $\varrho[\cB_1] = \bB$ and so $\cB_1 = \cB$. Therefore $\cB \le \cA$ is maximal generic.

\vspace{2mm} \noindent
(ii) By Proposition~\ref{prop:gen-reg} (ii), we can find $u \in \fU$ such that $\bB := \varrho[\cB] \le \bA$ satisfies $\cB \cong_u \bB$. To prove maximality, assume $\bB \le \bB_1 \le \bA$ is generic. By Remark~\ref{rem:hom-gen} and Lemma~\ref{lem:dRO-sub}, $d_u \restriction \bB_1$ is injective and $\cB \le d_u[\bB_1] \le \cA$ is generic. It follows that $d_u[\bB_1] = \cB$ and so $\bB_1 = \bB$. Therefore $\bB \le \bA$ is maximal generic.

\vspace{2mm} \noindent
(iii) Let $\bB = \varrho[\cB_0] \le \bA$, $\cB = d_v[\bB] \le \cA$ and $\varphi = d_v \circ \varrho : \cB_0 \to \cB$. Then $\bB$ is generic, $\cB$ is tidy and generic, $\cB \cong_v \bB$ and $\varphi$ is an isomorphism. If $B \in \cB_0$, then $\varrho(\varphi(B)) = \varrho(d_v(\varrho(B)) = (\varrho \circ d_v)(\varrho(B)) = \varrho(B)$, hence $\varphi(B) \symdif B \in \cNWD$. Using Remark~\ref{rem:gen-reg} (ii) together with techniques as in (i) and (ii), it is easy to prove that $\bB \le \bA$ and $\cB \le \cA$ are maximal generic.
\end{proof}

Although Theorem \ref{thm:max-reg} is the main result of the subsection, we note that a correspondence between maximal generic $G$-subalgebras still exists when tidiness is dropped and arbitrary $q \in S(\bRO)$ are allowed.

\begin{lemma} \label{lem:U-den} For each $v \in \fU$ we have that $\beta G \ast v = \cl(G \cdot v) = S(\bRO)$.
\end{lemma}

\begin{proof} The first equality holds as always. For the second, take any non-empty $V \in \bRO$ and pick $g \in V$. Then $e \in g^{-1} V$, hence $g^{-1} V \in v$ and so $gv \in [V]$.
\end{proof}

\begin{lemma} \label{lem:im-maxgen} Assume $\cA \le \cP(G)$ is d-closed, $\cB \le \cA$ is maximal generic and $p \in \beta G$. Then $d_p[\cB] \le \cA$ is maximal generic.
\end{lemma}
\begin{proof} By Corollary~\ref{cor:im-maxgen}, $\cB$ is an image algebra, hence we can write $\cB = d_q[\cA]$ for some almost periodic $q \in S(\cA)$. Then $d_p[\cB] = d_p[d_q[\cA]] = d_{p \ast q}[\cA] \le \cA$ is an image algebra since $p \ast q \in S(\cA)$ is almost periodic.
\end{proof}

\begin{proposition} \label{prop:max-reg} Assume $\cA \le \cSBP$ is d-closed, $\bA = \varrho[\cA] \le \bRO$.
\begin{enumerate}[label=(\roman{*})]
\item If $\bB \le \bA$ is maximal generic and $q \in S(\bRO)$, then $\cB := d_q[\bB] \le \cA$ is maximal generic.
\item If $\cB \le \cA$ is maximal generic, then $\cB \cong_q \bB$ for some maximal generic $\bB \le \bA$ and $q \in S(\bRO)$.
\end{enumerate}
\end{proposition}

\begin{proof} (i) Pick $v \in \fU$. By Lemma~\ref{lem:U-den}, we can write $q = p \ast v$ for some $p \in \beta G$. By Theorem~\ref{thm:max-reg}, $d_v[\bB] \le \cA$ is maximal generic. The conclusion follows from Lemma~\ref{lem:im-maxgen}, since $d_q[\bB] = d_{p \ast v}[\bB] = d_p[ d_v[\bB] ]$.

\vspace{2mm} \noindent
(ii) Use Remark~\ref{rem:gen-reg} (ii) to find a generic $\bB \le \bA$ and some $q \in S(\bRO)$ such that $\cB \cong_q \bB$. For maximality of $\bB$, repeat the proof of Theorem~\ref{thm:max-reg} (iii).
\end{proof}

A notable application of Theorem~\ref{thm:max-reg} is in the o-minimal structures. Let $(M, \le, \ldots)$ be an o-minimal structure, where $\le$ is a dense linear order without endpoints, and $G$ a group definable in $M$. We consider $G$ as a topological group with the structure given by Proposition~\ref{prop:def-man}. Let $\cA := \Def_{\ext, G}(M) \le \cP(G)$. It is a d-closed $G$-algebra and by Corollary~\ref{cor:ext-sbp}, $\cA \le \cSBP$. Thus Theorem~\ref{thm:max-reg} applies in this setting.

Let $\bA := \varrho[\cA] \le \bRO$. It is worth noting that in this case $\bA \subseteq \cA$. We prove it using the following result of Shelah:

\begin{fact}[\cite{She09}] \label{fact:nip-ext} Assume $N$ is a model of an NIP theory. Let $N^{\ext}$ be the expansion of $N$ by all externally definable subsets of $N$. Then the theory of $N^{\ext}$ has NIP and quantifier elimination.
\end{fact}

Take $A \in \cA$. It is easy to see that $\varrho(A)$ is a subset of $G$ definable in $M^{\ext}$. Since $M$ is o-minimal, its theory has NIP, so by Fact~\ref{fact:nip-ext}, $M^{\ext}$ has quantifier elimination. It follows that $\varrho(A)$ is externally definable, i.e. $\varrho(A) \in \cA$.

We can summarize the results of this subsection using the example at hand. In order to compute the Ellis group of $S_{\ext, G}(M)$, we first find a maximal generic $G$-subalgebra $\bB$ of $\bA$. This may be easier than finding such an algebra in $\cA$ since firstly, regular open sets have less complexity than arbitrary sets with SBP, and secondly, nowhere dense sets can be neglected. Next we pick any $v \in \fU$ and compute $\cB := d_v[\bB]$, which is a maximal generic $G$-subalgebra of $\Def_{\ext, G}(M)$. The computation should be feasible because of the regular behaviour of $d_v$ which we described before. Finally, we use $\cB$ to compute $\cE$, which is isomorphic to the Ellis group of $S_{\ext, G}(M)$.

Theorem~\ref{thm:max-reg} guarantees that essentially every maximal generic subalgebra of $\cA$ can be obtained this way. Therefore the approach is unlikely to introduce unnecessary complexity, because if any simple maximal generic subalgebras of $\cA$ exist, some of them can be found via the method we propose.

We illustrate the procedure on the following basic example.

\begin{example} \label{ex:reg-s1} Let $M = (\RR, \le, +, \cdot, \ldots)$ be an o-minimal expansion of the reals. Consider the definable group
\[
G := S^1 = \{ (x, y) \in \RR^2 : x^2 + y^2 = 1 \} \subseteq \RR^2
\]
equipped with complex numbers multiplication and the Euclidean topology. Let $\cA = \Def_{\ext, G}(M)$ and $\bA = \varrho[\cA]$. For $a, b \in G$, we let $(a, b)$ denote the open arc of $S^1$ that goes from $a$ to $b$ counterclockwise. We similarly define $[a, b]$, $[a, b)$ etc. Each $A \in \cA$ is a finite union of arcs and points and each $V \in \bA$ is a finite union of open arcs with distinct endpoints. Thus $\bA$ is generic and it is a unique maximal generic $G$-subalgebra of itself.

The set $\fU$ is large, but since we are only interested in the algebras $d_v[\bA]$, where $v \in \fU$, we only need to describe the image $\pi[\fU]$ under the restriction $\pi : S(\bRO) \to S(\bA)$. We have that $\pi[\fU] = \{ \hat{e}^+, \hat{e}^- \}$, where $\hat{e}^+$ and $\hat{e}^-$ are ultrafilters of $\bA$ determined by
\begin{align*}
\hat{e}^+ & \supseteq \{ (e, b) : b \in S^1 \setminus \{ e \} \}, \\[1ex]
\hat{e}^- & \supseteq \{ (a, e) : a \in S^1 \setminus \{ e \} \}.
\end{align*}
For any distinct $a, b \in S^1$ we have that $d_{\hat{e}^+}\big( (a, b) \big) = [a, b)$ and $d_{\hat{e}^-}\big( (a, b) \big) = (a, b]$. Note how $(a, b)$ is only strongly generic in $G$ up to nowhere dense sets, but its modification $[a, b)$ (or $(a, b]$) is strongly generic in the usual sense. Thus $\cB_+ := d_{\hat{e}^+}[\bA]$ consists of finite unions of arcs of the form $[a, b)$ and $\cB_- := d_{\hat{e}^-}[\bA]$ consists of finite unions of arcs of the form $(a, b]$. Both are maximal generic $G$-subalgebras of $\Def_{\ext, G}(M)$.

Consider any $q \in S(\cB_+)$. Then $q$ is the restriction to $\cB_+$ of some left translate $g \varrho^*(\hat{e}^+)$ or $g \varrho^*(\hat{e}^-)$, where $g \in S^1$ and $\varrho^* : S(\bA) \to S(\cA)$ is the dual of $\varrho$. In the first case $d_q[\cB_+] = \cB_+$ and in the second $d_q[\cB_+] = \cB_-$. Thus
\[
\cE = \{ q \in S(\cB_+) : d_q[\cB_+] \subseteq \cB_+ \} = \{ g \varrho^*(\hat{e}^+) \cap \cB_+ : g \in S^1 \}.
\]
When $g \varrho^*(\hat{e}^+)$ is identified with $g$, it is easy to check that the operation $\ast$ on $\cE$ coincides with the group operation of $S^1$. Therefore the Ellis group of $S_{\ext, G}(M)$ is isomorphic to $S^1$.
\end{example}

We conclude the subsection by discussing a final, more elaborate example.

\subsubsection{Groups with definable compact-torsion-free decomposition}

In \cite{Jag15} Jagiella explored the topological dynamics of an interesting class of groups definable in o-minimal expansions of the reals. We briefly recall the set-up and then analyse the generic algebras of (externally) definable subsets and their regular open counterparts. In the process, we negatively answer the following question, motivated by Example~\ref{ex:reg-s1}: if $G$ is a group definable in an o-minimal expansion $M$ of the reals and $\cA = \Def_{\ext, G}(M)$, must $\bA := \varrho[\cA]$ contain a unique maximal generic subalgebra?

Assume $M = (\RR, \le, +, \cdot, \ldots)$ is an o-minimal expansion of the ordered field of reals and fix a group $G$ definable in $M$. We consider $G$ as a topological group with the structure from Proposition~\ref{prop:def-man}. Moreover, we assume that $G$ admits a \emph{definable compact-torsion-free decomposition}, meaning that there exist definable subgroups $K, H \le G$ such that $K$ is definably compact, $H$ is torsion-free, $G = K \cdot H$ and $K \cap H = \{ e \}$. Since we work in $\RR$, the group $K$ is actually compact. Note that $H$ and $K$ are closed in $G$ by \cite[Corollary 2.8]{Pil88}.

The algebra $\cA := \Def_{\ext, G}(M) = \Def_G(\RR)$ of $\RR$-definable subsets of $G$ is d-closed. Let $S_G(\RR)$, $S_K(\RR)$, $S_H(\RR)$ denote the Stone space of $\Def_{G}(\RR)$, $\Def_{K}(\RR)$, $\Def_{H}(\RR)$, respectively. $S_K(\RR)$ and $S_H(\RR)$ are subsemigroups of $S_G(\RR)$ with respect to the usual operation $\ast$. The group $H$ naturally acts on $G/H \approx K$ and we denote this action as $\varphi_h : K \to K$, where $h \in H$. Explicitly, $\varphi_h(k)$ is the unique $k' \in K$ such that $hk = k'h'$ for some $h' \in H$. The action extends to an action of the semigroup $S_H(\RR)$ on $S_K(\RR)$, which also will be denoted as $\varphi_p : S_K(\RR) \to S_K(\RR)$ for $p \in S_H(\RR)$. We let $I$ denote the set of generic points of $S_K(\RR)$, which is a unique minimal subflow of $S_K(\RR)$ and is closed under the action of $H$. There is an $H$-invariant type $p_{\infty} \in S_H(\RR)$ and for any such type $I \ast p_{\infty} \trianglelefteqslant_m S_G(\RR)$. From now $p_{\infty}$ will be fixed.

Since $K$ is compact, each $q \in S_K(\RR)$ has a standard part $\st(q) \in K$. The map $\st : S_K(\RR) \to K$ is a semigroup homomorphism and for each $u \in J(I)$, its restriction $\st : uI \to K$ is an isomorphism. Define $\psi : K \to K$ as $\psi = {\st} \circ \varphi_{p_{\infty}}$ and let $Z = \{ z \in K : \psi(z) = 1_K \}$. The set of idempotents of $I \ast p_{\infty}$ is precisely $J(I \ast p_{\infty}) = \{ u \ast \hat{z} \ast p_{\infty} : u \in J(I), z \in Z \}$. We have that $\psi[K] = N_G(H) \cap K$ and $\psi \circ \psi = \psi$, hence $\psi \restriction N_G(H) \cap K = \id$. 

We direct the reader to \cite{Jag15} for more details.

Jagiella describes the set $J(I \ast p_{\infty})$ in a technically different way:
\[
J(I \ast p_{\infty}) = \{ q \ast p_{\infty} : q \in I, \psi(\st(q)) = 1_K \}.
\]
The following lemma explains why our description is equivalent.

\begin{lemma} \label{lem:ctf-equiv} Let $u \in J(I)$ and $k \in K$. Then $u \ast \hat{k}$ is the unique type $q \in uI$ such that $\st(q) = k$. Consequently,
\[
\{ q \in I : \psi(\st(q)) = 1_K \} = \{ u \ast \hat{z} : u \in J(I), z \in Z \}.
\]
\end{lemma}

\begin{proof} The function $r_k : S_K(\RR) \to S_K(\RR)$ defined by $r_k(q) = q \ast \hat{k}$ is a $G$-flow endomorphism, so the unique minimal subflow $I \trianglelefteqslant_m S_K(\RR)$ is mapped by $r_k$ onto itself. In particular, $u \ast \hat{k} = r_k(u) \in I$. Furthermore, $u \ast \hat{k} \in uI$ since $u \ast \hat{k} = u \ast u \ast \hat{k}$. Finally, $\st(u \ast \hat{k}) = \st(u) \cdot \st(\hat{k}) = 1_K \cdot k  = k$. Uniqueness follows from the fact that $\st : uI \to K$ is an isomorphism.
\end{proof}

\begin{lemma} \label{lem:ctf-prod} The function $m : K \times H \to G$ defined as $m(k, h) = k \cdot h$ is a homeomorphism, where $K \times H$ is equipped with the product topology.
\end{lemma}

\begin{proof} Clearly $m$ is a continuous bijection. To check that $m^{-1}$ is continuous, it suffices to prove that for any net $(g_i)_{i \in I}$ in $G$ convergent to some $g \in G$ there is a subnet $(g_{i_j})_{j \in J}$ such that $m^{-1}(g_{i_j})$ converges to $m^{-1}(g)$. Take any net $(g_i)_{i \in I}$ in $G$ convergent to $g \in G$ and let $(k_i, h_i) = m^{-1}(g_i)$. Since $K$ is compact, by passing to a subnet, we can assume that $(k_i)_{i \in I}$ converges to some $k \in K$. Then $h_i = k_i^{-1} g_i \to k^{-1} g$ and $k^{-1} g \in H$, since $H$ is closed. Thus $m^{-1}(g_i) = (k_i, h_i) \to (k, k^{-1} g) = m^{-1}(g)$ in $K \times H$, as desired.
\end{proof}

\begin{corollary} The function $\varphi : H \times K \to K$ defined by $\varphi(h, k) = \varphi_h(k)$ is continuous.
\end{corollary}

\begin{proof} Assume $(h_i, k_i)_{i \in I}$ is a net in $H \times K$ convergent to some $(h, k) \in H \times K$. Take $h', h'_i \in H$ such that $h \cdot k = \varphi_h(k) \cdot h'$ and $h_i \cdot k_i = \varphi_{h_i}(k_i) \cdot h'_i$. Then 
\[
\varphi_{h_i}(k_i) \cdot h'_i = h_i \cdot k_i \to h \cdot k = \varphi_h(k) \cdot h',
\]
so by Lemma~\ref{lem:ctf-prod}, $\varphi_{h_i}(k_i) \to \varphi_h(k)$. Hence $\varphi$ is continuous.
\end{proof}

Take any idempotent $u \ast \hat{z} \ast p_{\infty} \in J(I \ast p_{\infty})$, where $u \in J(I), z \in Z$. We will describe the image algebra $\Im d_{u \ast \hat{z} \ast p_{\infty}}$. Since $d_{u \ast \hat{z} \ast p_{\infty}} = d_{u \ast \hat{z}} \circ d_{p_{\infty}}$ we first compute the image of $d_{p_{\infty}}$.

\begin{lemma} When $d_{p_{\infty}}$ is treated as a function $\Def_G(\RR) \to \Def_G(\RR)$,
\[
\Im d_{p_{\infty}} = \{ A \cdot H : A \subseteq K \text{ is definable} \}.
\]
\end{lemma}
\begin{proof} $(\subseteq)$ Take any definable $X \subseteq G$. By Remark~\ref{rem:dlim} (i), for any $h \in H$
\[
d_{p_{\infty}}(X) \cdot h^{-1} = d_{\hat{h}}( d_{p_{\infty}}(X) ) = d_{\hat{h} \ast p_{\infty}}(X) = d_{p_{\infty}}(X),
\]
hence $d_{p_{\infty}}(X) = A \cdot H$ for some $A \subseteq K$. Moreover, $A$ is definable, since $A = d_{p_{\infty}}(X) \cap K$.

\vspace{2mm} \noindent
$(\supseteq)$ Take any definable $A \subseteq K$. As in the first part, $d_{p_{\infty}}(A \cdot H) = B \cdot H$ for some definable $B \subseteq K$. It suffices to prove that $B = A$. For $k \in K$,
\[
k \in d_{p_{\infty}}(A \cdot H) \iff k^{-1}A \cdot H \in p_{\infty} \iff (k^{-1}A \cdot H) \cap H \in p_{\infty} \iff k \in A,
\]
where the last equivalence holds because $(k^{-1}A \cdot H) \cap H$ is equal to $H$ if $k \in A$ and to $\varnothing$ otherwise. Hence $d_{p_{\infty}}(A \cdot H) \cap K = A$ and so $B = (B \cdot H) \cap K = A$.
\end{proof}

We proceed to compute $\Im d_{u \ast \hat{z} \ast p_{\infty}}$. Take any definable $A \subseteq K$ and let $Y = d_{u \ast \hat{z}}(A \cdot H)$. Since $Y \subseteq G \approx K \times H$, we may define $H$-sections $Y^h$, $h \in H$, of $Y$:
\[
Y^h = \{ k \in K : k \cdot h \in Y \}.
\]
We will describe $Y$ in terms of its $H$-sections. Let $k \in K, h \in H$. We have that 
\begin{align*}
k \cdot h \in Y & \iff h^{-1} k^{-1} A H \in u \ast \hat{z} \iff \varphi_{h^{-1}}[k^{-1}A] \cdot H \in u \ast \hat{z} \\
& \iff \varphi_{h^{-1}}[k^{-1}A] \in u \ast \hat{z} \iff k^{-1} A \in \varphi_h(u \ast \hat{z}).
\end{align*}
Therefore $Y$ is given by $Y^h = d_{\varphi_h(u \ast \hat{z})}(A)$ and so the corresponding maximal generic subalgebra of $\Def_G(\RR)$ is 
\[
\cB_{u, z} := \Im d_{u \ast \hat{z} \ast p_{\infty}} = \{ Y \subseteq K \times H : (\exists A \underset{\text{def.}}{\subseteq} K)(\forall h \in H) \, Y^h = d_{\varphi_h(u \ast \hat{z})}(A) \}.
\]
Now we compute $\bB_z := \varrho[\cB_{u, z}]$, which turns out to only depend on $z$. Fix $Y \in \cB_{u, z}$ and take a definable $A \subseteq K$ such that $Y^h = d_{\varphi_h(u \ast \hat{z})}(A)$ for $h \in H$.

\begin{lemma} \label{lem:ctf-form} For each $h \in H$ there is $v_h \in J(I)$ such that $\varphi_h(u \ast \hat{z}) = v_h \ast \widehat{\varphi_h(z)}$.
\end{lemma}

\begin{proof} Fix $h \in H$. As mentioned in the set-up description, $I$ is closed under the action of $H$, i.e. $\varphi_h[I] \subseteq I$. Hence there is $v_h \in J(I)$ such that $\varphi_h(u \ast \hat{z}) \in v_hI$. By the continuity of $\varphi_h$, 
\[
\st( \varphi_h( u \ast \hat{z} ) ) = \varphi_h( \st( u \ast \hat{z} ) ) = \varphi_h(z).
\]
The conclusion follows from Lemma~\ref{lem:ctf-equiv}.
\end{proof}

Reasoning similarly as in Proposition~\ref{prop:gen-reg} (i), we have:
\begin{fact} \label{fact:ctf-id} For any $v \in J(I)$ and definable $B \subseteq K$ we have that 
\[
\int(B) \subseteq d_v B \subseteq \cl(B). \noproof
\]
\end{fact}

In order to disambiguate the notation, we let $\varrho_G : \cSBP(G) \to \bRO(G)$ and $\varrho_K : \cSBP(K) \to \bRO(K)$ denote the functions defined by the same formula $\varrho(X) = \int(\cl(X))$, computed in the respective topological groups.

\begin{lemma} $\varrho_G(Y)$ is given by $\varrho_G(Y)^h = \varrho_K(A) \varphi_h(z)^{-1}$ for $h \in H$.
\end{lemma}

\begin{proof} Given $h \in H$, pick $v_h \in J(I)$ as in Lemma~\ref{lem:ctf-form}. Then 
\[
Y^h = d_{\varphi_h(u \ast \hat{z})} A = d_{v_h} \circ d_{\widehat{\varphi_h(z)}}(A) = d_{v_h} \big( A \cdot \varphi_h(z)^{-1} \big)
\]
so by Fact~\ref{fact:ctf-id},
\[
\int_K(A) \varphi_h(z)^{-1} \subseteq Y^h \subseteq \cl_K(A) \varphi_h(z)^{-1}.
\]
We continue to identify $G$ with $K \times H$. The function $T : K \times H \to K \times H$ defined by $T(k, h) = (k \varphi_h(z), h)$ is a homeomorphism and $T[Y]^h = Y^h \cdot \varphi_h(z)$. Hence
\[
\int_K(A) \subseteq T[Y]^h \subseteq \cl_K(A)
\]
for each $h \in H$, so 
\[
\int_K(A) \times H \subseteq T[Y] \subseteq \cl_K(A) \times H.
\]
By Corollary~\ref{cor:ext-sbp}, $A$ has SBP, so the sets on the left and on the right differ by a nowhere dense set from each other, and also from $\varrho_K(A) \times H$, which is regular open. It follows that 
\[
T[\varrho_G(Y)] = \varrho_G(T[Y]) = \varrho_K(A) \times H
\]
and so for $h \in H$,
\[
\varrho_G(Y)^h = T[\varrho_G(Y)]^h \cdot \varphi_h(z)^{-1} = \varrho_K(A) \cdot \varphi_h(z)^{-1}. \qedhere
\]
\end{proof}

Thus by Theorem~\ref{thm:max-reg}, the maximal generic subalgebras of $\varrho_G[\Def_G(\RR)]$ are given by
\[
\bB_z = \{ V \subseteq K \times H : (\exists U \in \bA_K)(\forall h \in H) \, V^h = U \varphi_h(z)^{-1} \},
\]
where $z \in Z$ and $\bA_K = \varrho_K[\Def_K(\RR)] = \Def_K(\RR) \cap \bRO(K)$. 

Now we wish to find out when $\varrho_G[\Def_G(\RR)]$ contains a unique maximal generic subalgebra. For $z \in Z$, $h \in H$, define $\delta_z(h) = z \varphi_h(z)^{-1}$.

\begin{lemma} \label{lem:maxreg-eq} For every $z_1, z_2 \in Z$,
\[
\bB_{z_1} = \bB_{z_2} \iff \delta_{z_1} = \delta_{z_2}.
\]
\end{lemma}
\begin{proof} Note that for $z \in Z$,
\[
\bB_{z} = \{ V \subseteq K \times H : V^{1_H} \in \bA_K \ \& \ (\forall h \in H) \, V^h = V^{1_H} \delta_z(h) \}.
\]
The right-to-left implication follows. For the other one, assume $\delta_{z_1}(b) \neq \delta_{z_2}(b)$ for some $b \in H$. Then we can find $U \in \bA_K$ such that $U \delta_{z_1}(b) \neq U \delta_{z_2}(b)$. Define $V \subseteq K \times H$ so that $V^h = U \delta_{z_1}(h)$ for $h \in H$. Clearly $V \in \bB_{z_1}$, but $V \notin \bB_{z_2}$ since $V^b = U \delta_{z_1}(b) \neq U \delta_{z_2}(b) = V^{1_H} \delta_{z_2}(b)$.
\end{proof}

\begin{proposition} $\varrho_G[\Def_G(\RR)]$ has a unique maximal generic $G$-subalgebra if and only if $Z = \{ 1_K \}$.
\end{proposition}
\begin{proof} Clearly $1_K \in Z$. By Lemma~\ref{lem:maxreg-eq}, for any $z \in Z$ we have
\[
\bB_z = \bB_{1_K} \iff \delta_z = \delta_{1_K} \iff \delta_z \equiv 1_K \iff z \in N_G(H) \cap K.
\]
The maximal generic subalgebra corresponding to all such $z$ is
\[
\bB_{1_K} = \{ U \times H : U \in \bA_K \}.
\]
It follows that this algebra is unique precisely when $Z \subseteq N_G(H) \cap K$. This is equivalent to $Z = \{ 1_K \}$ because $\psi \restriction N_G(H) \cap K = \id$ and $\psi \restriction Z \equiv 1_K$.
\end{proof}

Finally, to see that $\varrho_G[\Def_G(\RR)]$ can have more than one maximal generic subalgebra, we check that $Z \neq \{ 1_K \}$ holds in a concrete group $G$, originally studied in \cite{GPP15}. Consider $G := SL_2(\RR)$, the group of all $2 \times 2$ matrices with determinant equal to $1$. It admits a definable compact-torsion-free decomposition $G = KH$, where $K = SO_2(\RR)$ is the group of all orthogonal matrices with determinant equal to $1$ and $H = T^+_2(\RR)$ is the group of all upper-triangular matrices with determinant equal to $1$ and positive elements on the diagonal. 

We define an $H$-invariant type $p_{\infty} \in S_H(\RR)$ as
\[
p_{\infty} = \operatorname{\text{tp}} \left( \begin{pmatrix} b_{\infty} & c_{\infty} \\ 0 & (b_{\infty})^{-1} \end{pmatrix} / \RR \right),
\]
where $b_{\infty} > \RR$ and $c_{\infty} > \operatorname{\text{dcl}}(\RR \cup \{ b_{\infty} \})$ in the monster model. For $\alpha \in \RR$, let $R_{\alpha}$ denote the rotation of $\RR^2$ through $\alpha$ about the origin, so that
\[
SO_2(\RR) = \{ R_{\alpha} : \alpha \in [0, 2\pi) \}.
\]
If $h = \begin{pmatrix} b & c \\ 0 & \frac{1}{b} \end{pmatrix}$ for some $b > 0$ and $c \in \RR$, then $\varphi_h(R_{\alpha}) = R_{\beta}$, where $\beta$ is the angle of the vector
\[
h \cdot \begin{pmatrix} \cos \alpha \\ \sin \alpha \end{pmatrix} = \begin{pmatrix} b \cos \alpha + c \sin \alpha \\ \frac{1}{b} \sin \alpha \end{pmatrix}.
\]
Consequently,
\[
\psi(R_{\alpha}) = \st(\varphi_{p_{\infty}}(R_{\alpha})) = \begin{cases} R_0 & \text{if } \alpha \in [0, \pi), \\ R_{\pi} & \text{if } \alpha \in [\pi, 2\pi), \end{cases}
\]
and $Z = \{ R_{\alpha} : \alpha \in [0, \pi) \} \neq \{ R_0 \}$.

\newpage
\appendix

\section{Explicit description of an Ellis group} \label{app:semi}

In this appendix we compute the Ellis group of the flow $S(\cA)$ introduced in Example~\ref{ex:usg-nper1}. For convenience we repeat the definitions here. We also present a variant of the example where the Ellis group appears more difficult to describe.

Assume $\cG \neq \{ e \}$ is a finite group and let $G = \cG^{\omega} \rtimes_{\varphi} \Sym(\omega)$, where the underlying action of the product is $\varphi_{\sigma}(s) = s \circ \sigma^{-1}$ for $\sigma \in \Sym(\omega)$, $s \in \cG^{\omega}$. For $n \in \omega$, $g \in \cG$ let
\[
A_n^g = \{ s \in \cG^{\omega} : s(n) = g \} \times \Sym(\omega).
\]
Define $\cA \le \cP(G)$ as the $G$-algebra generated by $A_0^e$. For any $\< s, \sigma \> \in G$ we have that $\< s, \sigma \> A_0^e = A_{\sigma(0)}^{s(\sigma(0))}$. As a result, 
\[
\{ \< s, \sigma \> A_0^e : \< s, \sigma \> \in G \} = \{ A_n^g : g \in \cG, n \in \omega \}\]
and
\[
\cA = \{ C \times \Sym(\omega) : C \subseteq \cG^{\omega} \text{ is clopen} \}.
\]
It follows that the space $S(\cA)$ is canonically homeomorphic to $\cG^{\omega}$. Via this homeomorphism $\cG^{\omega}$ becomes a $G$-flow with the action $\< s, \sigma \> \cdot x = s \varphi_{\sigma}(x)$ for $s, x \in \cG^{\omega}$ and $\sigma \in \Sym(\omega)$. It suffices to compute the Ellis group of $\cG^{\omega}$. For $g \in \cG$, let $\tilde{g} \in \cG^{\omega}$ denote the constant function everywhere equal to $g$. Recall that for $\< s, \sigma \> \in G$, the functions $\pi_{\< s, \sigma \>} : \cG^{\omega} \to \cG^{\omega}$ are defined as 
\[
\pi_{\<s, \sigma\>}(x) = \< s, \sigma \> \cdot x = s \varphi_{\sigma}(x)
\]
and $E(\cG^{\omega}) = \cl \{ \pi_{\<s, \sigma\>} : \<s, \sigma\> \in G \} \subseteq (\cG^{\omega})^{\cG^{\omega}}$.

\begin{proposition} \label{prop:semi-ell} The Ellis group of $\cG^{\omega}$ is isomorphic to $\cG$.
\end{proposition}

Before proving the proposition we need to make a few basic observations. Recall from topological dynamics that the points $p, q \in \cG^{\omega}$ are called \emph{proximal} if $f(p) = f(q)$ for some $f \in E(\cG^{\omega})$. 

\begin{lemma} \label{lem:prox} $p, q \in \cG^{\omega}$ are proximal if and only if $p(n) = q(n)$ for infinitely many $n < \omega$. 
\end{lemma}

\begin{proof} $({\implies})$ Assume that $f(p) = f(q) =: r$ for some $f \in E(\cG^{\omega})$ and fix $K \in \omega$. It suffices to show that $p$ and $q$ agree on at least $K+1$-many values. The set 
\[
U = \{ y \in \cG^{\omega} : y(k) = r(k) \text{ for } k \le K \}
\]
is an open neighbourhood of $r$, so we can find $\< s, \sigma \> \in G$ such that $\< s, \sigma\> p \in U$ and $\< s, \sigma \> q \in U$. In particular, $\varphi_{\sigma}(p)(k) = \varphi_{\sigma}(q)(k)$ for $k \le K$, which means that $p$ and $q$ agree at least on the $K+1$-many values $\sigma^{-1}(0), \ldots, \sigma^{-1}(K)$.

\vspace{2mm}
\noindent
$({\impliedby})$ Assume that $p(n) = q(n)$ for infinitely many $n \in \omega$. To show that there is $f \in E(\cG^{\omega})$ such that $f(p) = f(q) = \tilde{e}$, by compactness it suffices to show that for any open neighbourhood $U \subseteq \cG^{\omega}$ of $\tilde{e}$ we can find $\< s, \sigma \> \in G$ such that $\< s, \sigma \> p \in U$ and $\< s, \sigma \> q \in U$. Take any such $U$. Without loss of generality it is of the form
\[
U = \{ y \in \cG^{\omega} : y(k) = e \text{ for } k \le K \}
\]
for some $K \in \omega$. Take $\sigma \in \Sym(\omega)$ such that $\sigma^{-1}(0), \ldots, \sigma^{-1}(K)$ are the first $K+1$-many indices $n$ satisfying $p(n) = q(n)$ so that $\varphi_{\sigma}(p)$ and $\varphi_{\sigma}(q)$ agree on each $k \le K$. Let $s = \varphi_{\sigma}(p)^{-1}$. Then $\< s, \sigma \> p \in U$ and $\< s, \sigma \> q \in U$.
\end{proof}

We shall now investigate the structure of almost periodic points in $E(\cG^{\omega})$. For $p \in \cG^{\omega}, u \in \beta \omega \setminus \omega$, define $f_u^p : \cG^{\omega} \to \cG^{\omega}$ by $f_u^p(x) = p \cdot \widetilde{x(u)}$, where $x(u) = \lim_{n \to u} x(n)$. We aim to show that the almost periodic points in $E(\cG^{\omega})$ are precisely the functions of the form $f_u^p$, where $p \in \cG^{\omega}, u \in \beta \omega \setminus \omega$.

\begin{fact} \label{fact:im-nprox} Assume $X$ is a $G$-flow and $f \in E(X)$ is almost periodic. Then no two distinct points in $f[X]$ are proximal.
\end{fact}

\begin{proof} Assume $f(p)$ and $f(q)$ are proximal and take $f' \in E(X)$ such that $f'(f(p)) = f'(f(q))$. Since $f$ is almost periodic, we can find $f'' \in E(X)$ such that $f'' \circ f' \circ f = f$. It follows that $f(p) = f(q)$.
\end{proof}

\begin{remark} \label{rem:uni} Every $f \in E(\cG^{\omega})$ satisfies $f(x \tilde{g}) = f(x) \tilde{g}$ for $x \in \cG^{\omega}$, $g \in \cG$.
\end{remark}

\begin{proof} The set of all functions $f : \cG^{\omega} \to \cG^{\omega}$ satisfying the formula is closed in $(\cG^{\omega})^{\cG^{\omega}}$ and contains $\{ \pi_{\<s, \sigma\>} : \<s, \sigma\> \in G \}$, hence it contains $E(\cG^{\omega})$.
\end{proof}

\begin{lemma} \label{lem:aper-im} Assume $f \in E(\cG^{\omega})$ is almost periodic. Then there is $p \in \cG^{\omega}$ such that $f[\cG^{\omega}] = \{ p \tilde{g} : g \in \cG \}$.
\end{lemma}

\begin{proof} By Lemma~\ref{lem:prox}, among any $|\cG|+1$-many points in $\cG^{\omega}$ there are two that are proximal. Hence by Fact~\ref{fact:im-nprox}, the image of $f$ contains at most $|\cG|$-many points. On the other hand, it follows from Remark~\ref{rem:uni} that $f[\cG^{\omega}]$ contains at least $|\cG|$-many elements which are of the form $\{ p \tilde{g} : g \in \cG \}$ for any fixed $p \in f[\cG^{\omega}]$. Therefore  $f[\cG^{\omega}] = \{ p \tilde{g} : g \in \cG \}$.
\end{proof}

\begin{lemma} Assume $f \in E(\cG^{\omega})$ and take $p \in f[\cG^{\omega}]$. 
\begin{enumerate}[label=(\roman{*})]
\item For every $m \in \NN$ and $x_1, \ldots, x_m \in f^{-1}[\{ p \}]$ there are infinitely many $n \in \omega$ such that $x_i(n) = x_j(n)$ for any $1 \le i, j \le m$.
\item There is $g \in \cG$ such that for every $m \in \NN$ and $x_1, \ldots, x_m \in f^{-1}[\{ p \}]$ there are infinitely many $n \in \omega$ such that $x_i(n) = g$ for any $1 \le i \le m$.
\item We can choose $p \in f[\cG^{\omega}]$ such that (ii) holds with $g = e$.
\end{enumerate}
\end{lemma}

\begin{proof} (i) Take any $x_1, \ldots, x_m \in f^{-1}[\{ p \}]$ and fix $K \in \omega$. Since $f \in E(\cG^{\omega})$ and 
\[
U = \{ y \in \cG^{\omega} : y(k) = p(k) \text{ for } k \le K \}
\]
is an open neighbourhood of $p$, we can find $\< s, \sigma \> \in G$ such that $\< s, \sigma \> \cdot x_i \in U$ for $1 \le i \le m$. In particular, $x_1, \ldots, x_m$ agree on $\sigma^{-1}(k)$ for $k \le K$, so they must agree on infinitely many $n \in \omega$ since $K$ is arbitrary.

\vspace{2mm}
\noindent
(ii) Assume for contradiction that for each $g \in \cG$ we can find $m_g \in \NN$ and $x^{(g)}_1, \ldots, x^{(g)}_{m_g} \in f^{-1}[\{ p \}]$ such that for almost all $n \in \omega$ at least one of $x^{(g)}_i(n)$, where $1 \le i \le m_g$, is different from $g$. Since $\cG$ is finite, we can rewrite all $x^{(g)}_i$ as $x_1, \ldots, x_m \in f^{-1}[\{ p \}]$, where $m = \sum_{g \in \cG} m_g$. Then for almost all $n \in \omega$ the elements $x_1(n), \ldots, x_m(n)$ are not all equal, which contradicts (i).

\vspace{2mm}
\noindent
(iii) Take $g \in \cG$ such that the pair $(p, g)$ satisfies (ii). Then the pair $(p \tilde{h}, gh)$ also satisfies it for any $h \in \cG$. Hence $p' := p \tilde{g}^{-1}$ is as desired.
\end{proof}

For the next two facts fix an almost periodic point $f \in E(\cG^{\omega})$. Furthermore, pick $p \in f[\cG^{\omega}]$ as in item (iii) of the last lemma and let
\[
u = \{ x^{-1}[\{ e \}] : x \in f^{-1}[\{ p \}] \} \subseteq \cP(\omega).
\]

\begin{lemma} $u$ is a non-principal ultrafilter on $\omega$.
\end{lemma}
\begin{proof} By the choice of $p$, for every $I_1, \ldots, I_m \in u$ we have that
\[
\left| \bigcap_{i=1}^m I_i \right| = \infty.
\]
It remains to show that for any $I \subseteq \omega$ we have $I \in u$ or $\omega \setminus I \in u$. Fix $I \subseteq \omega$, choose some distinct $h_1, h_2 \in \cG$ and define $x \in \cG^{\omega}$ by
\[
x(n) = \begin{cases} h_1 & \text{if } n \in I, \\ h_2 & \text{if } n \notin I. \end{cases}
\]
By Lemma~\ref{lem:aper-im}, $f(x) = p \tilde{g}$ for some $g \in \cG$. Hence by Remark~\ref{rem:uni}, $f(x \tilde{g}^{-1}) = p$, so $(x \tilde{g}^{-1})^{-1}[\{ e \}]$ belongs to $u$ and equals either $I$ or $\omega \setminus I$.
\end{proof}

\begin{proposition} We have that $f = f_u^p$.
\end{proposition}

\begin{proof} Take any $x \in \cG^{\omega}$ and using Lemma~\ref{lem:aper-im}, write $f(x) = p \tilde{g}$ for some $g \in \cG$. Then $f(x \tilde{g}^{-1}) = p$, so by the definition of $u$ we have that $(x \tilde{g}^{-1})^{-1}[ \{ e \} ] \in u$. Equivalently, $x(u) = g$, hence $f(x) = p \cdot \widetilde{x(u)} = f_u^p(x)$.
\end{proof}

So far we have proved that every almost periodic point of $E(\cG^{\omega})$ is of the form $f_u^p$ for some $p \in \cG^{\omega}$, $u \in \beta \omega \setminus \omega$. Now we prove the converse.

\begin{proposition} For any $p \in \cG^{\omega}, u \in \beta \omega \setminus \omega$ we have that $f_u^p \in E(\cG^{\omega})$. Moreover, it is an almost periodic point of $E(\cG^{\omega})$.
\end{proposition}

\begin{proof} Fix an open basic neighbourhood $V$ of $f_u^p$ in $(\cG^{\omega})^{\cG^{\omega}}$. To show that $f_u^p \in E(\cG^{\omega})$, it suffices to find $\< s, \sigma \> \in G$ such that $\pi_{\<s, \sigma\>} \in V$. We may write
\[
V = \{ f \in (\cG^{\omega})^{\cG^{\omega}} : f(x_i) \in U_i \text{ for } i = 1, \ldots, m \},
\]
where $x_i \in \cG^{\omega}$ and $U_i$ is an open neighbourhood of $f_u^p(x_i)$ for $i = 1, \ldots, m$. Let $g_i = x_i(u)$ so that $f_u^p(x_i) = p \tilde{g_i}$. We may also assume that 
\[
U_i = \{ y \in \cG^{\omega} : y(k) = p(k) g_i \text{ for } k \le K \}
\]
for each $i = 1, \ldots, m$ and some fixed $K \in \omega$.  The set
\[
I = \bigcap_{i=1}^m x_i^{-1}[\{ g_i \}] \in u
\]
is infinite, so we can find $\sigma \in \Sym(\omega)$ such that $\sigma^{-1}(0), \ldots, \sigma^{-1}(K) \in I$. It is routine to check that $\pi_{\< p, \sigma \>} \in V$.

To prove that $f_u^p$ is almost periodic, fix $f \in E(\cG^{\omega})$. It is enough to find $f' \in E(\cG^{\omega})$ such that $f' \circ f \circ f_u^p = f_u^p$. Equivalently, $f' \circ f(y) = y$ for $y \in f_u^p[\cG^{\omega}] = \{ p \tilde{g} : g \in \cG \}$. By Remark~\ref{rem:uni}, this is further equivalent to $f'(f(p)) = p$. Hence $f' = \pi_{\<s, \id\>}$, where $s = p f(p)^{-1}$, is a good choice.
\end{proof}

We are now ready to prove the main proposition.

\begin{proof}[Proof of Proposition~\ref{prop:semi-ell}] Given $x, p \in \cG^{\omega}$, $u \in \beta \omega \setminus \omega$ and $f \in E(\cG^{\omega})$, we have that
\[
f \circ f_u^p(x) = f \big( p \cdot \widetilde{x(u)} \big) = f(p) \cdot \widetilde{x(u)},
\]
hence $f \circ f_u^p = f_u^{f} \! \! \; {}^{(p)}$. It follows that $f_u^p$ generates the ideal
\begin{align*}
I & = \{ f \circ f_u^p : f \in E(\cG^{\omega}) \} \\
& = \{ f_u^{f(p)} : f \in E(\cG^{\omega}) \} \\
& = \{ f_u^q : q \in \cG^{\omega} \},
\end{align*}
which only depends on $u$, so we denote it as $I_u$. In this ideal $f_u^p$ belongs to the Ellis group
\begin{align*}
\cE = f_u^p I_u & = \{ f_u^p \circ f_u^q : q \in \cG^{\omega} \} \\
& = \{ f_u^{p \cdot \widetilde{q(u)}} : q \in \cG^{\omega} \} \\
& = \{ f_u^{p \cdot \tilde{g}} : g \in \cG \}.
\end{align*}
We may choose $f_u^p$ as the idempotent of this group, which means that $p(u) = e$. It remains to show that the map $\psi : \cG \to \cE$, $\psi(g) = f_u^{p \tilde{g}}$ is an isomorphism. It is clearly bijective. For $g, h \in \cG$ we have that
\[
f_u^{p \tilde{g}} \circ f_u^{p \tilde{h}} = f_u^{p \cdot \tilde{g} \cdot \widetilde{(p \tilde{h})(u)}} = f_u^{p \cdot \tilde{g} \cdot \widetilde{p(u) \cdot h}} = f_u^{p \widetilde{gh}},
\]
hence $\psi(gh) = \psi(g) \circ \psi(h)$.
\end{proof} 

\subsection{A variant}

An interesting variant of the previous example is the following: assume $\cG \neq \{ e \}$ is a finite group and let $G = \cG^{\ZZ} \rtimes_{\varphi} \ZZ$. Here $\ZZ$ is identified with a subgroup $\{ \sigma_m : m \in \ZZ \} \le \Sym(\ZZ)$, where $\sigma_m(n) = n+m$ for $m, n \in \ZZ$, and the underlying action of the semidirect product is the same as before, that is, $\varphi_m(x) = x \circ \sigma_m^{-1}$ for $m \in \ZZ$, $x \in \cG^{\ZZ}$. For $n \in \ZZ$, $g \in \cG$ let
\[
A_n^g = \{ x \in \cG^{\ZZ} : x(n) = g \} \times \ZZ.
\]
The subset $A := A_0^e \subseteq G$ is uniformly strongly generic and not periodic. Let $\cA \le \cP(G)$ denote the $G$-algebra generated by $A$. By the same reasoning as in the previous example we have that the $G$-flow $S(\cA)$ is canonically isomorphic to $\cG^{\ZZ}$ with the action $\< s, m \> \cdot x = s \varphi_m(x)$ for $\< s, m \> \in G$ and $x \in \cG^{\ZZ}$.

We have not been able to compute the Ellis group of $\cG^{\ZZ}$. We will prove the following:

\begin{proposition} \label{prop:semi-Z-ell} The Ellis group of $\cG^{\ZZ}$ has cardinality at least $2^{\aleph_0}$.
\end{proposition}

\begin{lemma} Two points $p, q \in \cG^{\ZZ}$ are proximal if and only if they agree on arbitrary long intervals in $\ZZ$, i.e. for each $K \in \NN$ there is $n \in \ZZ$ such that $p(n+k) = q(n+k)$ for $0 \le k \le K$.
\end{lemma}

\begin{proof} Essentially the same as that of Lemma~\ref{lem:prox}.
\end{proof}

Let $\cC \subseteq \cG^{\ZZ}$ denote the set of all continuous functions $p : \ZZ \to \cG$, where $\ZZ$ is equipped with the profinite topology and $\cG$ is discrete. Clearly $\cC \le \cG^{\ZZ}$. Furthermore, it is not hard to see that $|\cC| = 2^{\aleph_0}$ (adjust the construction from Example~\ref{ex:tf}).

\begin{lemma} \label{lem:pfcont-nprox} Distinct points in $\cC$ are not proximal.
\end{lemma}

\begin{proof} Take any distinct $p, q \in \cC$. Then $p^{-1} q \neq \tilde{e}$, so there is a coset $K \subseteq \ZZ$ of a finite index subgroup of $\ZZ$ such that $p(k) \neq q(k)$ for $k \in K$. In particular, it is not the case that $p$ and $q$ agree on arbitrary long intervals in $\ZZ$, hence they are not proximal.
\end{proof}

\begin{proof}[Proof of Proposition~\ref{prop:semi-Z-ell}]
Take any minimal ideal $I \trianglelefteqslant_m E(\cG^{\ZZ})$ and $f \in J(I)$. By Lemma~\ref{lem:pfcont-nprox}, $f$ is injective on $\cC$, so $|f[\cG^{\ZZ}]| = 2^{\aleph_0}$. Hence it suffices to show that the Ellis group $fI$ acts transitively on $f[\cG^{\ZZ}]$.

Take any $p, q \in f[\cG^{\ZZ}]$ and let $x = qp^{-1} \in \cG^{\ZZ}$. Recall that $\pi_{\< x, 0 \>}$ is an element of $E(\cG^{\ZZ})$ defined by $\pi_{\< x, 0 \>}(y) = xy$.  Then $f \circ \pi_{\< x, 0 \>} \circ f \in fI$ and 
\[
(f \circ \pi_{\< x, 0 \>} \circ f)(p) = (f \circ \pi_{\< x, 0 \>})(p) = f(xp) = f(q) = q,
\]
where the first and fourth equality hold because $f \in J(I)$ and thus $f \restriction \Im f$ is the identity map.
\end{proof}

We have proved that when $f \in E(\cG^{\ZZ})$ is an almost periodic point, $f[\cG^{\ZZ}]$ has $2^{\aleph_0}$-many elements. This is very different from the previous example, where $f[\cG^{\ZZ}]$ would be finite of size $|\cG|$.

\begin{question} Can the Ellis group of $S(\cA)$ be computed explicitly?
\end{question}

\newpage

\section{USG families} \label{app:usg}

This appendix is devoted to some exploration of Question~\ref{q:usg}. In an arbitrary group $G$ we show a correspondence between non-periodic USG subsets of $ G$ and certain families of finite partial functions $G \to \{ 0, 1 \}$, which we call \emph{non-periodic USG families}. A family corresponding to $A \subseteq G$ can be thought of as consisting of finite ``chunks'' of $A$, which turn out to contain enough information to capture the property of $A$ being USG or periodic. The condition of non-periodicity of a USG family takes a particularly simple form in case $G = (\ZZ, +)$. It follows that in order to prove the existence of a non-periodic USG subset of $\ZZ$, it suffices to construct a non-periodic USG family there, which might be an easier task.

Assume $G$ is a group. We regard $2^G$ as a $G$-flow with the action given by
\[
(g \odot f)(x) = f(xg), \qquad \text{or equivalently,} \qquad g \odot \chi_A = \chi_{Ag^{-1}}
\]
for $g \in G$, $f \in 2^G$, $A \subseteq G$. When $X \subseteq G$ and $g \in G$, $f \in 2^X$, we also define $g \odot f \in 2^{Xg^{-1}}$ in the same way. Clearly $f_1 \subseteq f_2 \iff g \odot f_1 \subseteq g \odot f_2$ for $g \in G, f_1 \in 2^X, f_2 \in 2^Y$ and $X, Y \subseteq G$. 

\begin{notation*} \leavevmode
\begin{itemize}[itemsep=0pt] 
\item For any set $X$, let $[X]^{<\omega}$ denote the family of all finite subsets of $X$.
\item $2^{\le G}$ is the set of all functions $f : X \to \{ 0, 1 \}$, where $X \subseteq G$.
\item $2^{<G}$ is the set of all functions $\eta : U \to \{ 0, 1 \}$, where $U \in [G]^{<\omega}$.
\item For $\eta \in 2^{<G}$, $[\eta] = \{ f \in 2^{G} : \eta \subseteq f \}$.
\item For $U \in [G]^{<\omega}$, $f \in 2^{\le G}$, we write $U \preccurlyeq f$ when 
\[
(\exists g \in G) \, U \subseteq \dom (g \odot f).
\]
\item For $\eta \in 2^{<G}$, $f \in 2^{\le G}$ we say that \emph{$\eta$ occurs in $f$}, denoted $\eta \sqsubseteq f$, when 
\[
(\exists g \in G) \, \eta \subseteq g \odot f.
\]
\end{itemize}
\end{notation*}

\begin{definition}
Assume $\cA \subseteq 2^{<G}$. 
\begin{enumerate}[label=(\roman{*})]
\item We say that a property $P(\eta)$ is satisfied by \emph{cofinally many} elements $\eta \in \cA$, denoted $(\exists^{\infty} \eta \in \cA) \, P(\eta)$, provided that for every $U \in [G]^{<\omega}$ there is $\eta \in \cA$ such that $U \preccurlyeq \eta$ and $P(\eta)$ holds.
\item We say that a property $P(\eta)$ is satisfied by \emph{sufficiently large} $\eta \in \cA$, denoted $(\forall^{\infty} \eta \in \cA) \, P(\eta)$, when there is $U \in [G]^{<\omega}$ such that for each $\eta \in \cA$ with $U \preccurlyeq \eta$, $P(\eta)$ holds.
\end{enumerate}
We have the de Morgan-like duality
\[
\neg (\forall^{\infty} \eta \in \cA) \, P(\eta) \iff (\exists^{\infty} \eta \in \cA) \, \neg P(\eta).
\]
\end{definition}

\begin{remark} Assume $\cA \subseteq 2^{<G}$.
\begin{itemize}
\item If both $P(\eta)$ and $Q(\eta)$ hold for sufficiently large $\eta \in \cA$, then $P(\eta) \wedge Q(\eta)$ holds for sufficiently large $\eta \in \cA$.
\item If $P(\eta)$ holds for sufficiently large $\eta \in \cA$ and $Q(\eta)$ holds for cofinally many $\eta \in \cA$, then $P(\eta) \wedge Q(\eta)$ holds for cofinally many $\eta \in \cA$.
\end{itemize} 
\end{remark}

\begin{proof}Follows from the fact that if $U \cup V \preccurlyeq \eta$, then $U \preccurlyeq \eta$ and $V \preccurlyeq \eta$.
\end{proof}

\begin{definition} Assume $\cA \subseteq 2^{<G}$.
\begin{enumerate}[leftmargin=7mm, label=(\arabic{*})]
\item $\cA$ is \emph{cofinal} if $(\exists^{\infty} \eta \in 2^{<G}) \, \eta \in \cA$.
\item $\cA$ is a \emph{strongly generic family} when for each $\eta \in 2^{<G}$ occurring in cofinally many elements of $\cA$ there is $V \in [G]^{<\omega}$ such that all sufficiently large $\theta \in \cA$ satisfy
\[
(\forall g \in G) \big( V \subseteq \dom (g \odot \theta) \implies \eta \sqsubseteq (g \odot \theta) \restriction V \big).
\]
\item We call $\cA$ a \emph{uniformly strongly generic family}, abbreviated as \emph{USG family}, if it satisfies the strengthening of (2) such that $|V|$ only depends on $|\eta|$.
\item $t \in G$ is a \emph{period} of $\cA$ if for sufficiently large $\theta \in \cA$ neither of the two elements of $2^{<G}$
\[
\{ (e, 0), (t, 1) \}, \qquad \{ (e, 1), (t, 0) \}
\]
occurs in $\theta$.\footnote{We abuse the terminology and consider the condition satisfied when $t = e$, even though technically $\{ e \} \times \{ 0, 1 \} \notin 2^{<G}$. Thus we always have $e \in \Per(\cA)$.} The set of all periods of $\cA$ is denoted $\Per(\cA)$.
\item $\cA$ is \emph{periodic} if $\Per(\cA) $ is a generic subset of $G$.
\end{enumerate}
\end{definition}

It is worth noting that item (2) of the last definition is to some degree analogous to Definition~\ref{def:rep}, where $\eta$ corresponds to $f \restriction U$ and $\sqsubseteq$ replaces $(\exists h)$. The analogy further extends between item (3) and Definition~\ref{def:uni} (a). As we are about to show, a [uniformly] strongly generic family encodes a [uniformly] strongly generic subset of $G$ up to the minimal subflow it generates in $2^G$. Because the condition (2) is relaxed with the use of \emph{cofinally many} and \emph{sufficiently large} instead of \emph{some} and \emph{all}, the family is allowed to contain a certain amount of ``noise'' which does not affect the subset of $G$ it encodes. However, each pattern that exhibits itself across cofinally many elements of the family will be reflected in the encoded subset. 

Now we state the main result of the appendix.
\begin{theorem} \label{thm:usg-fam}
The following are equivalent:
\begin{enumerate}[label=(\arabic{*})]
\item There is a non-periodic USG subset $A \subseteq G$.
\item There is a non-periodic USG cofinal family $\cA \subseteq 2^{<G}$.
\end{enumerate}
\end{theorem}

Before proving the theorem we introduce some ideas.

\begin{proposition}
Assume $A \subseteq G$ is USG. Then every $B \subseteq G$ such that $\chi_B \in \cl(G \odot \chi_A)$ is USG. Moreover, for each Boolean term $\tau$ precisely the same numbers $m_{\tau}$ witness that $A$ and $B$ satisfy the condition (\ref{eq:usg}) from the definition of a USG set.
\end{proposition}

\begin{proof} Take $B \subseteq G$ with $\chi_B \in \cl(G \odot \chi_A)$. By Remark~\ref{rem:dlim} (iv), there is $q \in \beta G$ such that $B = d_q A$, so $B$ is USG by Remark~\ref{rem:hom-usg}. The rest follows from the analysis of the proof of Remark~\ref{rem:hom-usg} and the fact that $\chi_A$ is almost periodic by Theorem~\ref{thm:sgen-aper}, hence symmetrically $\chi_A \in \cl(G \odot \chi_B)$.
\end{proof}

It follows from the last proposition that whether $A \subseteq G$ is USG only depends on the subflow of $2^{G}$ generated by $\chi_A$. Therefore we may think of being USG as a property of subflows of $2^{G}$. Since $2^{G}$ is a Stone space, its closed subsets correspond to filters of the Boolean algebra $\cCO(2^{G})$ in the classical way. Below we define a similar correspondence which better suits our purposes.

\begin{definition} The \emph{content} of a subflow $X \subseteq 2^{G}$ is the family 
\[
c(X) = \{ \eta \in 2^{<G} : X \cap [\eta] \neq \varnothing \}.
\]
\end{definition}

The following notion allows to retrieve a subflow back from its content.

\begin{definition} Assume $\cA \subseteq 2^{<G}$. A point $f \in 2^{G}$ is said to be a \emph{limit point of $\cA$} when each finite $\eta \subseteq f$ occurs in cofinally many elements of $\cA$.
\end{definition} 

\begin{remark} \label{rem:lim-cont} If $X \subseteq 2^{G}$ is a subflow, then $X$ is the set of limit points of $c(X)$.
\end{remark}

\begin{proof} $(\subseteq)$ Assume $f \in X$, take any finite $\eta \subseteq f$ and fix $U \in [G]^{<\omega}$. Then $\theta := f \restriction (U \cup \dom \eta) \in c(X)$ satisfies $U \preccurlyeq \theta$ and $\eta$ occurs in $\theta$.

\vspace{2mm}
\noindent
$(\supseteq)$ Assume $f \notin X$. Then there is a finite $\eta \subseteq f$ such that $X \cap [\eta] = \varnothing$. Thus $\eta$ does not occur in any element of $c(X)$, so $f$ is not a limit point of $c(X)$.
\end{proof}

\begin{lemma} Every cofinal $\cA \subseteq 2^{<G}$ has a limit point. Moreover, the set of limit points of any $\cA \subseteq 2^{<G}$ is a subflow of $2^{G}$.
\end{lemma}

\begin{proof} To prove that $\cA$ has a limit point, assume for contradiction that for each $f \in 2^{G}$ there is a finite $\eta \subseteq f$ such that for sufficiently large $\theta \in \cA$, $\eta$ does not occur in $\theta$. By the compactness of $2^{G}$, we can find $\eta_1, \ldots, \eta_n \in 2^{<G}$ such that $2^{G} = [\eta_1] \cup \ldots \cup [\eta_n]$ and for sufficiently large $\theta \in \cA$ none of $\eta_i$ occurs in $\theta$. Since $\cA$ is cofinal, it follows that for cofinally many $\theta \in \cA$ we have that $\eta_1, \ldots, \eta_n \not \sqsubseteq \theta$. Thus we can find $\theta \in \cA$ with that property and $g \in G$ such that $\dom \eta_1 \cup \ldots \cup \dom \eta_n \subseteq \dom(g \odot \theta)$. Take $f \in 2^G$ extending $g \odot \theta$. Then $f \notin [\eta_1] \cup \ldots \cup [\eta_n]$, which is a contradiction. The rest is clear.
\end{proof}

Note that the correspondence from Remark~\ref{rem:lim-cont} does not work the other way, that is, an arbitrary family $\cA \subseteq 2^{<G}$ need not coincide with the content of the subflow consisting of its limit points. For instance, when $\cA \subseteq 2^{<G}$ is a content, it has the following property:
\[
(\forall \eta \in \cA)(\forall g \in G)(\exists \varepsilon \in \{ 0, 1 \}) \, \eta \cup \{ (g, \varepsilon) \} \in \cA,
\]
which an arbitrary family need not have. For this reason there is slightly more freedom in building USG families than in directly building a USG subset of $G$. We hope this little additional freedom can make it easier to construct such subsets in concrete groups like $(\ZZ, +)$.

\begin{remark} For $f \in 2^G$ and $\eta \in 2^{<G}$ we have $\eta \in c(\cl(G \odot f)) \iff \eta \sqsubseteq f$.
\end{remark}

The following two lemmas explain the connection between USG families and USG subsets of $G$.

\begin{lemma} \label{lem:usg-sub-fam} Assume $A \subseteq G$ is USG. Then $\cA := c(\cl(G \odot \chi_A))$ is a USG family.
\end{lemma}

\begin{proof} Fix $\eta \in 2^{<G}$ occurring in cofinally many elements of $\cA$. By the last remark, $\eta$ occurs in $\chi_A$. Since shifting $\chi_A$ changes neither the assumption nor the conclusion, we can assume that $\eta \subseteq \chi_A$.

Let $U = \dom \eta \in [G]^{<\omega}$. By Remark~\ref{rem:usg-char}, $\chi_A$ is uniformly self-replicating, so we can find $V \in [G]^{<\omega}$ such that
\begin{equation} \tag{$*$}
(\forall s \in G)(\exists t \in G) \, t \odot \eta \subseteq \chi_A \restriction Vs.
\end{equation}
Moreover, $|V|$ only depends on $|U| = |\eta|$. 

Take any $\theta \in \cA$ and $g \in G$ such that $V \subseteq \dom(g \odot \theta)$. Since $\theta \in \cA$, there is $h \in G$ such that $h \odot \theta \subseteq \chi_A$. Letting $s = gh^{-1}$ in ($*$), we get that $\eta \sqsubseteq \chi_A \restriction Vs$. Since 
\[
s \odot (\chi_A \restriction Vs) = (gh^{-1} \odot \chi_A) \restriction V = (g \odot \theta) \restriction V,
\]
it follows that $\eta \sqsubseteq (g \odot \theta) \restriction V$.
\end{proof}

\begin{lemma} \label{lem:usg-fam-sub} Assume $\cA \subseteq 2^{<G}$ is a USG family, $A \subseteq G$ and $\chi_A$ is a limit point of $\cA$. Then $A$ is USG.
\end{lemma}

\begin{proof} By Remark~\ref{rem:usg-char}, it suffices to show that $\chi_A$ is uniformly self-replicating. Take any $U \in [G]^{<\omega}$ and let $\eta = \chi_A \restriction U \in 2^{<G}$. Then $\eta$ occurs in cofinally many elements of $\cA$, so there is $V \in [G]^{<\omega}$ such that all sufficiently large $\theta \in \cA$ satisfy
\begin{equation} \tag{$*$}
(\forall g \in G) \big( V \subseteq \dom (g \odot \theta) \implies \eta \sqsubseteq (g \odot \theta) \restriction V \big).
\end{equation}
Moreover, $|V|$ only depends on $|\eta| = |U|$. Fix $s \in G$. It remains to show that $\eta$ occurs in $\chi_A \restriction Vs$. Since $\chi_A \restriction Vs$ occurs in cofinally many elements of $\cA$, we can find $\theta \in \cA$ satisfying ($*$) such that $\chi_A \restriction Vs \sqsubseteq \theta$. Take $h \in G$ such that $\chi_A \restriction Vs \subseteq h \odot \theta$. In particular, $V \subseteq \dom(sh \odot \theta)$, hence $\eta$ occurs in $(sh \odot \theta) \restriction V$. It then occurs in $(h \odot \theta) \restriction Vs = \chi_A \restriction Vs$, as desired.
\end{proof}

Now we establish the relation between periods of strongly generic families and their limit points.

\begin{lemma} \label{lem:per-fam-sub} Assume $\cA \subseteq 2^{<G}$ is a strongly generic family and $f \in 2^G$ is a limit point of $\cA$. Define $f^* \in 2^G$ by $f^*(x) = f(x^{-1})$. Then $\Per(\cA) = \Per(f^*)$.
\end{lemma}

\begin{proof} $(\subseteq)$ Fix $t \in \Per(\cA)$ so that for sufficiently large $\theta \in \cA$ neither of $\{ (e, 0), (t, 1) \}, \{ (e, 1), (t, 0) \}$ occurs in $\theta$. Then neither of these occurs in $f$, as $f$ is a limit point of $\cA$. It follows that $f(tx) = f(x)$ for each $x \in G$, hence $t^{-1} \in \Per(f^*)$ and thus $t \in \Per(f^*)$.

\vspace{2mm} \noindent
$(\supseteq)$ Fix $t \notin \Per(\cA)$ so that one of the patterns $\{ (e, 0), (t, 1) \}, \{ (e, 1), (t, 0) \}$ occurs in cofinally many $\theta \in \cA$. Let $\eta$ denote that pattern. Since $\cA$ is strongly generic, we can find $V \in [G]^{<\omega}$ such that for all sufficiently large $\theta \in \cA$
\begin{equation} \tag{$*$}
(\forall g \in G) \big( V \subseteq \dom (g \odot \theta) \implies \eta \sqsubseteq (g \odot \theta) \restriction V \big).
\end{equation}
The pattern $f \restriction V$ occurs in cofinally many elements of $\cA$, so we can find $\theta \in \cA$ satisfying ($*$) such that $f \restriction V$ occurs in $\theta$. Take $g \in G$ such that $f \restriction V \subseteq g \odot \theta$. By ($*$), $\eta$ occurs in $(g \odot \theta) \restriction V = f \restriction V$, which means that $f(x) \neq f(tx)$ for some $x \in G$. Hence $t^{-1} \notin \Per(f^*)$ and so $t \notin \Per(f^*)$.
\end{proof}

\begin{corollary} Assume $\cA \subseteq 2^{<G}$ is strongly generic. Then $\Per(\cA) \le G$.  In particular, $\cA$ is periodic if and only if $\Per(\cA)$ has finite index.
\end{corollary}

\begin{proof} If $\cA$ is not cofinal, then $\Per(\cA) = G$. Otherwise $\cA$ has a limit point $f \in 2^G$. By Lemma~\ref{lem:per-fam-sub}, $\Per(\cA) = \Per(f^*) \le G$.
\end{proof}

\begin{corollary} Consider the group $G = (\ZZ, +)$ and assume $\cA \subseteq 2^{<\ZZ}$ is strongly generic. Then $\cA$ is periodic if and only if some $n \in \NN \setminus \{ 0 \}$ is a period of $\cA$.
\end{corollary}

We are now ready to prove the main theorem.

\begin{proof}[Proof of Theorem~\ref{thm:usg-fam}] 
(1)$\implies$(2) Take a USG non-periodic subset $A \subseteq G$ and let $\cA := c(\cl(G \odot \chi_A))$. Clearly $\cA$ is cofinal and by Lemma~\ref{lem:usg-sub-fam}, it is a USG family. Furthermore, $\chi_A$ is a limit point of $\cA$, so by Lemma~\ref{lem:per-fam-sub}, $\Per(\cA) = \Per(\chi_A^*) = \Per(\chi_{A^{-1}})$. Since $A$ is non-periodic, by Remark~\ref{rem:per}, $A^{-1}$ is also non-periodic. Therefore $\Per(\cA)$ has infinite index and so $\cA$ is non-periodic.

\vspace{2mm}
\noindent
(2)$\implies$(1) Assume $\cA \subseteq 2^{<G}$ is a non-periodic USG cofinal family and take $A \subseteq G$ such that $\chi_A$ is a limit point of $\cA$. By Lemma~\ref{lem:usg-fam-sub}, $A$ is USG. Again $\Per(\cA) = \Per(\chi_{A^{-1}})$, hence $A^{-1}$ is non-periodic and so $A$ is non-periodic.
\end{proof}

Thus the following is an equivalent formulation of Question~\ref{q:usg}:

\begin{question} Is there a non-periodic USG cofinal family $\cA \subseteq 2^{<\ZZ}$?
\end{question}

\newpage
\bibliography{bibliography.bib}

\end{document}